\documentclass[preprint]{imsart}

\usepackage{graphicx}
\usepackage{amsmath,epsfig,amssymb,mathrsfs}
\usepackage{algorithmic}
\usepackage{cite}
\usepackage{amsthm}
\usepackage{arydshln}
\usepackage{amsopn}
\usepackage{color}
\usepackage{verbatim}
\usepackage{amsopn}
\usepackage{amssymb}
\usepackage{algorithm}
\usepackage[margin=1in]{geometry}
\usepackage{enumerate}

\usepackage[numbers,sort&compress]{natbib}

\newtheorem{theorem}{Theorem}[section]
\newtheorem{lemma}[theorem]{Lemma}
\newtheorem{definition}[theorem]{Definition}
\newtheorem{proposition}[theorem]{Proposition}

\newtheorem{remark}[theorem]{Remark}

\newcommand{\bE}{\mathbb{E}}

\newcommand{\bI}{\mathbb{I}}

\newcommand{\bP}{\mathbb{P}}
\newcommand{\hsig}{\hat{\sigma}}

\newcommand{\ts}{\textsuperscript}
\begin{document}

\begin{frontmatter}

\title{Consistent Parameter Estimation for LASSO and Approximate Message Passing}
\runtitle{Consistent Parameter Estimation for LASSO}

\begin{aug}
  \author{\fnms{Ali}  \snm{Mousavi}\corref{}\ead[label=e1]{ali.mousavi@rice.edu}},
  \author{\fnms{Arian} \snm{Maleki}\ead[label=e2]{arian@stat.columbia.edu}}
  \and
  \author{\fnms{Richard G.}  \snm{Baraniuk}
  \ead[label=e3]{richb@rice.edu}}


  \runauthor{A. Mousavi et al.}

  \affiliation{Rice University and Columbia University}

  \address{Addresses of the First, Second and Third authors,\\ 
          \printead{e1,e2,e3}}

\end{aug}

\begin{abstract}
%
We consider the problem of recovering a vector $\beta_o \in \mathbb{R}^p$ from $n$ random and noisy linear observations $y= X\beta_o + w$, where $X$ is the measurement matrix and $w$ is noise. 
The LASSO estimate is given by the solution to the optimization problem $\hat{\beta}_{\lambda} = \arg \min_{\beta} \frac{1}{2} \|y-X\beta\|_2^2 + \lambda \| \beta \|_1$. 
Among the iterative algorithms that have been proposed for solving this optimization problem, approximate message passing (AMP) has attracted attention for its fast convergence. The iterations of AMP are given by
\begin{eqnarray*}
\beta^{t+1} &=& \eta(\beta^t + X^* z^t; \tau^t), \nonumber \\
z^t &=& y- X\beta^t + \frac{|I^t|}{n} z^{t-1},
\end{eqnarray*}
where $\beta^t$, $z^t$, and $I^t$ denote estimates of $\beta_o$, $y-X \beta_o$, and the active set of $\beta_o$ at iteration $t$ respectively.  $\eta(x_i; \tau) = (|x_i|- \tau)_+{\rm sign}(x_i)$ denotes the soft thresholding function with threshold parameter $\tau$. 

Despite significant progress in the theoretical analysis of the estimates of LASSO and AMP, little is known about their behavior as a function of the regularization parameter $\lambda$, or the thereshold parameters $\tau^t$. For instance the following basic questions have not yet been studied in the literature:  (i) How does the size of the active set $\|\hat{\beta}^\lambda\|_0/p$ behave as a function of $\lambda$? (ii) How does the mean square error $\|\hat{\beta}_{\lambda} - \beta_o\|_2^2/p$ behave as a function of $\lambda$? (iii) How does $\|\beta^t - \beta_o \|_2^2/p$ behave as a function of $\tau^1, \ldots, \tau^{t-1}$? Answering these questions will help in addressing practical challenges regarding the optimal tuning of $\lambda$ or $\tau^1, \tau^2, \ldots$.

This paper answers these questions in the asymptotic setting (i.e., where $p \rightarrow \infty$, $n \rightarrow \infty$ while the ratio $n/p$ converges to a fixed number in $(0,1)$) and shows how these results can be employed in deriving simple and theoretically optimal approaches for tuning the parameters $\tau^1, \ldots, \tau^t$ for AMP or $\lambda$ for LASSO. It also explores the connection between the optimal tuning of the parameters of AMP and the optimal tuning of LASSO.  
\end{abstract}

\begin{keyword}[class=MSC]
\kwd{62G05}
\kwd{62J05}
\end{keyword}

\begin{keyword}
\kwd{Sparsity}
\kwd{LASSO}
\kwd{Approximate Message Passing}
\end{keyword}

\end{frontmatter}

\newpage

\section{Introduction}\label{sec:intro}

\subsection{Motivation}\label{ssec:motive}
Consider the problem of recovering a vector $\beta_o \in \mathbb{R}^p$ from a set of undersampled random linear measurements $y= X \beta_o+w$, where $X \in \mathbb{R}^{n \times p}$ is the measurement matrix and $w \in \mathbb{R}^n$ denotes noise. One of the most successful recovery algorithms, called the LASSO or basis pursuit denoising \citep{Tiblasso96, ChDoSa98}, employs the following optimization problem to obtain an estimate of $\beta_o$:

\begin{equation}\label{eq:LASSO}
\hat{\beta}^\lambda = \arg \min_\beta \frac{1}{2}\|y- X\beta\|_2^2 + \lambda \|\beta\|_1.   
\end{equation}

A rich literature has provided a detailed analysis of this algorithm \citep{DET, Tropp, ZhYu06, MeYu09, BiRiTs08, MeBu06, GeBu09, BuTsWe07, KnFu2000, ZoHaTib2007, DoTa05, Do05, DoTa08, DoTa09, DoTa09b, MalekiThesis, BaMo10, BaMo11,amelunxen2013living,oymak2012relation, tibshirani2013lasso, lockhart2012significance,wasserman2009high, zhang2009some, cai2009recovery, cai2010shifting,cai2010stable,raskutti2011minimax,zhang2008sparsity,zhang2006model}.  Most of the results fall into two categories: (i) non-asymptotic and (ii) asymptotic results. The non-asymptotic results consider $p$ and $n$ to be large but finite numbers and characterize the reconstruction error as a function of $p$ and $n$. These analyses provide qualitative guidelines on how to design compressive sensing (CS) and machine learning systems. However, they suffer from loose constants and are incapable of providing quantitative guidelines. Therefore, inspired by the seminal work of Donoho and Tanner \citep{DoTa05}, researchers have started performing asymptotic analyses of LASSO. Such analyses not only provide sharp quantitative guidelines, but have also led to  new recovery algorithms such as {\em Approximate Message Passing} (AMP) \citep{DoMaMo09}.

AMP was first proposed as a fast converging iterative algorithm for solving LASSO. Unlike traditional techniques for solving LASSO such as interior point method that are computationally demanding, AMP employs the following inexpensive iterations:
\vspace{-.2cm}
  \begin{eqnarray}\label{eq:ampeq1}
\beta^{t+1} &=& \eta(\beta^t + X^* z^t; \tau^t), \nonumber \\
z^t &=& y- X\beta^t + \frac{|I^t|}{n} z^{t-1}.
\end{eqnarray}
Here $t$ is the index of iteration, $\beta^t$ is the estimate of $\beta_o$ at iteration $t$, and $I^t \triangleq \{i \ : \  \beta_i^t \neq 0  \}$. $\eta$ is the soft thresholding function applied component-wise to the elements of the vector; for $a \in \mathbb{R}$, $\eta (a ; \tau) \triangleq (|a| - \tau)_{+} {\rm sign}(a)$. $\tau^t$ is called the threshold parameter. 

The performance of AMP and LASSO relies on the tuning of the free parameters $\lambda$ or $\tau^1, \tau^2, \ldots$. One important analysis that may help in searching for the optimal value of $\lambda$ and/or $\tau^t$ is the behavior of the solution $\hat{\beta}^{\lambda} $ (or $\beta^t$) as a function of $\lambda$ (or $\tau^t$). In this paper, we conduct such an analysis and demonstrate how such results can be employed for designing efficient algorithms for tuning the parameters of AMP or LASSO.

\subsection{Related work on parameter tuning}
The properties of the solution path of LASSO and AMP as their parameters change have not been studied before. However, both the tuning of the regularization parameter of LASSO and the threshold parameters of AMP have been studied in the literature. The proposed methods fall into the following three different categories: 
\begin{itemize}
\item [(i)]  General model selection ideas such as cross validation are probably the most popular approach in practice. For a review of these schemes, see Chapter 7 of \cite{hastie2005elements}. While very useful in applications, these ideas have their own limitations. We summarize some of their limitations in the context of our paper: (i) AMP has many free parameters (the threshold parameters at every iteration) and if we blindly apply these techniques their estimate of the risk will suffer from high variance and will lead to poor estimates of the threshold parameters. While we can resolve this issue by considering specific types of thresholds, they may degrade the performance of AMP. (ii) There are very few papers that have studied the accuracy of these generic model selection techniques in the high dimensional settings. For the case of LASSO there has been a few papers tackling this issue \cite{chatterjee2015prediction, homrighausen2014leave}. Similar to most of the analysis of the LASSO algorithm, these two papers have analyzed the performance of the cross validation (or methods inspired by cross validation) in the regime where $p$ and $n$ are both large but finite. Hence, their results also suffer from limitations similar to the ones we discussed in Section \ref{ssec:motive}.       
 In this paper we employ one of the standard model selection techniques, namely Stein Unbiased Risk Estimate (SURE), in our asymptotic framework and show how the properties of the solution path of AMP and LASSO enable us to not only obtain an efficient parameter tuning scheme, but also prove the consistency of these schemes under the asymptotic setting.  
 
\item[(ii)] The second approach employs the upper bounds derived in the literature on the risk of the estimators such as LASSO. For instance, \cite{candes2006stable,candes2008restricted, bickel2009simultaneous} suggest that the regularization parameter $\lambda$ should have the form of $c \sigma \sqrt{\log p}$ (when $k$ is much smaller than $p$), where $c$ is a fixed number that does not depend on the dimension of the problem. Such approaches have their limitations as well: (i) They are usually based on the minimax principle and hence might be considered as a pessimistic approach for tuning. (ii) The constants of these calculations are loose (even though the bounds are usually order-optimal), and hence the tuning that is based on such bounds do not lead to good performance in practice.  

\item[(iii)] The third approach, which is the closest to our paper, is based on asymptotic analysis of recovery algorithms. The first step in this approach is to employ asymptotic settings to obtain an accurate estimate of the reconstruction error of the recovery algorithms. This is done through either theoretical work or simulation. The next step is to employ this asymptotic analysis to obtain the optimal value of the parameters. This approach is employed in \cite{DoMaMo09, DoMaMoNSPT}. The main drawback of this approach is that the user must know the signal model (or at least an upper bound on the sparsity level of the signal) to obtain the optimal value of the parameters. Usually, an accurate signal model is not available in practice and hence the tuning should consider the least favorable signal, which leads to pessimistic tuning of the parameters. Our tuning approach is date-dependent and hence does not require any oracle information and it adapts itself to the statistics of the signal. Also, by employing the properties of the solution path we show how the tuning can be performed efficiently. 

\end{itemize}

\section{Asymptotic framework}\label{sec:asympframe}
In this section we review the asymptotic framework under which we analyze the solution path of LASSO and AMP. Furthermore, we review some of the existing results that will be used later in our analysis. 

\subsection{Notation}\label{sec:not}
 Capital letters denote both matrices and random variables. As we may consider a sequence of vectors of different sizes, sometimes we denote $\beta$ with $\beta(p)$ to emphasize its dependency on the ambient dimension. For a matrix $X$, $X^*$, $\sigma_{\rm min} (X)$, and $\sigma_{\max} (X)$ denote the transpose of $X$, the minimum, and the maximum singular values, respectively. Calligraphic letters such as $\mathcal{A}$ denote sets. For a vector $\beta \in \mathbb{R}^p$, $\beta_i$, $\| \beta\|_q \triangleq (\sum |\beta_i|^q)^{1/q}$, and $\|\beta\|_0 = |\{i \ : \  |\beta_i| \neq 0 \} |$ represent the $i^{\rm th}$ component, $\ell_q$, and $\ell_0$ norms respectively. The notation $\mathbb{E}_B $ denotes the expected value with respect to the randomness in the random variable $B$. The two functions $\phi$ and $\Phi$ denote the probability density function and cumulative distribution function of the standard normal distribution. We will also use notations $\overset{p}{\rightarrow}$ and $\overset{a.s.}{ \rightarrow}$ for the convergence in probability and almost sure respectively. Finally, $\mathbb{I}(\cdot)$ denotes the indicator function.

\subsection{LASSO in the asymptotic framework}\label{sec:asympframework}
In this paper we analyze the properties of the solution of LASSO and AMP when (i) the measurement matrix has iid Gaussian entries,\footnote{With the recent advances in high dimensional statistics \citep{BaMoLe12} our results can be easily extended to subgaussian matrices. However, for notational simplicity we focus on the Gaussian setting here. } and  (ii) the ambient dimension and the number of measurements are large. We let the measurement matrix to have iid Gaussian entries and adopt the {\em asymptotic framework} to incorporate these two features. Here is the formal definition of this framework \citep{DoMaMoNSPT, BaMo10}: Let $n,p \rightarrow \infty$ while $\delta = \frac{n}{p}$ is fixed. We write the vectors and matrices as $\beta_o(p), X(p), y(p)$, and $w(p)$ to emphasize on the ambient dimension of the problem. Clearly, the number of rows of the matrix $X$ is equal to $\delta p$, but we assume that $\delta$ is fixed, and therefore we do not include $n$ in our notation for $X$.  The same argument is applied to $y(p)$ and $w(p)$. Now we define a specific type of converging sequence as the following:

\begin{definition}\label{def:convseq}
A sequence of instances $\{\beta_o(p), X(p), w(p)\}$ is called a converging sequence if the following conditions hold:
\begin{itemize}
\item[-] The empirical distribution of $\beta_o(p) \in \mathbb{R}^p$ converges weakly to a probability measure $p_{\beta}$ with bounded second moment. Further, $\frac{1}{p} \|\beta_o(p)\|_2^2$ converges to the second moment of $p_{\beta}$.
\item[-] The empirical distribution of $w(p) \in \mathbb{R}^n$ ($n = \delta p$) converges weakly to a probability measure $p_W$ with bounded second moment. Furthermore, $\frac{1}{n} \|w(p)\|_2^2$ converges to the second moment of $N(0, \sigma_W^2)$. 
\item[-] If $\{e_i\}_{i=1}^p$ denotes the canonical basis for $\mathbb{R}^p$, then $\max_i \|X(p) e_i \|_2$ $, \min_i \|X(p) e_i \|_2 \rightarrow 1$ as $p \rightarrow \infty$. 
\end{itemize}
\end{definition}

Note that for the purposes of this section, $p_{\beta}$ is not necessarily a sparsity promoting prior. The last condition is equivalent to saying that the columns of $X$ have asymptotically unit $\ell_2$ norm. 

For each problem instance $\beta_o(p), X(p),$ and $w(p)$, we solve the LASSO and obtain $\hat{\beta}^{\lambda}(p)$ as the estimate. We now evaluate certain measures of performance for this estimate such as the MSE, etc. The next generalization formalizes the types of performance measures we are interested in. 

\begin{definition}\label{def:observablelasso}
Let $\hat{\beta}^{\lambda}(p)$ be the sequence of solutions of the LASSO problem for the converging sequence of instances $\{\beta_o(p), X(p), w(p)\}$. Consider a function $\psi: \mathbb{R}^2 \rightarrow \mathbb{R}$. An observable $J_{\psi}$ is defined as 
\[
J_{\psi} \left(\beta_o(p),\hat{\beta}^{\lambda}(p)\right) \triangleq \lim_{p \rightarrow \infty} \frac{1}{p} \sum_{i=1}^p \psi \left( \beta_{o,i}(p), \hat{\beta}^{\lambda}_i(p)\right).
\]
\end{definition}
Note that in the above definition we have assumed that the limit exists. A popular choice for $\psi$ is $\psi_M(u,v)= (u-v)^2$, which yields the MSE. 

 In the rest of the paper we assume that the elements of $X$ are iid $N(0,1/n)$ and $w_i \sim N(0, \sigma_w^2)$. Since $\hat{\beta}^{\lambda}(p)$ is random, there are two major questions that need to be addressed about $\lim_{p \rightarrow \infty}$ $ \frac{1}{p} \sum_{i=1}^N \psi \left( \beta_{o,i}(p), \hat{\beta}^{\lambda}_i(p)\right)$. (i) Does the limit exist and in what sense? (ii) Does the limit converge to a random variable or to a deterministic quantity?  The following theorem, conjectured in \citep{DoMaMoNSPT} and proved in \citep{BaMo11}, shows that under some restrictions on the $\psi$ function, not only does the almost sure limit exist in this scenario, but also it converges to a non-random number.

\begin{theorem}\label{thm:eqpseudolip} Consider a converging sequence $\{\beta_o(p), X(p), w(p)\}$, and let the elements of $X$ be drawn iid from $N(0,1/n)$. Suppose that $\hat{\beta}^{\lambda}(p)$ is the solution of the LASSO problem. 
Then for any pseudo-Lipschitz function\footnote{A function $\psi:\mathbb{R}^2 \rightarrow \mathbb{R}$ is pseudo-Lipschitz of order $k$ if there exists a constant $L>0$ such that for all $x,y \in \mathbb{R}^2$ we have $|\psi(x)-\psi(y)|\leq L(1+\|x\|_2+\|y\|_2)^k \|x-y\|_2$.} $\psi: \mathbb{R}^2 \rightarrow \mathbb{R}$, almost surely
\begin{equation}\label{eq:lassoobs}
\lim_{p \rightarrow \infty} \frac{1}{p} \sum_i \psi \left(\hat{\beta}^{\lambda}_i(p),{\beta}_{o,i}(p) \right) = \bE_{B,W} [\psi(\eta(B+\hsig W; \chi \hsig), B)],
\end{equation}
where $B$ and $W$ are two independent random variables with distributions $p_\beta$ and $N(0, 1)$, respectively, $\eta$ is the soft thresholding operator, and $\hsig$ and $\chi$ satisfy the following equations with $\sigma_{\omega}$ being the variance of the input noise:
\begin{eqnarray} \label{eq:fixedpoint}
\hsig^2 &=& \sigma_{\omega}^2+\frac{1}{\delta} \mathbb{E}_{B, W} [(\eta(B +\hsig W; \chi \hsig) -B)^2], \label{eq:fixedpoint11} \\
\lambda &=& \chi \hsig \left(1-\frac{1}{\delta} \mathbb{P}(|B +\hsig W| > \chi \hsig) \right). \label{eq:fixedpoint21}
\end{eqnarray}
\end{theorem}

Theorem \ref{thm:eqpseudolip} will provide the first step in our analysis of the LASSO's solution path. Before we proceed to the implications of this theorem, let us explain some of its interesting features. Suppose that $\hat{\beta}^{\lambda}$ has iid elements, and each element is in law equal to $\eta(B +\hsig W; \chi \hsig)$, where $B {\sim } p_\beta$ and $W {\sim} N(0, 1)$. Also, assume that $\beta_{o,i} \overset{iid}{\sim} p_\beta$. If these two assumptions were true, then we could use strong law of large numbers (SLLN) and argue that \eqref{eq:lassoobs} were true under some mild conditions (as required for the SLLN). While this heuristic is not quite correct, and the elements of $\hat{\beta}^{\lambda}_i$ are not necessarily independent, at the level of calculating observables defined in Definition \ref{def:observablelasso} (and $\psi$ being pseudo Lipschitz) this theorem confirms the heuristic. Note that the key elements that have led to this heuristic is the randomness in the measurement matrix and the large size of the problem.



%
%
%

\begin{remark}\label{thm:lassodetstate}
We are also interested in another observable and that is $\lim_{p \rightarrow \infty} \frac{\|\hat{\beta}_{\lambda} \|_0}{p}$. This observable can be constructed by using $\psi(u,v) = \mathbb{I} (v \neq 0)$. However, the $\psi$ function is not pseudo-Lipschitz, and hence Theorem \ref{thm:eqpseudolip} does not apply.  However, as conjectured in \citep{DoMaMoNSPT} and proved in \citep{BaMo11} we can still claim that 
\[
\lim_{p \rightarrow \infty} \frac{1}{p} \sum_i \mathbb{I} \left(\hat{\beta}^{\lambda}_i(p) \neq 0 \right)  = \mathbb{P} (| \eta(B+\hsig W; \chi \hsig)|>0),
\]
where  $\chi$, $\tau$, and $\hsig$ satisfy \eqref{eq:fixedpoint11} and \eqref{eq:fixedpoint21}.
\end{remark}

\subsection{AMP in the asymptotic setting}
In this section we review some background on the asymptotic analysis of AMP. This section is mainly based on the results in \citep{DoMaMo09, DoMaMoNSPT, BaMo10}, and the interested reader is referred to these papers for further details. AMP is an iterative thresholding algorithm. Therefore, we would like to know the discrepancy of its estimate at every iteration from the original vector $\beta_o$. The following definition formalizes different discrepancy measures for the AMP estimates.

\begin{definition} \label{def:observables}
Let $\{\beta_o(p), X(p), w(p)\}$ denote a converging sequence of instances. 
Let $\beta^{t}(p)$ be the estimate of AMP at iteration $t$. Consider a function $\psi: \mathbb{R}^2 \rightarrow \mathbb{R}$. An observable $J_{\psi}$ at time $t$ is defined as 
\[
J_{\psi} \left(\beta_o(p), \beta^{t}(p) \right) = \lim_{p \rightarrow \infty} \frac{1}{p} \sum_{i=1}^p \psi \left(\beta_{o,i}(p), {\beta}^{t}_{i}(p) \right).
\]
\end{definition}
As before, we can consider $\psi(u,v) = (u-v)^2$ that leads to the normalized MSE of AMP at iteration $t$.  The following result originally conjectured in \citep{DoMaMo09, DoMaMoNSPT} and finally proved in \citep{BaMo10}, provides a simple description of the almost sure limits of the observables. 

\begin{theorem}\label{thm:ampeqpseudo_lip} Consider the converging sequence $\{\beta_o(p), X(p), w(p)\}$, and let the elements of $X$ be drawn iid from $N(0,1/n)$. Suppose that ${\beta}^{t}(p)$ is the estimate of AMP at iteration $t$. 
Then for any pseudo-Lipschitz function $\psi: \mathbb{R}^2 \rightarrow \mathbb{R}$
\[
\lim_{p \rightarrow \infty} \frac{1}{p} \sum_i \psi \left({\beta}^{t}_{i}(p),{\beta}_{o,i}(p) \right) = E_{B,W} \left[\psi(\eta(B+ \sigma^t W; \tau^t), B)\right]
\] 
almost surely, where $B$ and $W$ are two independent random variables with distributions $p_\beta$ and $N(0,1)$, respectively. Furthermore, starting with $(\sigma^0)^2 = {\mathbb{E} \left[B^2\right]}/{\delta}$, $\sigma^t$ satisfies
\begin{eqnarray} \label{eq:ampevolution}
(\sigma^{t+1})^2 = \sigma_{\omega}^2+\frac{1}{\delta} \mathbb{E}_{B, W} \left[(\eta(B + \sigma^t W; \tau^t ) -B)^2\right].
\end{eqnarray}
\end{theorem}

Equation \eqref{eq:ampevolution} is known as the {\em state evolution (SE)} for AMP. Similar to our discussion of the solution of the LASSO, Theorem \ref{thm:ampeqpseudo_lip} claims that, as long as the calculation of the pseudo-Lipschitz observables is concerned, we can assume that estimates of AMP are modeled as iid elements with each element modeled in law as $\eta (B + \sigma^{t}W;\tau^t)$, where $B \sim p_\beta$ and $W \sim N(0,1)$.  As in Remark \ref{thm:lassodetstate}, we can establish $\lim_{p \rightarrow \infty} \frac{\| {\beta}^{t}(p)\|_0 }{p} = \mathbb{P} (|B+ \sigma^t W| \geq \tau^t )$ even though the $\psi(u,v) = I(v \neq 0)$ function is not pseudo-Lipschitz  \citep{BaMo10}.

One major feature of AMP that we will employ in this paper is that if we set $\tau^t$ ``appropriately,'' then the fixed point of AMP corresponds to the solution of LASSO in the asymptotic regime. One such choice of parameters is the fixed false alarm threshold given by $\tau^t = \chi \sigma^t$, where $\sigma^t$ satisfies \eqref{eq:ampevolution} and $\chi$ is a fixed number. The following result, conjectured in \citep{DoMaMoNSPT, DoMaMo09} and later proved in \citep{BaMo11} formalizes this statement. 

\begin{theorem}\citep{BaMo11}\label{thm:equLASSOAMP}
Consider the converging sequence $\{\beta_o(p), X(p), w(p)\}$ and let the elements of $X$ be drawn iid from $N(0,1/n)$. Let $\beta^t(p)$ be the estimate of the AMP algorithm with parameter $\tau^t =\chi \sigma^t$, where $\sigma^t$ satisfies \eqref{eq:ampevolution}. Assume that $\lim_{t \rightarrow \infty} (\sigma^t)^2=\hsig^2$. Finally, let $\hat{\beta}^{\lambda}$ denote the solution of the LASSO with parameter $\lambda$ that satisfies $\lambda = \chi \hsig (1- \mathbb{P} (|B+ \hsig W| \geq \chi \hsig ))$. Then, $\lim_{t \rightarrow \infty} \lim_{p \rightarrow \infty} \frac{1}{p} \|\hat{\beta}^\lambda(p) - {\beta}^t(p)\|_2^2 =0$ almost surely. 
\end{theorem}


\section{Main contributions}\label{sec:maincontribution}

\subsection{LASSO's solution path}\label{sec:lassopath}

We start by analyzing $\hat{\beta}^{\lambda}$ as $\lambda$ changes. The two main problems that we address are: (Q1)  How does $\frac{1}{p}\|\hat{\beta}^{\lambda}\|_0$ change as $\lambda$ varies? (Q2)  How does $\frac{1}{p}\|\hat{\beta}^\lambda  - \beta_o\|_2^2$ change as $\lambda$ varies?

The first question is about $\|\hat{\beta}^{\lambda}\|_0$, i.e., the number of active elements in the solution of the LASSO, and the second one is about the mean squared error (MSE). 
Intuitively speaking, one would expect the size of the active set to shrink as $\lambda$ increases and the mean squared error to be a bowl-shaped function of $\lambda$. Unfortunately the peculiar behavior of LASSO breaks this intuition. See Figure \ref{fig:activeset} for a counter-example. This figure exhibits the number of active elements in the solution as we increase the value of $\lambda$. It is clear that the size of the active set is not monotonically decreasing. The details of this simulation are described in Section \ref{sec:simulationdetails}.
\begin{figure}[h!]
\includegraphics[width= 6cm]{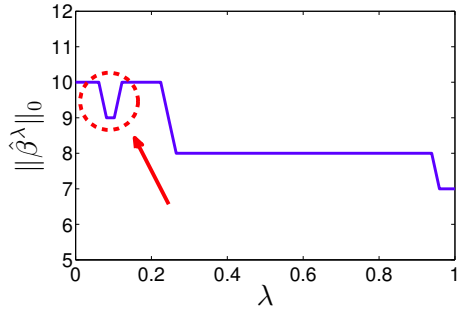}
\caption{The number of active elements in the LASSO's solution as a function of $\lambda$. The size of the active set at one location grows as we increase $\lambda$ and hence this function does not match the intuition. For the details of this experiment, see Section \ref{subsec:fig:activeset}. }
\label{fig:activeset}
\end{figure}

 Such pathological examples have discouraged further investigation of these problems in the literature. One of the main objectives of this paper is to show that such examples are quite rare, and if we consider the asymptotic setting (that will be described in Section \ref{sec:asympframework}), then we can provide quite intuitive answers to the two questions raised above. Let us summarize our results here in a non-rigorous way: (A1)  In the asymptotic setting, $\frac{1}{p}\|\hat{\beta}^{\lambda}\|_0$ is a decreasing function of $\lambda$. (A2)  In the asymptotic setting, $\frac{1}{p}\|\hat{\beta}^\lambda  - \beta_o\|_2^2$ is a quasi-convex function of $\lambda$. These results are formally stated below.

\begin{theorem}\label{lem:activeset}
Let  $\{\beta_o(p), X(p), w(p)\}$  denote a converging sequence of problem instances as defined in \ref{def:convseq}. Suppose that $X_{ij}\overset{iid}{\sim} N(0, 1/n)$. If $\hat{\beta}^{\lambda}(p)$ is the solution of LASSO with regularization parameter $\lambda$, then
\[
\frac{d}{d \lambda} \left(\lim_{p \rightarrow \infty} \frac{1}{p} \sum_i \mathbb{I} \left(\hat{\beta}^{\lambda}_i(p) \neq 0 \right) \right) <0. 
\]
Furthermore, $\lim_{p \rightarrow \infty} \frac{1}{p} \sum_i \mathbb{I} \left(\beta^{\lambda}_i(p) \neq 0 \right) \leq \delta$ no matter how we select $\lambda>0$.
\end{theorem}
\noindent We summarize the proof of this Theorem in Section \ref{sec:proofactiveset}. \\

 Intuitively speaking, Theorem \ref{lem:activeset} claims that, as we increase the regularization parameter $\lambda$, the number of elements in the active set, i.e., $\|\hat{\beta}^{\lambda}\|_0$ is decreasing. Also, according to the condition $ \lim_{p \rightarrow \infty} \frac{1}{p} \sum_i \mathbb{I} \left(\beta^{\lambda}_i(p) \neq 0 \right)\allowbreak \leq \delta$, the largest it can get is $\delta = n/p$. Since the number of active elements is a decreasing function of $\lambda$, $\delta$ appears only in the limit $\lambda \rightarrow 0$. Intuitively speaking in the limit $\lambda \rightarrow 0$ the number of non-zero coefficients converges to the number of measurements.  Figure \ref{fig:LassoPathRandom} plots the number of active elements as a function of $\lambda$ for a setting described in Section \ref{subsec:fig:LassoPathRandom}. 
 
\begin{figure}[h!]
\includegraphics[width= 6cm]{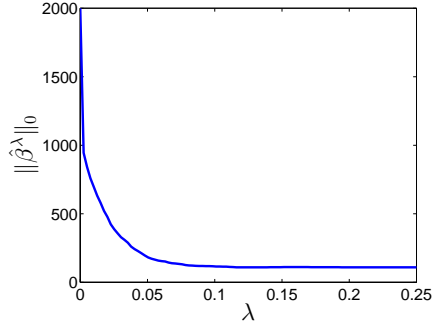}
\caption{The number of active elements in the solution of LASSO as a function of $\lambda$. The size of the active set decreases monotonically as we increase $\lambda$. See Section \ref{subsec:fig:LassoPathRandom} for the details.}
\label{fig:LassoPathRandom}
\end{figure}

Our next result is regarding the behavior of the normalized MSE in terms of the regularization parameter $\lambda$. In asymptotic setting, we prove that the normalized MSE is a quasi-convex function of $\lambda$. See Section 3.4 of \citep{BoydVanderberghe} for a short introduction on quasi-convex functions. Figure \ref{fig:MSE22} exhibits the behavior of MSE as a function of $\lambda$. The detailed description of this problem instance can be found in Section \ref{subsec:fig:2RiskFunctions}. Before we proceed further, we define bowl-shaped functions.

\begin{definition}\label{def:Qcvx}
A quasi-convex function $f:\mathbb{R} \rightarrow \mathbb{R}$  is called bowl-shaped if and only if there exists a unique and finite $x_0 \in \mathbb{R}$ at which $f$ achieves its minimum, i.e., $f(x_0) \leq f(x),~  \forall x \in \mathbb{R}$.
\end{definition}

Here is the formal statement of our result.

\begin{figure}
\includegraphics[width= 6cm]{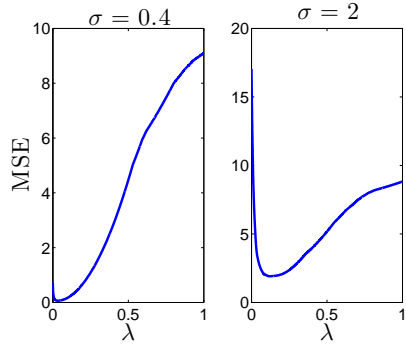}
\caption{Behavior of the MSE as a function of $\lambda$ of LASSO for two different noise variances. See Section \ref{subsec:fig:2RiskFunctions} for details.}
\label{fig:MSE22}
\end{figure}

\begin{theorem}\label{thm:quasiconvex}
Let  $\{\beta_o(p), X(p), w(p)\}$  denote a converging sequence of problem instances as defined in Definition \ref{def:convseq}. Suppose $X_{ij} \overset{iid}{\sim} N(0, 1/n)$. If $\hat{\beta}^{\lambda}(p)$ is the solution of LASSO with regularization parameter $\lambda$, then
 $\lim_{p \rightarrow \infty} \frac{1}{p}\| \hat{\beta}^{\lambda}(p) - \beta_o(p)\|_2^2$ is a quasi-convex function of $\lambda$. Furthermore, if $p_\beta(B=0) \neq 1$, then the function is bowl-shaped.
\end{theorem}
\noindent For the proof, See Section \ref{sec:proofquasiconvex}. 

This result can be employed to design efficient algorithms for tuning $\lambda$. However, since the approach is very similar to what we will explore for AMP, and optimizing the parameters of AMP leads to a solution that is equivalent to the solution of LASSO for optimal $\lambda$ (refer to Section \ref{ssec:connectLASSOAMPopt} for more information), for the sake of brevity we do not use Theorem \ref{thm:quasiconvex} for the optimal tuning of $\lambda$.

\subsection{Solution path and optimal tuning of AMP}\label{sec:CntrbAMP}
This section summarizes our contributions on the approximate message passing algorithm. 
\subsubsection{Solution path of AMP}
A major challenge in employing iterative thresholding algorithms, such as AMP, is the tuning of their free parameters. For instance, in \eqref{eq:ampeq1} $\tau^1, \tau^2, \ldots$ can be considered as free parameters of AMP that need to be tuned properly if an accurate estimate of $\beta_o$ is desired. These parameters have a major impact on both the final reconstruction error, $\lim_{t \rightarrow \infty} \|{\beta}^t- \beta_o\|_2^2/p$, and the convergence rate of the algorithm to its final solution. Ideally speaking, one would like to select the parameters in a way that the final reconstruction error is the smallest, and at the same time the algorithm converges to this solution at the fastest achievable rate. There are two main challenges here: (i) It is not clear if these two criteria can be satisfied at the same time. (ii) To obtain the minimum of $\lim_{p \rightarrow \infty} \|{\beta}^t- \beta_o\|_2^2/p$ we have to solve an optimization problem on $\tau^1, \tau^2, \ldots, \tau^t$ that is computationally demanding and statistically challenging.

To address these two challenges we study the solution path of AMP in terms of the parameters $\tau^1, \ldots, \tau^t$. Intuitively speaking, we will show that achieving the fastest convergence rate at every iteration is equivalent to achieving the minimum of $\lim_{t \rightarrow \infty} \lim_{p \rightarrow \infty} \|{\beta}^t- \beta_o\|_2^2/p$. Furthermore, we will prove that the optimization of $\tau^1, \tau^2, \ldots, \tau^t$ does not need to be done jointly. Below we formalize these statements.

We start with the definition of the optimal threshold parameters of AMP. We overload the notation $\sigma^t(\tau^1, \tau^2, \ldots, \tau^{t-1})$ to emphasize on the fact that the variance of the noise at iteration $t$ depends on all the parameters $\tau_1, \tau_2, \ldots, \tau_{t-1}$.
\begin{definition}\label{def:tau}
A sequence of threshold parameters $\tau^{*,1}, \tau^{*,2}, \ldots, \tau^{*,T-1}$ is called asymptotically optimal for iteration $T$, if and only if
\begin{align}
&\sigma^T(\tau^{*,1}, \ldots, \tau^{*,T-1}) \leq \sigma^{T} (\tau^1, \tau^2, \ldots, \tau^{T-1}), 
\ \ \  \forall \tau^1, \tau^2, \ldots, \tau^{T-1} \in [0, \infty)^{T-1} \nonumber .
\end{align}
\end{definition}
Note that in the above definition we have assumed that the optimal value of $\sigma^T$ is achieved by $ (\tau^{*,1}, \ldots, \allowbreak \tau^{*,T-1})$. This assumption is violated for the case $\beta_o =0$. While we can generalize the definition to include this case, for notational simplicity we skip this special case. The following two remarks clarify some of the main features of this definition:

\begin{remark}
According to Theorem \ref{thm:ampeqpseudo_lip} we have 
\[
\lim_{p \rightarrow \infty} \frac{1}{p} \|\beta^t- \beta_o\|_2^2 = \mathbb{E}_{B,W} [\eta(B+ \sigma^t W; \tau^t) -B]^2, 
\]
almost surely. Hence the optimal choice of parameters introduced in Definition \ref{def:tau}, also minimizes the asymptotic MSE. 
\end{remark}

\begin{remark}
According to Definition  \ref{def:tau}, it seems that in order to tune AMP optimally, we need to know the number of iterations $T$ we plan to run it  (that is usually not known in practice) and then perform a joint optimization over the parameters $\tau^1,\tau^2,\ldots,\tau^{T-1}$ (that is computationally infeasible). However, the properties of AMP's solution path resolve both issues. 
\end{remark}

\begin{theorem}\label{thm:stepwiseoptimal}
Let $\tau^{*,1}, \tau^{*,2}, \ldots, \tau^{*,T-1}$ be asymptotically optimal for iteration $T$. Then, $\tau^{*,1}, \tau^{*,2}$, $\ldots, \tau^{*,t-1}$ are asymptotically optimal for any iteration $t < T$. 
\end{theorem}

 See the proof of this Theorem in Section \ref{proof:greedisgood}. A special case of the above theorem implies that $\tau^{*,1}$ must be optimal for the first iteration. Once $\tau^{*,1}$ is calculated we calculate $\tilde{\beta}^2=\beta^2+X^*z^2$ and again from the above theorem we know that $\tau^{*,2}$ should be optimal for  this step. The same strategy might be used for the other iterations as well. An interesting implication of this result is that the sequence $\tau^{*,1}, \tau^{*,2}, \ldots$ not only achieves the minimum MSE as $t \rightarrow \infty$, but also achieves the fastest convergence rate toward the final solution. 
 
Note that Theorem \ref{thm:stepwiseoptimal} simplifies the problem of tuning all the threshold parameters to the problem of tuning $\tau^t$ at iteration $t$. In other words, as will be clear from the proof of Theorem \ref{thm:stepwiseoptimal} the optimal choice of $\tau^t$ minimizes 
\begin{align}\label{equ:BayesRisk}
R_{B} (\sigma^t, \tau^t ; p_{\beta}) \triangleq \mathbb{E} (\eta(B_o+ \sigma^t W ; \tau^t) - B_o)^2,
\end{align}

Despite the success of Theorem \ref{thm:stepwiseoptimal} in reducing the computational complexity of the optimal tuning of $\tau^1, \ldots, \tau^t$, two major challenges still remain in solving \eqref{equ:BayesRisk}: (i) The distribution of $B_o$ is not known, and hence  $\mathbb{E} (\eta(B_o+ \sigma^t W ; \tau^t) - B_o)^2$ must be estimated. (ii) Even if the distribution of $B$ is known, finding the optimal value of $\tau^t$ may require a search over a grid of values of $\tau$ that is computationally demanding. For the moment we assume that $\mathbb{E} (\eta(B_o+ \sigma^t W ; \tau^t) - B_o)^2$  is known. In the next section we show an asymptotically consistent estimate of this quantity. 

 \begin{lemma}\label{lem:amp:quasi}
If $\mathbb{P} (B = 0) <1$, then $R_{B} (\sigma^t, \tau^t ; p_{\beta})$ is a bowl-shaped function of $\tau^t$. 
\end{lemma} 
This result is similar to Lemma \ref{lem:uniquefixedpointconc} that is proved in Section \ref{sec:Thms}. Hence, we skip the proof. 

As is suggested by the above theorem the risk function that should be minimized is bowl-shaped. As will be clear from the proof, the derivative of $R_{B} (\sigma^t, \tau^t ; p_{\beta})$ is only zero at the optimal value of $\tau^t$. You may see an example of $R_{B} (\sigma^t, \tau^t ; p_{\beta})$ in Figure \ref{fig:diff_est_real}.
 Hence, instead of finding the optimal value of $\tau$ by a grid search we may use faster convex optimization algorithms such as bisection or gradient descent. We will clarify this point in the next section. 

In this section we assumed that $\lim_{p \rightarrow \infty} \frac{1}{p} \|\beta^t- \beta_o\|_2^2$ is given, which is not the case in practice. Next section presents an asymptotically accurate estimate of this quantity and demonstrates the performance of the bisection method on the risk estimate. 

\subsubsection{Stein Unbiased Risk Estimate and Optimal Tuning of AMP}\label{ssec:sureestimate}

One of the issues we raised in the previous section regarding the optimal tuning of AMP algorithm was the fact that $R_{B} (\sigma^t, \tau^t ; p_{\beta}) $ is not given in practice and hence has to be estimated. In this section, we first present an estimate of $R_{B} (\sigma^t, \tau^t ; p_{\beta})$ and will then describe the modifications we have to make to the tuning scheme proposed in the last section.  We consider the following empirical risk estimate of AMP at iteration $t$:
\begin{align}\label{eq:empriskdefamp}
{\hat{R}^t_{h,p} (\tau^t, \tau^{t-1}, \ldots, \tau^1)}&\triangleq\frac{1}{p} \|\tilde{\eta}_h(\beta^t+X^*z^t;\tau^t)-(\beta^t+X^*z^t)\|_2^2+\left(\sigma^t\right)^2\nonumber \\
&~~+\frac{2}{p}\left(\sigma^t\right)^2\left[ \mathbf{1}^T(\tilde{\eta}_h'(\beta^t+X^*z^t;\tau^t)-\mathbf{1})\right],
\end{align}
where $\tilde{\eta}_h(u; \tau)= \eta(u; \tau)* \frac{1}{\sqrt{2 \pi} h} {\rm e}^{-\frac{u^2}{2 h^2}}$, with $*$ and $h$ denoting the convolution operator and a small number, respectively. The role of this convolution is to smooth out the soft thresholding function at the threshold. We would like to emphasize three aspects of this risk estimate:
\begin{enumerate}
\item The dependence of $\hat{R}^t_{h,p} $ on $\tau^1, \tau^2, \ldots, \tau^{t-1}$ might not be clear from the expression we have written in \eqref{eq:empriskdefamp}. However, $\beta^t, z^t$ and $\sigma^t$ depend on $\tau^1, \tau^2, \ldots ,\tau^{t-1}$. 

\item As will be clarified later in the proof, this estimate is inspired by the Stein Unbiased Risk Estimate (SURE). Since SURE can be applied to any weakly differentiable function, the introduction of the smoothing kernel $\frac{1}{\sqrt{2 \pi} h} {\rm e}^{-\frac{u^2}{2 h^2}}$ seems to be unnecessary. Our simulation results seem to agree with this observation too. However, we require it for proving  $\mathbb{P} (\sup_{\tau^t}|\hat{R}_{h,p}^t(\tau^t, \ldots, \tau^1) - \mathbb{R}_B(\sigma^t,\tau^t; p_\beta)  | > \epsilon) \rightarrow 0$. This uniform convergence of the risk estimate is the base of the tuning approach we propose in this section and will be proved in Section \ref{sec:proofuniform} (Theorem \ref{thm:amptune1} ). Hence, the introduction of $h$ might be unnecessary and an artifact of our proof technique.

\item Note that in this estimate $(\sigma^t)^2$ is employed despite the fact that it is not known in practice. It is straightforward to use Lemma 1 of \cite{BaMo10} to show that $\frac{1}{n} (z^t)^T z^t \rightarrow (\sigma^t)^2$, almost surely.\footnote{This estimate has been introduced elsewhere \cite{MalekiThesis}. } Hence, we can replace $(\sigma^t)^2$ in \eqref{eq:empriskdefamp} with  $\frac{1}{n}(z^t)^T z^t$ and all the discussions of this section will be still valid. However, to save some space and simplify the notation we assume that $\sigma^t$ is given.

\end{enumerate}

Let $\mathcal{T}^1, \mathcal{T}^2, \ldots$ denote some known compact subsets of $\mathbb{R}$ such that $\tau^{*,i} \in \mathcal{T}^i$. Combining Theorem \ref{thm:stepwiseoptimal} and the risk estimate, ${\hat{R}^t_{h,p} (\tau^t, \tau^{t-1}, \ldots, \tau^1)}$,  we obtain the following algorithm for tuning the parameters of AMP:

\begin{enumerate}
\item[(i)] Let $\hat{\tau}^1_{p,h} = \arg \min_{\tau^1 \in \mathcal{T}^1} {\hat{R}^1_{h,p} (\tau^1)} $.
\item[(ii)] Fix, $\tau^1, \tau^2, \ldots, \tau^{t-1}$ to $\hat{\tau}^1_{p,h}, \hat{\tau}^2_{p,h}, \ldots, \hat{\tau}^{t-1}_{p,h}$, and calculate $\beta^t$, $z^t$,  \newline ${\hat{R}^t_{h,p} (\tau^t,  \hat{\tau}^{t-1}_{p,h}, \ldots, \hat{\tau}^1_{p,h})}$, and  
\begin{equation}\label{eq:tuningitt1}
\hat{\tau}^t_{p,h} \triangleq \arg \min_{\tau^t \in \mathcal{T}^t} {\hat{R}^t_{h,p} (\tau^t,  \hat{\tau}^{t-1}_{p,h}, \ldots, \hat{\tau}^1_{p,h})}.
\end{equation}
\end{enumerate}
Note that the only extra calculation required for the optimal tuning is the univariate optimization \eqref{eq:tuningitt1} at each iteration. For now, we suppose that we can solve this univariate optimization problem accurately. Under this assumption, the following Theorem proves the consistency of $\hat{\tau}_{p,h}$. 

\begin{theorem} \label{cor:empconvopt}
Consider the converging sequence $\{\beta_o(p), X(p), w(p)\}$, and let the elements of $X$ be drawn iid from $N(0,1/n)$.
Let $ \tau^{*,t}$ denote the optimal threshold according to Definition \ref{def:tau}. Then, for any fixed iteration $t$
\[
\lim_{h \rightarrow 0} \lim_{p \rightarrow \infty } \hat{\tau}^t_{p,h}= \tau^{*,t}
\]
in probability.
\end{theorem}
 The proof of this result is summarized in Section \ref{sec:proofuniform}. 
\begin{figure}
\centering
\includegraphics[width= 7cm]{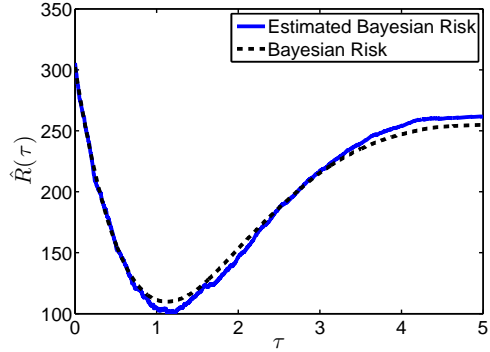}
\caption{The dashed black curve denotes the risk function corresponding to noiseless measurements and the solid blue curve indicates its estimate. For the simulation details, refer to the supplementary file.} 
\label{fig:diff_est_real}
\end{figure}
 The tuning algorithm we described above can be implemented in practice, but requires an exhaustive search over each $\mathcal{T}^t$. As we described in the previous section, since $R_{B} (\sigma^t, \tau^t ; p_{\beta})$ is quasi-convex function of $\tau^t$ we can employ bisection method or gradient descent to reduce the computations further. However, the algorithm has to work with the risk estimate ${\hat{R}^t_{h,p} (\tau^t,  \hat{\tau}^{t-1}_{p,h}, \ldots, \hat{\tau}^1_{p,h})}$ that is not necessarily quasi-convex. Hence, the last challenge is to modify these algorithms in a way that they can work on ${\hat{R}^t_{h,p} (\tau^t,  \hat{\tau}^{t-1}_{p,h}, \ldots, \hat{\tau}^1_{p,h})}$. Here we present an approximate bisection algorithm, but the interested reader may also see the performance of an approximate gradient descent algorithm in our unpublished report \cite{mousavi2013parameterless}.   
 
  We assume that  $\mathcal{T}^t= [\underline{\tau}^t ,\overline{\tau}^t]$. We select two small numbers $\varepsilon$ and $\Delta$, set $\tau = (\underline{\tau}+ \overline{\tau})/2$, and do the following: If $(\hat{R}^t_{p,h}(\tau+\Delta) - \hat{R}^t_{p,h} (\tau))/\Delta < - \varepsilon$, then set $\underline{\tau}^t = \tau$ and repeat the process. If $(\hat{R}^t_{p,h}(\tau+\Delta) - \hat{R}^t_{p,h} (\tau))/\Delta > \varepsilon$, then set $\overline{\tau}^t = \tau$ and repeat the process. Otherwise, stop the process and return $\tau$. This is a slight modification of the bisection method that is popular in optimization. We can analyze the performance of this algorithm under the asymptotic settings.

\begin{theorem}
Consider the converging sequence $\{\beta_o(p), X(p), w(p)\}$, and let the elements of $X$ be drawn iid from $N(0,1/n)$.
Let $\hat{\tau}^t_{B,p}$ denote the value of $\tau$ at which our bisection algorithm stops. Then, there exists $\bar{\tau} \in [\hat{\tau}^t_{B,p}, \hat{\tau}^t_{B,p}+ \Delta] $ such that with probability one
\begin{equation*}\label{eq:fixedpointbisec}
\lim_{h \rightarrow 0} \lim_{p \rightarrow \infty} \left|\frac{ \partial R_B(\sigma, \bar{\tau}; p_\beta)}{\partial \tau} \right| < \epsilon.
\end{equation*}
\end{theorem}
This result is a straightforward application of Theorem \ref{thm:amptune1} (proved in Section \ref{sec:proofuniform}) and is skipped here. 

This theorem implies that if $\epsilon$ and $\Delta$ are small numbers, then the derivative of $R_B$ is also small at $\hat{\tau}^t_{B,p}$ (Note that the derivative of the risk is a continuous function); hence, our estimate will be close to the $\tau^{*,t}$. We will show in Section \ref{sec:simulation results} that (i) The performance of the bisection method is not sensitive to the exact value of $\epsilon$ and $\Delta$, and (ii) These two parameters are easy to tune.

\subsection{Connection to Optimal Tuning of $\lambda$ in LASSO}\label{ssec:connectLASSOAMPopt}
In the last section we showed how one can optimize the threshold parameters of AMP at every iteration. There is a connection between the estimate of the optimally-tuned AMP and the solution of LASSO for the optimal value of $\lambda$. This section explores this connection. Suppose that we run AMP with the optimal parameters $\tau^{*,1}, \tau^{*,2}, \ldots$ defined in Definition \ref{def:tau}, and obtain $\beta^1_*, \beta^2_*, \ldots$. 

\begin{proposition}\label{prop:optlamoptampcon}
Consider the converging sequence $\{\beta_o(p), X(p), w(p)\}$, and let the elements of $X$ be drawn iid from $N(0,1/n)$. Let $\hat{\beta}_\lambda(p)$ denote the solution of LASSO with regularization parameter $\lambda$. Then,
$$\lim_{t \rightarrow \infty} \lim_{p \rightarrow \infty} \frac{1}{p} \|\beta_o(p) - {\beta}_*^t(p)\|_2^2 =\inf_\lambda  \lim_{p \rightarrow \infty} \frac{1}{p} \|\hat{\beta}_\lambda(p) - \beta_o(p)\|_2^2.  $$ 
\end{proposition}
 
 The proof of this result can be found in Section \ref{sec:prooflasttheorem}.  This theorem implies that the final solution, the optimal AMP converges to, has the same MSE as the solution of LASSO with the optimal value of the regularization parameter $\lambda$. This means that not only AMP can provide a fast and low cost algorithm for solving LASSO, but also by optimal tuning of AMP we can avoid the optimal tuning of LASSO.

\section{Proofs of the main results }\label{sec:Thms}
\subsection{Background}\label{ssec:background}

\subsubsection{Quasiconvex functions}

Here, we briefly mention several properties of quasi-convex functions. For a more detailed introduction, see Section 3.4 of \citep{BoydVanderberghe}.  The following basic theorem regarding the quasi-convex functions is a key element in our proofs.

\begin{theorem}\citep{BoydVanderberghe}  \label{thm:quasiconvexnessuf}
A continuous function $f: \mathbb{R} \rightarrow \mathbb{R}$ is quasiconvex if and only if at least one of the following conditions holds:
\begin{itemize}
\item[1.] $f$ is non-decreasing.
\item[2.] $f$ is non-increasing.
\item[3.] There is a point $c$ in the domain of $f$ such that for $t< c$ f is non-increasing and for $t \geq c$, it is non-decreasing. 
\end{itemize}
\end{theorem} 

The following simple lemma shows that shifting and scaling preserve quasi-convexity. 

\begin{lemma} \label{lem:quascvx}
Let $a$ and $b> 0$ be two fixed numbers. Then
$f(x)$ is a quasi-convex function if and only if $g(x)=a+bf(x)$ is a quasi-convex function. 
\end{lemma}
\begin{proof}
First, assume that $f(x)$ is quasi-convex. Then, according to the definition of the quasi-convexity we can write
\begin{align}
g(\alpha x+((1-\alpha)y)&=a+bf(\alpha x+((1-\alpha)y)\nonumber \\ &\leq a+b\max(f(x),f(y))\nonumber \\ &= \max(a+bf(x),a+bf(y))\nonumber \\ &=\max(g(x),g(y)).
\end{align}
Hence, $g(x)$ is quasi-convex as well. On the other hand, suppose that $g(x)$ is quasi-convex. Then, according to the definition, we can write
\begin{align}
f(\alpha x+((1-\alpha)y)&=\frac{g(\alpha x+((1-\alpha)y)-a}{b}\nonumber \\ &\leq \frac{\max(g(x),g(y))-a}{b} \nonumber \\ &\leq \max\left(\frac{g(x)-a}{b},\frac{g(y)-a}{b}\right) \nonumber \\ &\leq \max(f(x),f(y)).
\end{align}
Therefore, $f(x)$ is quasi-convex as well.
\end{proof}

\subsubsection{Risk of the soft thresholding function}
In this section we review some of the basic results (some of which have been proved elsewhere) that will be used in our paper. Let 
$$\Psi(\sigma^2) \triangleq  \sigma_{\omega}^2 + \frac{1}{\delta} \mathbb{E}_{B,Z}\left[(\eta(B+ \sigma Z ; \chi \sigma) -B)^2\right],$$ 
where $B \sim p_\beta$  and $Z \sim N(0,1)$ are two independent random variables. Note that $\Psi$ is a function of $(\delta, \chi,  \sigma_{\omega}^2)$. We assume that all of these parameters are fixed and $\Psi$ is only a function of $\sigma^2$. 

\begin{lemma}\citep{DoMaMo09}\label{lem:convaceRisk}
$\Psi (\sigma^2)$ is a concave function of $\sigma^2$. 
\end{lemma}

One major implication of this theorem that is related to the fixed points of $\Psi(\sigma^2)$ is summarized in the next lemma.

\begin{lemma}\label{lem:uniquefixedpointconc} \citep{DoMaMoNSPT}
For $ \sigma_{\omega}^2>0$, $\Psi$ has a unique fixed point. 
\end{lemma}

Another interesting result that will play crucial role below is the quasi-convexity of the risk of soft thresholding in terms of the threshold. Let $\mu$ be a random variable with distribution $\mu \sim G$ independent of $Z \sim N(0,1)$ and define
\[
r(\tau, G) \triangleq \mathbb{E}_{\mu, Z} \left[(\eta(\mu+Z ; \tau) -\mu)^2\right].
\]
We claim that $r(\tau, G)$ is a quasi-convex function of $\tau$. Define $\delta_0$ as a point mass at zero. For instance $G(x)= 0.5 \delta_0(x) + 0.5 \delta_0(x-1)$ is a distribution of a random variable that takes values $0$ and $1$ with probability half.

\begin{lemma}\label{lem:quasiconvexsoft}
$r( \tau; G)$ is a quasi-convex function of $\tau$. Furthermore, if $\mathbb{P} (\mu = 0) \neq 1$, the function is bowl-shaped. 
\end{lemma}
\begin{proof}
In order to prove this result we use Theorem \ref{thm:quasiconvexnessuf}. It is straightforward to prove that $r(\tau, G)$ is a differentiable function of $\tau$. According to Theorem \ref{thm:quasiconvexnessuf}, we have to prove that the derivative either has no sign change or has one sign change from negative to positive.  Figure \ref{fig:GammaFunc2} shows $\frac{\partial}{\partial \tau} r( \tau; G)$ as a function of $\tau$. As is clear from the figure, the derivative is not an increasing function, and hence we should not expect the second derivative to always be positive.

\begin{figure}[h!]
\includegraphics[width= 6cm]{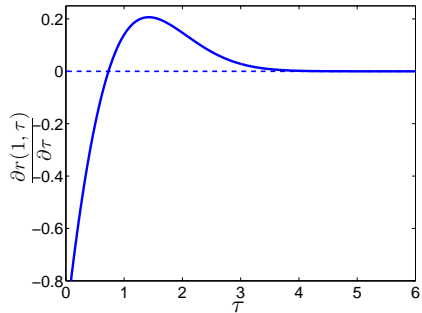}
\caption{The derivative of $\frac{\partial}{\partial \tau} r(\mu, \tau)$ as a function of $\tau$ for $\mu=1$ and $0< \tau < \infty$. Note that the derivative of the risk has only one sign change. Below that point the derivative is negative and above of that point is positive (even though it converges to zero as $\tau \rightarrow \infty$). Hence we expect the risk to be quasi-convex. }
\label{fig:GammaFunc2}
\end{figure}

 Instead we will show that  the ratio $V(\tau;G)\triangleq \frac{\frac{\partial}{\partial \tau}r(\tau, G)}{\left|\frac{\partial}{\partial\tau}r(\tau, \delta_0)\right|}$ is strictly increasing in $\tau \in [0,\infty)$. Once we prove this, we can conclude that $V(\tau, \mu)$ will have at most one sign change, and since $\left|\frac{\partial}{\partial\tau}r(\tau, \delta_0)\right|$ is always positive we can conclude that $\frac{\partial r( \tau, G )}{ \partial \tau}$ will have at most one sign change. According to Theorem \ref{thm:quasiconvexnessuf} this completes the proof of quasi-convexity. Therefore, our main goal in the rest of the proof is to show that $V(\tau,G)$ is strictly increasing. We have
\begin{align}
r(\tau, G)&=\bE_{ \mu, Z}[(\eta(\mu+Z;\tau)-\mu)^2] \nonumber \\
&=\bE_{\mu,Z}[(Z-\tau)^2\bI(\mu+Z\geq\tau)+(Z+\tau)^2\bI(\mu+Z\leq-\tau)+\mu^2\bI(|\mu+Z|<\tau)].
\end{align}
Therefore
\begin{align}\label{equ:DerRisk}
\frac{\partial r(\tau, G)}{\partial \tau}&=\bE_{\mu, Z}[-2(Z-\tau)\bI(\mu+Z\geq\tau)-(Z-\tau)^2\delta(\mu+Z-\tau)\nonumber \\ &\quad \quad+2(Z+\tau)\bI(\mu+Z\leq-\tau)-(Z+\tau)^2\delta(\tau+\mu+Z)\nonumber +\mu^2\delta(\tau-|\mu+Z|)] \nonumber \\ &= 
\bE_{\mu,Z}\left[2(Z+\tau)\bI(\mu+Z\leq -\tau)-2(Z-\tau)\bI(\mu+Z\geq\tau)\right] \nonumber \\ 
&=\underbrace{2 \mathbb{E}_{\mu} \left[\int_{-\infty}^{-(\mu+\tau)}(z+\tau)\phi(z)dz\right]}_{\Gamma_1}-\underbrace{2 \mathbb{E}_{\mu} \left[\int_{\tau-\mu}^{\infty}(z-\tau)\phi(z)dz\right]}_{\Gamma_2},
\end{align}
where $\phi$ denotes the probability density function of the standard normal random variable. Changing the integral variables to  $w=z+\tau$ in $\Gamma_1$ and $w=z-\tau$ in $\Gamma_2$ results in
\begin{align}\label{equ:DerRiskMod}
\frac{\partial r(\tau, G)}{\partial \tau}=2\mathbb{E}_{\mu} \left[\int_{-\infty}^{-\mu}w\phi(w-\tau)dw\right]+2 \mathbb{E}_{\mu}\left[ \int_{-\infty}^{\mu}w\phi(w-\tau)dw\right].
\end{align}
Consequently, we can write
\begin{align}
V(\tau;G) &= \frac{2 \mathbb{E}_{\mu} \bigg[\int_{-\infty}^{-\mu}w\phi(w-\tau)dw\bigg]+2 \mathbb{E}_{\mu} \bigg[\int_{-\infty}^{\mu}w\phi(w-\tau)dw\bigg]}{4\int_{-\infty}^{0}|w|\phi(w-\tau)dw} \nonumber\\
&=  \frac{2\mathbb{E}_{\mu}\bigg[ \int_{-\infty}^0 w \phi(w - \tau) dw -  \int_{-\mu}^0 w \phi(w- \tau) dw\bigg] }{4\int_{-\infty}^{0}|w|\phi(w-\tau)dw} +\frac{2\bE_{\mu}\bigg[ \int_{-\infty}^0 w \phi(w - \tau) dw + \int_0^\mu w \phi(w- \tau)dw\bigg]}{4\int_{-\infty}^{0}|w|\phi(w-\tau)dw}       \nonumber \\
&=\frac{ \mathbb{E}_{\mu} \bigg[\int_{-\mu}^{\mu}|w|\phi(w-\tau)dw\bigg]}{2\int_{-\infty}^{0}|w|\phi(w-\tau)dw}-1.
\end{align}Let $R(\tau, G )\triangleq \frac{\mathbb{E}_{\mu}\left[ \int_{-\mu}^{\mu} |w|\phi(w-\tau)dw\right]}{\int_{-\infty}^{0}|w|\phi(w-\tau)dw}$. Taking the derivative of $R(\tau,G)$ with respect to $\tau$ gives
\begin{align}
\frac{\partial R(\tau,G)}{\partial \tau}&=\frac{\mathbb{E}_{\mu}\left[\int_{-\mu}^{\mu}(w-\tau)|w|\phi(w-\tau)dw\right]\int_{-\infty}^0 |w|\phi(w-\tau)dw}{\left(\int_{-\infty}^{0}|w|\phi(w-\tau)dw\right)^2} \nonumber \\&~~~~-\frac{\mathbb{E}_{\mu}\left[\int_{-\mu}^{\mu}|w|\phi(w-\tau)dw\right] \int_{-\infty}^{0}(w-\tau)|w|\phi(w-\tau)dw} {\left(\int_{-\infty}^{0}|w|\phi(w-\tau)dw\right)^2} \nonumber \\ &=\frac{\overbrace{\mathbb{E}_{\mu}\left[ \int_{-\mu}^{\mu}w|w|\phi(w-\tau)dw\right]\ }^{\Gamma_3}\int_{-\infty}^0 |w|\phi(w-\tau)dw}{\left(\int_{-\infty}^{0}|w|\phi(w-\tau)dw\right)^2} \nonumber \\&~~~~ +\frac{\overbrace{ \mathbb{E}_{\mu}\left[ \int_{-\mu}^{\mu}|w|\phi(w-\tau)dw\right]\int_{0}^{\infty}w|w|\phi(w+\tau)dw}^{\Gamma_4}}{\left(\int_{-\infty}^{0}|w|\phi(w-\tau)dw\right)^2}. 
\end{align}
 $\Gamma_3$ can be simplified to
\begin{align}
\Gamma_3= \mathbb{E}_{\mu} \left[\int_0^{|\mu|}w^2(\phi(w-\tau)-\phi(w+\tau))dw\right]>0.
\end{align}
Therefore, since $\Gamma_4>0$, we have $\frac{\partial R(\tau,G)}{\partial \tau}>0$ which proves that $R(\tau, G)$ is an increasing function. This, in addition with the fact that $V(0; G) <0$, completes the proof of quasi-convexity. 
\end{proof}

According to Theorem \ref{thm:quasiconvexnessuf}, combining the facts that $\frac{\partial V(\tau, G)}{\partial \tau} \geq 0$ and $V(0, G) <0$ proves the quasi-convexity of the risk. However, this does not mean that the risk function is bowl-shaped (See Definition \ref{def:Qcvx}). In fact, the risk function is bowl-shaped if sign-change happens at certain value of $\tau$. We would now like to prove that the zero crossing in fact happens if $G \neq \delta_0$ (as mentioned before $\delta_0$ denotes the distribution of a point mass at zero). 
\begin{lemma} \label{cor:ZC} The risk of the soft thresholding function satisfies
$$\frac{\partial r(\tau, G)}{\partial \tau}\big|_{\tau=0}<0.$$
\end{lemma}
\begin{proof}
We have
\[
\frac{\partial r(\tau, G)}{\partial \tau}\bigg|_{\tau=0}=2\mathbb{E}_{\mu}\left[\int_{-\infty}^{-\mu}z\phi(z)dz-\int_{-\mu}^{\infty}z\phi(z)dz\right] =2\mathbb{E}_{\mu}\left[-\frac{2e^{\frac{-\mu^2}{2}}}{\sqrt{2\pi}}\right] <0.
\]
\end{proof}
If we prove that for large enough $\hat{\tau}$ we have $\frac{\partial r(\tau, G)}{\partial \tau}\big|_{\tau=\hat{\tau}}>0$, then we can conclude that $\frac{\partial r(\tau, G)}{\partial \tau}$ has at least one zero-crossing. On the other hand, according to Lemma \ref{lem:quasiconvexsoft}, $\frac{\partial r(\tau, G)}{\partial \tau}$ has at most one zero crossing. Therefore, we would actually prove that $\frac{\partial r(\tau, G)}{\partial \tau}$ has exactly one zero-crossing and, consequently, the risk of soft-thresholding function is a bowl-shaped function.
We require the following two lemmas in the proof of our main claim, which is summarized in Proposition \ref{pro:ZC}.

\begin{lemma}\label{lem:eps1}
Suppose that $\mathbb{P} (\mu \neq  0) \neq 0$. Then, there exists an $\epsilon_0>0$ such that $\mathbb{P}(|\mu|>\epsilon_0)>0$.
\end{lemma}
\noindent The proof is trivial and hence is skipped here.

\begin{proposition}\label{pro:ZC}
If $\mathbb{P} (\mu \neq 0) \neq 0$. Then $\frac{\partial r(\tau, G)}{\partial \tau}\big|_{\tau=\hat{\tau}}>0$ for the large values of $\hat{\tau}$ that make \eqref{eq:referintheorem} positive.
\end{proposition}
\begin{proof} Without loss of generality and for simplicity of notation, we assume that $\mathbb{P} (\mu \geq 0) =1$.
Using the same change of variable technique as in (\ref{equ:DerRiskMod}), we have
\begin{align}
\frac{\partial r(\tau, G)}{\partial \tau}&=2\mathbb{E}_{\mu} \left[\int_{-\infty}^{-\mu}w\phi(w-\tau)dw\right]+2 \mathbb{E}_{\mu}\left[ \int_{-\infty}^{\mu}w\phi(w-\tau)dw\right] \nonumber \\
&=2\mathbb{E}_{\mu} \left[\int_{-\infty}^{-\mu}w\phi(w-\tau)dw+\int_{-\infty}^{0}w\phi(w-\tau)dw \right] +\mathbb{E}_{\mu}\left[ \int_{0}^{\mu}w\phi(w-\tau)dw\right] \nonumber \\
&\geq 4\underbrace{\mathbb{E}_{\mu} \left[\int_{-\infty}^{0}w\phi(w-\tau)dw\right]}_{\Gamma_1}+2\underbrace{\mathbb{E}_{\mu}\left[ \int_{0}^{\mu}w\phi(w-\tau)dw\right]}_{\Gamma_2}.
\end{align}
Note that $\Gamma_1 <0$, and $\Gamma_2 >0$. Our goal is to show that for large values of $\tau$, $\Gamma_2 > |\Gamma_1|$. To achieve this goal,
we find an upper bound for $|\Gamma_1|$ and a lower bound for $\Gamma_2$. Simplifying $\Gamma_1$ gives
\begin{align}\label{eq:Gamma1Simp}
\Gamma_1&=\mathbb{E}_{\mu} \left[\int_{-\infty}^{0}w\phi(w-\tau)dw\right] =\frac{2e^{\frac{-\tau^2}{2}}}{\sqrt{2\pi}}\int_{-\infty}^{0}we^{\frac{-w^2}{2}+w\tau}dw \nonumber \\
&\geq \frac{2e^{\frac{-\tau^2}{2}}}{\sqrt{2\pi}}\int_{-\infty}^{0}we^{\frac{-w^2}{2}}dw = -\frac{e^{\frac{-\tau^2}{2}}}{\sqrt{2\pi}}\int_{-\infty}^{0}|w|e^{\frac{-w^2}{2}}dw \geq -\sqrt{\frac{2}{\pi}} e^{\frac{-\tau^2}{2}}.
\end{align}
Therefore, it is sufficient to prove that $\Gamma_2> \sqrt{\frac{2}{\pi}}  e^{\frac{-\tau^2}{2}}$. To achieve this goal, we first use $\epsilon_0$  as introduced in Lemma \ref{lem:eps1}; hence, we can obtain a lower bound for $\Gamma_2$ as:
\begin{align}\label{eq:Gamma2Simp}
\Gamma_2&=\mathbb{E}_{\mu} \left[\int_{0}^{\mu}w\phi(w-\tau)dw\right] \geq \mathbb{P}(\mu > \epsilon_0) \int_{0}^{\epsilon_0} w \phi (w- \tau) dw \nonumber \\
 &=\mathbb{P}(\mu > \epsilon_0) \frac{e^{\frac{-\tau^2}{2}}}{\sqrt{2\pi}} \int_{0}^{\epsilon_0} we^{\frac{-w^2}{2}+w \tau}dw \geq \mathbb{P}(\mu > \epsilon_0) \frac{e^{\frac{-\tau^2}{2}- \frac{\epsilon_0^2}{2}} }{\sqrt{2\pi}} \int_{0}^{\epsilon_0} we^{w \tau}dw \nonumber \\
& = \mathbb{P}(\mu > \epsilon_0) \frac{e^{\frac{-\tau^2}{2}- \frac{\epsilon_0^2}{2}} }{\sqrt{2\pi}} \left( \frac{\epsilon_0 {\rm e}^{\epsilon_0 \tau}}{\tau} - \frac{{\rm e}^{\epsilon_0 \tau}}{\tau^2} + \frac{1}{\tau^2} \right).
\end{align}
Define $c_1 \triangleq \sqrt{\frac{2}{\pi}}$ and $c_2 \triangleq \mathbb{P} (\mu > \epsilon_0) \frac{\rm e^{- \epsilon_0^2/2}}{\sqrt{2 \pi}}$. Combining (\ref{eq:Gamma1Simp}) and (\ref{eq:Gamma2Simp}) gives 
\begin{align}\label{eq:referintheorem}
\Gamma_1+\Gamma_2 \geq \underbrace{ \left(c_2 \frac{ \epsilon_0 e^{\epsilon_0 \tau}}{\tau}- c_2 \frac{\rm e^{\epsilon_0 \tau}}{\tau^2} +  \frac{c_2}{\tau^2}+ c_1\right) }_{\Lambda}e^{\frac{-\tau^2}{2}}.
\end{align} 
Clearly, for large values of $\tau$, the first term of $\Lambda$ is dominant, and hence $\Lambda$ is positive. Therefore, the proof is complete.
\end{proof}
Combining Corollary \ref{cor:ZC}, Proposition \ref{pro:ZC}, and the fact that $\frac{\partial r(\tau,G)}{\partial \tau}$ has at most one zero-crossing proves that $r(G,\tau)$ is a quasi-convex and bowl-shaped function (see Definition \ref{def:Qcvx}).

\begin{proposition}
$r(G,\tau)$ is a strictly decreasing function of $\tau$ for the case $G=\delta_0$.
\end{proposition}
\begin{proof}
We can write
\begin{align}
r(\delta_0,\tau)&=\bE\left[\eta^2(Z;\tau)\right] =\bE\left[(Z-\tau)^2\bI(Z>\tau)+(Z+\tau)^2\bI(Z<-\tau)\right] \nonumber \\
&=\int_{\tau}^{\infty}(z-\tau)^2\phi(z)dz+\int_{-\infty}^{-\tau}(z+\tau)^2\phi(z)dz.
\end{align}
Therefore, if we take the derivative with respect to $\tau$, we have $\frac{dr(\delta_0,\tau)}{d\tau}=-4\int_{\tau}^{\infty}z\phi(z)dz<0$.
\end{proof}

\subsection{Proof of Theorem \ref{lem:activeset}}\label{sec:proofactiveset}
First note that, according to Theorem \ref{thm:lassodetstate}, we have \newline $\lim_{N \rightarrow \infty} \frac{1}{N} \sum_{i=1}^N \mathbb{I} (\hat{\beta}^{\lambda}_{ i}(N) \neq 0) = \mathbb{P} (|\eta(B + \hat{\sigma} Z ; \chi \hat{\sigma}) | > 0)$, where $(\hat{\sigma},\lambda,\chi)$ satisfy the following equations: 
\begin{align}\label{equ:FixedPoint1}
&\hat{\sigma}^2=\sigma_z^2+\frac{1}{\delta}\bE_{B,Z}\left[(\eta(B+\hat{\sigma}Z; \chi \hat{\sigma})-B)^2\right], \nonumber \\
& \lambda=\chi\hat{\sigma}\left(1-\frac{1}{\delta}\bP(|B+\hat{\sigma}Z|>\chi \hat{\sigma})\right).
\end{align}
Therefore, we have to prove that $\frac{ d  \mathbb{P}(|\eta(B + \hat{\sigma} Z ; \chi \hat{\sigma}) | > 0) }{d \lambda} <0$. The difficulty of this task is clear from \eqref{equ:FixedPoint1}: $(\chi, \hat{\sigma}, \lambda)$ are complicated functions of each other whose explicit formulations are not known. Using the chain rule we have
\begin{eqnarray}\label{equ:ChainRule}
\frac{d}{d\lambda}\bP(|\eta(B+\hat{\sigma}Z;\chi \hat{\sigma})|>0)=\frac{d}{d\chi}\bP(|\eta(B+\hat{\sigma}Z;\chi\hat{\sigma})|>0)\frac{d\chi}{d\lambda}.
\end{eqnarray}
This expression enables us to break the proof into the following two simpler parts: (i) We prove that $\Big( \frac{d}{d\chi}\bP(|\eta(B+\hat{\sigma}Z;\chi\hat{\sigma})|>0) \Big) <0$. (ii) We prove that $\frac{d \chi} {d \lambda} >0$. Combining these two results with \eqref{equ:ChainRule} completes the proof. 

\begin{lemma}\label{lem:detovertau}
Let $(\chi, \hat{\sigma})$ satisfy \eqref{equ:FixedPoint1}. Then 
$\Big( \frac{d}{d\chi}\bP(|\eta(B+\hat{\sigma}Z;\chi \hat{\sigma})|>0) \Big) <0.$
\end{lemma}

\begin{proof}
Since $|\eta(B+\hat{\sigma}Z;\chi \hat{\sigma})|>0$ if and only if $|\frac{B}{\hat{\sigma}}+Z|>\chi$, it is sufficient to prove $\frac{d}{d\chi}\bP(|\frac{B}{\hat{\sigma}}+Z|>\chi)<0$. We have
\begin{align}\label{equ:Derivative}
\frac{d}{d\chi}\bP\left(\left|\frac{B}{\hat{\sigma}}+Z\right|>\chi \right)=\frac{\partial}{\partial\hat{\sigma}}\bP\left(\left|\frac{B}{\hat{\sigma}}+Z\right|>\chi\right)\frac{d\hat{\sigma}}{d\chi}+\frac{\partial}{\partial\chi}\bP\left(\left|\frac{B}{\hat{\sigma}}+Z\right|>\chi \right).
\end{align}
The rest of the proof has four main steps: (i) Calculation of $\frac{\partial}{\partial\hat{\sigma}}\bP\left(\left|\frac{B}{\hat{\sigma}}+Z\right|>\chi \right)$. (ii) Calculation of $\frac{\partial}{\partial\chi}\bP\big(\big|\frac{B}{\hat{\sigma}}\allowbreak+Z\big|>\chi \big)$. (iii) Calculation of $\frac{d\hat{\sigma}}{d\chi}$. Finally, (iv) Plugging the results of the above three steps in \eqref{equ:Derivative} and proving that $\Big( \frac{d}{d\chi}\bP\left(\left|\frac{B}{\hat{\sigma}}+Z\right|>\chi \right) \Big) <0$. 
Here are these four steps realized: \\

\textbf{Step i: Calculation of $\frac{\partial}{\partial\hat{\sigma}}\bP\left(\left|\frac{B}{\hat{\sigma}}+Z\right|>\chi \right)$}: Simple algebra leads us to 
\begin{align}
\frac{\partial}{\partial \hat{\sigma}} \mathbb{P} \left( \left| \frac{B}{\hsig} +Z \right| \geq \chi \right)&= \mathbb{E}_B\left[ \frac{\partial}{\partial \hsig}\bE_Z\left[\bI\left(\left|\frac{B}{\hsig}+Z\right|\geq\chi \right)\bigg|B\right]\right] \nonumber \\ & =\bE_B \left[\frac{\partial}{\partial \hsig}\left(\int_{\chi-\frac{B}{\hsig}}^{\infty}\phi(z)dz+\int_{-\infty}^{-\chi-\frac{B}{\hsig}}\phi(z)dz\right)\right] \nonumber \\ &=\bE_B\left[\frac{B}{\hsig^2}\phi\left(\frac{B}{\hsig}+\chi\right)-\frac{B}{\hsig^2}\phi\left(\frac{B}{\hsig}-\chi \right)\right].
\end{align}

\textbf{Step ii: Calculation of $\frac{\partial}{\partial\chi}\bP\left(\left|\frac{B}{\hat{\sigma}}+Z\right|>\chi  \right)$}: Similar to Step 1, we have
\begin{align}
\frac{\partial}{\partial \chi} \mathbb{P} \left( \left| \frac{B}{\hsig} +Z \right| \geq \chi \right)&= \mathbb{E}_B\left[ \frac{\partial}{\partial \chi} \bE_Z \left[\bI\left(\left|\frac{B}{\hsig}+Z\right|\geq\chi \right)\bigg|B\right]\right] \nonumber \\ & =\bE_B \left[\frac{\partial}{\partial \chi}\left(\int_{\chi -\frac{B}{\hsig}}^{\infty}\phi(z)dz+\int_{-\infty}^{-\chi-\frac{B}{\hsig}}\phi(z)dz\right)\right] \nonumber \\ &=\bE_B\left[-\phi\left(\frac{B}{\hsig}+\chi \right)-\phi\left(\frac{B}{\hsig}-\chi \right)\right].
\end{align}

\textbf{Step iii: Calculation of $\frac{d\hat{\sigma}}{d\chi}$}: We can rewrite (\ref{equ:FixedPoint1}) as 
\begin{eqnarray}\label{equ:ReFixedPoint}
\delta=\frac{\delta \sigma_z^2}{\hsig^2}+\bE_{B,Z}\left[\left(\eta\left(\frac{B}{\hsig}+Z;\chi \right)-\frac{B}{\hsig}\right)^2\right].
\end{eqnarray}
Taking the derivative of both sides of (\ref{equ:ReFixedPoint}) with respect to $\chi$ yields
\begin{align}
\displaybreak[0]
 0&=\frac{-2\delta\sigma_z^2}{\hsig^3}\frac{d\hsig}{d\chi}+\frac{d}{d \chi}\bE_{B,Z}\left[\left(\eta\left(\frac{B}{\hsig}+Z;\chi \right)-\frac{B}{\hsig}\right)^2\right] \nonumber \\
&  =\frac{-2\delta\sigma_z^2}{\hsig^3}\frac{d\hsig}{d\chi}+\frac{\partial}{\partial \chi}\bE_{B,Z}\left[\left(\eta\left(\frac{B}{\hsig}+Z;\chi \right)-\frac{B}{\hsig}\right)^2\right] \nonumber \\ &~~~~+\frac{\partial}{\partial \hsig}\bE_{B,Z}\left[\left(\eta\left(\frac{B}{\hsig}+Z;\chi \right)-\frac{B}{\hsig}\right)^2\right] \frac{d\hsig}{d\chi} \nonumber \\
&\displaybreak[0] =\frac{-2\delta\sigma_z^2}{\hsig^3}\frac{d\hsig}{d\chi}+2\bE_{B,Z}\left[\left(\eta\left(\frac{B}{\hsig}+Z;\chi\right)-\frac{B}{\hsig}\right)\frac{\partial}{\partial \chi}\left(\eta\left(\frac{B}{\hsig}+Z;\chi \right)-\frac{B}{\hsig}\right)\right] \nonumber \\ \displaybreak[0]&~~~~+2\bE_{B,Z}\left[\left(\eta\left(\frac{B}{\hsig}+Z;\chi \right)-\frac{B}{\hsig}\right)\frac{\partial}{\partial \hsig}\left(\eta\left(\frac{B}{\hsig}+Z;\chi \right)-\frac{B}{\hsig}\right)\right] \frac{d\hsig}{d\chi} \nonumber \\
&\displaybreak[0]=\frac{-2\delta\sigma_z^2}{\hsig^3}\frac{d\hsig}{d\chi}+2\bE_{B,Z}\bigg[\left(\eta\left(\frac{B}{\hsig}+Z;\chi\right)-\frac{B}{\hsig}\right)\bigg(-\bI\left(\frac{B}{\hsig}+Z>\chi \right)\nonumber \\ \displaybreak[0]
&\qquad \qquad \quad~~-\left(\frac{B}{\hsig}+Z-\chi\right)\delta\left(\frac{B}{\hsig}+Z-\chi \right)\nonumber \\ \displaybreak[0] &
\qquad \qquad \quad~~+\bI\left(\frac{B}{\hsig}+Z<-\chi \right)-\left(\frac{B}{\hsig}+Z+\chi \right)\delta\left(\frac{B}{\hsig}+Z+\chi \right)\bigg)\bigg]\nonumber \\
 &~~~~+2\bE_{B,Z}\bigg[\left(\eta\left(\frac{B}{\hsig}+Z;\chi \right)-\frac{B}{\hsig}\right)\bigg(\left(\frac{-B}{\hsig^2}\right)\bI\left(\frac{B}{\hsig}+Z>\chi \right)\nonumber \\
&\qquad \qquad \quad~~+\left(\frac{B}{\hsig}+Z-\chi \right)\delta\left(\frac{B}{\hsig}+Z-\chi \right)\left(\frac{-B}{\hat{\sigma}^2}\right)\nonumber \\
&\qquad \qquad \quad~~+\left(\frac{-B}{\hsig^2}\right)\bI\left(\frac{B}{\hsig}+Z<-\chi \right)\nonumber \\
&\qquad \qquad \quad~~+\left(\frac{B}{\hsig}+Z+\chi \right)\delta\left(\frac{B}{\hsig}+Z+\chi \right)\left(\frac{B}{\hsig^2}\right)\bigg)\bigg] \frac{d\hsig}{d\chi}\nonumber \\ 
&\displaybreak[0]=\frac{-\delta\sigma_z^2}{\hsig^3}\frac{d\hsig}{d\chi}+\bE_{B,Z}\left[\frac{B}{\hsig^2}\left(\eta\left(\frac{B}{\hsig}+Z;\chi \right)-\frac{B}{\hsig}\right)\left(\bI\left(\Big|\frac{B}{\hsig}+Z\Big|<\chi \right)\right)\right]\frac{d\hsig}{d\chi}
\nonumber \\ & ~~~~+\bE_{B,Z}\bigg[\left(\eta\left(\frac{B}{\hsig}+Z;\chi \right)-\frac{B}{\hsig}\right)\nonumber \\ &\qquad \qquad \quad~~\left(-\bI\left(\frac{B}{\hsig}+Z>\chi \right)+\bI\left(\frac{B}{\hsig}+Z<-\chi \right)\right)\bigg].\end{align}
Therefore, we can write
\begin{align}\label{equ:FinalDer}
\frac{d\hsig}{d\chi}&=\frac{E_{B,Z}\left[\left(\eta\left(\frac{B}{\hsig}+Z;\chi \right)-\frac{B}{\hsig}\right)\left(-\bI\left(\frac{B}{\hsig}+Z>\chi \right)+\bI\left(\frac{B}{\hsig}+Z<-\chi \right)\right)\right]}{\frac{\delta\sigma_z^2}{\hsig^3}+\bE_{B,Z}\left[\frac{B^2}{\hsig^3}\left(\bI\left(\Big|\frac{B}{\hsig}+Z\Big|<\chi \right)\right)\right]} \nonumber \\ 
&=\frac{\bE_B\left[\int_{\chi-\frac{B}{\hsig}}^{\infty}(\chi-z)\phi(z)dz+\int_{-\infty}^{-\chi-\frac{B}{\hsig}}(z+\chi)\phi(z)dz\right]}{\frac{\delta\sigma_z^2}{\hsig^3}+\bE_{B,Z}\left[\frac{B^2}{\hsig^3}\left(\bI\left(\Big|\frac{B}{\hsig}+Z\Big|<\chi \right)\right)\right]}.
\end{align}

\noindent \textbf{Step iv:  Proving that $\Big( \frac{d}{d\chi}\bP\left(\left|\frac{B}{\hat{\sigma}}+Z\right|>\chi \right) \Big) <0$}: If we plug (\ref{equ:FinalDer}) into (\ref{equ:Derivative}), we obtain
\begin{align}\label{equ:DerivativeContinued}
\displaybreak[0]
&\frac{d}{d\chi}\bP\left(\left|\frac{B}{\hat{\sigma}}+Z\right|>\chi \right) =\frac{d}{d\chi}\bE_{B,Z}\left[\bI\left(\left|\frac{B}{\hsig}+Z\right|>\chi \right)\right] \nonumber \\
&=\frac{\partial}{\partial \chi}\bE_{B,Z}\left[\int_{\chi-\frac{B}{\hsig}}^{\infty} \phi(z)dz+\int_{-\infty}^{-\chi-\frac{B}{\hsig}}\phi(z)dz\right]\nonumber \\ &~~~~+\frac{\partial}{\partial \hsig}\bE_{B,Z}\left[\int_{\chi -\frac{B}{\hsig}}^{\infty} \phi(z)dz +\int_{-\infty}^{-\chi-\frac{B}{\hsig}}\phi(z)dz\right]\frac{d\hsig}{d\chi}  \nonumber \\  \displaybreak[0] &=-\overbrace{\frac{\bE_{B,Z}\left[\phi\left(\chi-\frac{B}{\hsig}\right)+\phi\left(\chi+\frac{B}{\hsig}\right)\right]\frac{\delta\sigma_z^2}{\hsig^3}}{\frac{\delta\sigma_z^2}{\hsig^3}+\bE_{B,Z}\left[\frac{B^2}{\hsig^3}\left(\bI\left(\Big|\frac{B}{\hsig}+Z\Big|<\chi \right)\right)\right]}}^{\Lambda_1}\nonumber \\&~~~~-\overbrace{\frac{\bE_{B,Z}\left[\phi\left(\chi-\frac{B}{\hsig}\right)+\phi\left(\chi+\frac{B}{\hsig}\right)\right] \bE_{B,Z}\left[\frac{B^2}{\hsig^3}\left(\bI\left(\Big|\frac{B}{\hsig}+Z\Big|<\chi \right)\right)\right]}{\frac{\delta\sigma_z^2}{\hsig^3}+\bE_{B,Z}\left[\frac{B^2}{\hsig^3}\left(\bI\left(\Big|\frac{B}{\hsig}+Z\Big|<\chi \right)\right)\right]}}^{\Lambda_2} \nonumber \\ \displaybreak[0] &~~~~+\overbrace{\frac{\bE_{B,Z}\left[\frac{B}{\hsig^2}\left(-\phi\left(\chi -\frac{B}{\hsig}\right)+\phi\left(\chi+\frac{B}{\hsig}\right)\right)\right]\chi\left(\int_{\chi-\frac{B}{\hsig}}^{\infty}\phi(z)dz+\int_{-\infty}^{-\chi-\frac{B}{\hsig}}\phi(z)dz\right)}{\frac{\delta\sigma_z^2}{\hsig^3}+\bE_{B,Z}\left[\frac{B^2}{\hsig^3}\left(\bI\left(\Big|\frac{B}{\hsig}+Z\Big|<\chi \right)\right)\right]}}^{\Lambda_3}\nonumber \\ \displaybreak[0] &~~~~+\overbrace{\frac{\bE_{B,Z}\left[\frac{B}{\hsig^2}\left(-\phi\left(\chi-\frac{B}{\hsig}\right)+\phi\left(\chi+\frac{B}{\hsig}\right)\right)\right]\left(-\int_{\chi-\frac{B}{\hsig}}^{\infty}z\phi(z)dz+\int_{-\infty}^{-\chi-\frac{B}{\hsig}}z\phi(z)dz\right)}{\frac{\delta\sigma_z^2}{\hsig^3}+\bE_{B,Z}\left[\frac{B^2}{\hsig^3}\left(\bI\left(\Big|\frac{B}{\hsig}+Z\Big|<\chi \right)\right)\right]}}^{\Lambda_4}.
\end{align}
Considering the terms in (\ref{equ:DerivativeContinued}), we can write: (i) $\Lambda_1 \geq 0$, since all the terms in $\Lambda_1$ are positive. (ii) $\Lambda_3 \leq 0$, since for $B<0$ we have  $\phi\left(\frac{B}{\hsig}+\chi \right)>\phi\left(\frac{B}{\hsig}-\chi \right)$ and for $B>0$ we have $\phi\left(\frac{B}{\hsig}+\chi \right)<\phi\left(\frac{B}{\hsig}-\chi \right)$. (iii) For the other two terms, $\Lambda_2$ and $\Lambda_4$, we have 
\begin{align}\label{equ:Simplify2}
& \Lambda_2+\Lambda_4 \nonumber \\
\displaybreak[0]
&\displaybreak[0] = -\frac{\bE_{B}\left[\left(\phi(\chi-\frac{B}{\hsig})+\phi(\chi +\frac{B}{\hsig})\right)\left(\frac{B}{\hsig^2}\int_{-\chi-\frac{B}{\hsig}}^{\chi-\frac{B}{\hsig}}z\phi(z)dz+\frac{B^2}{\hsig^3}\int_{-\chi-\frac{B}{\hsig}}^{\chi-\frac{B}{\hsig}}\phi(z)dz\right)\right]}{\frac{\delta\sigma_z^2}{\hsig^3}+\bE_{B,Z}\left[\frac{B^2}{\hsig^3}\left(\bI\left(\Big|\frac{B}{\hsig}+Z\Big|<\chi\right)\right)\right]}  
\nonumber \\
 &=-\frac{\bE_{B}\left[\left(\phi(\chi-\frac{B}{\hsig})+\phi(\chi+\frac{B}{\hsig})\right)\left(\frac{B}{\hsig^2}\int_{-\chi-\frac{B}{\hsig}}^{\chi-\frac{B}{\hsig}}\left(z+\frac{B}{\hsig}\right)\phi(z)dz\right)\right]}{\frac{\delta\sigma_z^2}{\hsig^3}+\bE_{B,Z}\left[\frac{B^2}{\hsig^3}\left(\bI\left(\Big|\frac{B}{\hsig}+Z\Big|<\chi \right)\right)\right]}
\nonumber \\ 
 &=-\frac{\bE_{B}\left[\left(\phi(\chi-\frac{B}{\hsig})+\phi(\chi+\frac{B}{\hsig})\right)\overbrace{\left(\frac{B}{\hsig^2}\int_{-\chi}^{\chi}w\phi\left(w-\frac{B}{\hsig}\right)dz\right)}^{\Upsilon}\right]}{\frac{\delta\sigma_z^2}{\hsig^3}+\bE_{B,Z}\left[\frac{B^2}{\hsig^3}\left(\bI\left(\Big|\frac{B}{\hsig}+Z\Big|<\chi \right)\right)\right]}.
\end{align}
Regarding $\Upsilon$ in (\ref{equ:Simplify2}), for any $\epsilon \in [-\chi,0)$, If $B<0~(\text{or}~ B>0)$ then $\epsilon\left(\phi\left(\epsilon-\frac{B}{\hsig}\right)-\phi\left(-\epsilon-\frac{B}{\hsig}\right)\right)<0 ~\Big( \text{or}~\epsilon\left(\phi\left(\epsilon-\frac{B}{\hsig}\right)-\phi\left(-\epsilon-\frac{B}{\hsig}\right)\right)>0\Big)$. Therefore, $\Upsilon>0$ and consequently $\Lambda_2+\Lambda_4<0$. 

Putting (i), (ii), and (iii) together we conclude $\frac{d}{d\chi}\bP\left(\left|\frac{B}{\hsig}+Z\right| >0 \right)<0$.
\end{proof}

\begin{lemma}\label{lem:lambdataumon}
Let $(\chi, \hat{\sigma}, \lambda)$ satisfy \eqref{equ:FixedPoint1}. Then $\frac{d \lambda}{d \chi} >0$. 
\end{lemma}
\begin{proof}
 Taking the derivative from both sides of (\ref{equ:FixedPoint1}) with respect to $\chi$ yields
\begin{align}
\frac{d\lambda}{d\chi}=\underbrace{\frac{d(\chi \hsig)}{d\chi}\left(1-\frac{1}{\delta}\bP(|B+\hat{\sigma}Z|> \chi \hat{\sigma} )\right)}_{\Theta_1}-\underbrace{\chi\hsig\frac{d}{d\chi}\bP(|B+\hat{\sigma}Z|>\chi\hat{\sigma})}_{\Theta_2}.
\end{align}

According to Lemma \ref{lem:detovertau}, $\Theta_2<0$. Regarding $\Theta_1$, since $\lambda \geq 0$ satisfies (\ref{equ:FixedPoint1}), \newline $\left(1-\frac{1}{\delta}\bP(|B+\hat{\sigma}Z|>\chi \hat{\sigma})\right)>0$, and hence it is sufficient to prove that $\frac{d(\chi\hsig)}{d\chi}\geq 0 $. We have 
\begin{align}\label{equ:DerTauSigma}
&\frac{d(\chi \hsig)}{d\chi}=\hsig+\chi \frac{d\hsig}{d\chi}\nonumber \\ &\overset{(a)}{=} \hsig+\chi\frac{E_{B,Z}\left[\left(\eta\left(\frac{B}{\hsig}+Z;\chi \right)-\frac{B}{\hsig}\right)\left(-\bI\left(\frac{B}{\hsig}+Z>\chi \right)+\bI\left(\frac{B}{\hsig}+Z<-\chi \right)\right)\right]}{\frac{\delta\sigma_z^2}{\hsig^3}+\bE_{B,Z}\left[\frac{B^2}{\hsig^3}\left(\bI\left(\Big|\frac{B}{\hsig}+Z\Big|<\chi \right)\right)\right]} 
\nonumber \\ &=\frac{\frac{\delta\sigma_z^2}{\hsig^2}+\bE_{B,Z}\left[\frac{B^2}{\hsig^2}\left(\bI\left|\frac{B}{\hsig}+Z\right|<\chi \right)\right]}{\frac{\delta\sigma_z^2}{\hsig^3}+\bE_{B,Z}\left[\frac{B^2}{\hsig^3}\left(\bI\left(\Big|\frac{B}{\hsig}+Z\Big|<\chi \right)\right)\right]}
\displaybreak[0]
 \nonumber \\ &~~~~+\frac{\bE_{B}\left[-\int_{\chi-\frac{B}{\hsig}}^{\infty}(z-\chi)\phi(z)dz+\int_{-\infty}^{-\frac{B}{\hsig}-\chi }(z+\chi)\phi(z)dz\right]}{\frac{\delta\sigma_z^2}{\hsig^3}+\bE_{B,Z}\left[\frac{B^2}{\hsig^3}\left(\bI\left(\Big|\frac{B}{\hsig}+Z\Big|<\chi \right)\right)\right]},
\end{align}
where Equality (a) is due to \eqref{equ:FinalDer}. In order to simplify (\ref{equ:DerTauSigma}), we use (\ref{equ:FixedPoint1}) again
\begin{align}\label{equ:FixedPointSimp}
\delta&=\frac{\sigma_z^2 \delta}{\hsig^2}+\bE_{B,Z}\left[\left(\eta\left(\frac{B}{\hsig}+Z;\chi \right)-\frac{B}{\hsig}\right)^2\right] \nonumber \\
&= \frac{\sigma_z^2 \delta}{\hsig^2}+\bE_{B,Z}\left[\frac{B^2}{\hsig^2}\bI\left(\left|\frac{B}{\hsig}+Z\right|<\chi \right)\right] \nonumber \\ &~~~~ +\bE_{B}\left[\int_{\chi-\frac{B}{\hsig}}^{\infty}(z-\chi )^2\phi(z)dz \right] +\bE_B \left[\int_{-\infty}^{-\chi-\frac{B}{\sigma}}(z+\chi)^2\phi(z)dz\right] 
\nonumber \\ 
&= \frac{\sigma_z^2 \delta}{\hsig^2}+\bE_{B,Z}\left[\frac{B^2}{\hsig^2}\bI\left(\left|\frac{B}{\hsig}+Z\right|<\chi \right)\right] \nonumber \\
&~~~~+\beta \bE_{B}\left[-\int_{\chi-\frac{B}{\hsig}}^{\infty}(z-\chi)\phi(z)dz+\int_{-\infty}^{-\frac{B}{\hsig}-\chi}(z+\chi)\phi(z)dz\right] \nonumber \\&~~~~+\bE_{B}\left[\int_{\chi-\frac{B}{\hsig}}^{\infty}z(z-\beta)\phi(z)dz+\int_{-\infty}^{-\frac{B}{\hsig}-\chi}z(z+\chi)\phi(z)dz\right].
\end{align}
Thus, we can rewrite (\ref{equ:DerTauSigma}) as
\allowdisplaybreaks
\begin{align}
\frac{d(\chi\hsig)}{d\chi}&=\frac{\delta-\bE_{B}\left[\int_{\chi-\frac{B}{\hsig}}^{\infty}z(z-\chi)\phi(z)dz+\int_{-\infty}^{-\frac{B}{\hsig}-\chi}z(z+\chi)\phi(z)dz\right]}{\frac{\delta\sigma_z^2}{\hsig^3}+\bE_{B,Z}\left[\frac{B^2}{\hsig^3}\left(\bI\left(\Big|\frac{B}{\hsig}+Z\Big|<\chi \right)\right)\right]} \nonumber \\ &=  
\frac{\delta-\bE_{B}\left[\int_{\chi -\frac{B}{\hsig}}^{\infty}z^2\phi(z)dz+\int_{-\infty}^{-\chi -\frac{B}{\hsig}}z^2\phi(z)dz\right]}{\frac{\delta\sigma_z^2}{\hsig^3}+\bE_{B,Z}\left[\frac{B^2}{\hsig^3}\left(\bI\left(\Big|\frac{B}{\hsig}+Z\Big|<\chi \right)\right)\right]} \nonumber \\ &~~~~+\frac{-\chi \bE_B \left[-\int_{\chi-\frac{B}{\hsig}}^{\infty}z\phi(z)dz+\int_{-\infty}^{-\chi-\frac{B}{\hsig}}z\phi(z)dz\right]}{\frac{\delta\sigma_z^2}{\hsig^3}+\bE_{B,Z}\left[\frac{B^2}{\hsig^3}\left(\bI\left(\Big|\frac{B}{\hsig}+Z\Big|<\chi \right)\right)\right]} \nonumber \\& = \frac{\delta-\bE_B \left[\left(\chi-\frac{B}{\hsig}\right)\phi\left(\chi-\frac{B}{\hsig}\right)+\left(\chi+\frac{B}{\hsig}\right)\phi\left(\chi +\frac{B}{\hsig}\right)\right]}{\frac{\delta\sigma_z^2}{\hsig^3}+\bE_{B,Z}\left[\frac{B^2}{\hsig^3}\left(\bI\left(\Big|\frac{B}{\hsig}+Z\Big|<\chi \right)\right)\right]} \nonumber \\ &~~~~+\frac{\bE_B \left[\int_{\chi -\frac{B}{\hsig}}^{\infty}\phi(z)dz+\int_{-\infty}^{-\chi-\frac{B}{\hsig}}\phi(z)dz\right]}{\frac{\delta\sigma_z^2}{\hsig^3}+\bE_{B,Z}\left[\frac{B^2}{\hsig^3}\left(\bI\left(\Big|\frac{B}{\hsig}+Z\Big|<\chi \right)\right)\right]} \nonumber \\&~~~~+\frac{\chi \bE_{B}\left[\phi\left(\chi-\frac{B}{\hsig}\right)+\phi\left(\chi+\frac{B}{\hsig}\right)\right]}{\frac{\delta\sigma_z^2}{\hsig^3}+\bE_{B,Z}\left[\frac{B^2}{\hsig^3}\left(\bI\left(\Big|\frac{B}{\hsig}+Z\Big|<\chi \right)\right)\right]} \nonumber \\ &=\overbrace{\frac{\delta-\bE_{B,Z}\left[\bI \left(\left|\frac{B}{\hsig}+Z\right|>\chi \right)\right]}{\frac{\delta\sigma_z^2}{\hsig^3}+\bE_{B,Z}\left[\frac{B^2}{\hsig^3}\left(\bI\left(\Big|\frac{B}{\hsig}+Z\Big|<\chi \right)\right)\right]}}^{\Delta_1}\nonumber \\ &~~~~+ \overbrace{\frac{\bE_{B}\left[\frac{B}{\hsig}\left(\phi\left(\chi-\frac{B}{\hsig}\right)-\phi\left(\chi+\frac{B}{\hsig}\right)\right)\right] }{\frac{\delta\sigma_z^2}{\hsig^3}+\bE_{B,Z}\left[\frac{B^2}{\hsig^3}\left(\bI\left(\Big|\frac{B}{\hsig}+Z\Big|<\chi \right)\right)\right]}}^{\Delta_2}.
\end{align}
$\Delta_1>0$ due to the fact that $\lambda>0$ in (\ref{equ:FixedPoint1}). Furthermore, If $B<0~ \big( \text{or}~ B>0 \big)$, then $\phi\left(\chi-\frac{B}{\hsig}\right)<\phi\left(\chi +\frac{B}{\hsig}\right)~\Big( \text{or}~\phi\left(\chi-\frac{B}{\hsig}\right)>\phi\left(\chi +\frac{B}{\hsig}\right)\Big)$.
Therefore $\Delta_2>0$ and hence, $\frac{d\lambda}{d\chi}>0$.
Applying \newline $\frac{d}{d\chi}\bP\left(\left|\frac{B}{\hsig}+Z\right| >0 \right)<0$ and $\frac{d\lambda}{d\chi}>0$ into (\ref{equ:ChainRule}) completes the proof. Also, combining Theorem \ref{thm:lassodetstate} and the fact that $\lambda$ is positive in \eqref{eq:fixedpoint21} results in $\lim_{p \rightarrow \infty} \frac{1}{p} \sum_i \mathbb{I} \left(\beta^{\lambda}_i(p) \neq 0 \right) \leq \delta.$
\end{proof}

\subsection{Proof of Theorem \ref{thm:quasiconvex}} \label{sec:proofquasiconvex}

First note that, according to Theorem \ref{thm:eqpseudolip}, we have $\lim_{p \rightarrow \infty} \frac{1}{p} \|\hat{\beta}^{\lambda}(p)- \beta_o \|_2^2 = \mathbb{E} _{B,Z}\left[(\eta(B + \hsig Z; \chi \hsig) - B)^2\right]$, where $(\chi, \hsig, \lambda)$ satisfies
\begin{align}\label{equ:FixedPoint}
&\hat{\sigma}^2=\sigma_z^2+\frac{1}{\delta}\bE_{B,Z}\left[(\eta(B+\hat{\sigma}Z; \chi \hat{\sigma})-B)^2\right], \nonumber \\
& \lambda=\chi \hat{\sigma}\left(1-\frac{1}{\delta}\bP(|B+\hat{\sigma}Z|>\chi \hat{\sigma})\right).
\end{align}
Clearly, the quasi-convexity of $\mathbb{E}\left[(\eta(B+ \hsig W; \chi \hsig) - B)^2\right]$ in terms of $\lambda$ is equivalent to the quasi-convexity of $\hsig^2$ in terms of $\lambda$ ( see Lemma \ref{lem:quascvx} for more information). Therefore, in the rest of the proof, our goal is to prove that $\hsig^2$ is a quasi-convex function of $\lambda$. 

$\hsig^2$ is a differentiable function of $\lambda$. Therefore, according to Theorem \ref{thm:quasiconvexnessuf}, $\hsig^2$ is a quasi-convex function of $\lambda$ if and only if $\frac{d\hsig^2}{d\lambda}$ has at most one sign change from negative to positive. We have $\frac{d\hsig^2}{d\lambda}=\frac{d\hsig^2}{d\chi}\frac{d\chi}{d\lambda}$. Hence, the rest of the proof involves two main steps: (i) $\frac{d \chi}{d \lambda} >0$. We have proved that this is true in Lemma \ref{lem:lambdataumon}. (ii) $\frac{d\hsig^2}{d\chi}$ has exactly one sign change from negative to positive. According to (\ref{equ:FinalDer}) we have 
\begin{align}
\frac{d\hsig^2}{d\chi }=\frac{\frac{\partial}{\partial \chi} \bE_{B,Z} \left[\left(\eta \left (\frac{B}{\hsig}+Z;\chi \right)-\frac{B}{\hsig} \right) ^2\right]}{\frac{\delta\sigma_z^2}{\hsig^3}+\bE_{B,Z}\left[\frac{B^2}{\hsig^3}\left(\bI\left(\Big|\frac{B}{\hsig}+Z\Big|<\chi \right)\right)\right]}.
\end{align}
As a result, the sign of $\frac{d\hsig^2}{d\chi}$ is the same as sign of $\frac{\partial}{\partial \chi } \bE \left[\left(\eta \left (\frac{B}{\hsig}+Z;\chi \right)-\frac{B}{\hsig} \right) ^2\right]$. However, this term is the derivative of the risk of the soft thresholding function with respect to its second parameter. This is what we proved above in Lemma \ref{lem:quasiconvexsoft}. This, combined with the fact that $d\chi/d \lambda <0$, proves that $ d \hat{\sigma}^2/ d \lambda$ has exactly one sign change.

\subsection{Proof of Theorem \ref{thm:stepwiseoptimal}} \label{proof:greedisgood}
We start with the following simple but useful lemma that will be used in our proofs multiple times.
\vspace{-.2cm}
\begin{lemma}\label{thm:SURE} 
\citep{St81Es}
If $g: \mathbb{R} \rightarrow \mathbb{R}$ is a weakly differentiable function and $W \sim N(0,1)$, then 
\begin{align}
\mathbb{E} (g(B_o+ \sigma W; \tau) -B_o)^2 &= \mathbb{E}  (g(B_o+ \sigma W; \tau) -B_o- \sigma W)^2 +  \sigma^2 \nonumber \\  &~~~~+ 2\sigma^2 \mathbb{E}  (g'(B_o+ \sigma W; \tau) -1) \nonumber
\end{align}
\end{lemma}

We call this result Stein's lemma in this paper.  We first prove one of the main features of the risk function defined in \eqref{equ:BayesRisk}
\vspace{-.2cm}
 \begin{lemma} \label{lem:inftau}
 If $\mathbb{P}(B \neq 0) \neq 0$, then $\inf_{\tau} R_{B} (\sigma, \tau; p_\beta) $ is an increasing function of $\sigma$. 
 \end{lemma}
 \vspace{-.4cm}
 
 \begin{proof}
 First, we prove this Lemma for the modified risk function defined as $\bar{R}_B(\sigma, \chi;p_\beta) \triangleq \bE_{W,B_o}\Big(\big(\eta(B_o+\sigma W;\sigma \chi)-B_o\big)^2\Big)$. Then, we will switch to the original risk function $R_B(\sigma, \tau;p_\beta)$. 

According to Lemma \ref{lem:convaceRisk}, the risk function $\bar{R}_B(\sigma, \chi;p_{\beta})$ is a concave function of $\sigma^2$. Therefore, the derivative of the risk function is a decreasing function itself. This means that if we prove that the derivative is positive when $\sigma \rightarrow \infty$, then for any fixed $\chi$ the risk function is an increasing function in terms of $\sigma$. Hence, we first prove that $\lim_{\sigma \rightarrow \infty} \frac{\partial \bar{R}_{B}(\sigma, \chi; p_\beta)}{\partial \sigma^2} \geq 0$.

\begin{align} \label{eq:DerSigma}
&\frac{\partial \bar{R}_{B}(\sigma, \chi; p_\beta)}{\partial \sigma^2} \nonumber \\
&=\frac{1}{2\sigma}\frac{\partial}{\partial \sigma} \bE_{W,B_o}\left(\left(\eta(B_o+\sigma W;\sigma \chi)-B_o\right)^2\right) \nonumber \\
\displaybreak[0]
&= \frac{1}{2\sigma}\bE_{W,B_o} \Big(2\left(\eta(B_o+\sigma W;\sigma \chi)-B_o\right)((W-\chi)\bI\left(B_o+\sigma W>\sigma \chi \right) \nonumber \\
\displaybreak[0]
&\quad\quad\quad+(W+\chi)\bI\left(B_o+\sigma W<-\sigma \chi \right))\Big) \nonumber \\
&= \bE_{W,B_o} \Big( ( W- \chi)^2 \bI(B_o+\sigma W>\sigma \chi)\nonumber \\
&\quad\quad\quad+( W+ \chi)^2 \bI(B_o+\sigma W<-\sigma \chi)\Big) \nonumber \\
\displaybreak[0]
&\overset{(a)}{=}\bE_{W,B_o} \Big( \sigma W \delta(B_o+\sigma W- \sigma \chi)+ \bI(B_o+\sigma W> \sigma \chi) \nonumber \\
&\quad\quad\quad\quad-2\sigma\chi  \delta(B_o+\sigma W-\sigma \chi) +\tau^2 \bI(B_o+\sigma W> \sigma chi)\nonumber \\
&\quad\quad\quad\quad  - \sigma W  \delta(B_o+\sigma W+\sigma \chi)+  \bI(B_o+\sigma W< -\sigma \chi) \nonumber \\
\displaybreak[0]
&\quad\quad\quad\quad + \chi^2  \bI(B_o+\sigma W<-\sigma \chi)-2\sigma \chi  \delta(B_o+\sigma W+\sigma \chi) \Big) \nonumber \\
&= \bE_{B_o} \left( \left(\frac{\sigma \chi - B_o}{\sigma}\right) \phi\left(\frac{\sigma \chi -B_o}{\sigma}\right)+ \left(\frac{\sigma \chi + B_o}{\sigma}\right) \phi\left(\frac{\sigma \chi+B_o}{\sigma}\right)\right) \nonumber \\
&\quad+ \bE_{B_o,W} \left((1+\chi^2)(\bI(B_o+\sigma W>\sigma \chi)+\bI(B_o+\sigma W<-\sigma \chi)\right) \nonumber \\
&\quad -\bE_{B_o}\left(2 \chi \left(\phi\left(\frac{\sigma \chi-B_o}{\sigma}\right)+\phi\left(\frac{\sigma \chi+B_o}{\sigma}\right)\right)\right) \nonumber \\
\displaybreak[0]
&=(1+\chi^2) \bE_{B_o} \left(\Phi\left(\frac{B_o}{\sigma}-\chi\right)+\Phi\left(\frac{-B_o}{\sigma}-\chi\right)\right) \nonumber \\
&\quad -\chi \bE_{B_o}  \left(\phi\left(\frac{B_o}{\sigma}-\chi \right)+\phi\left(\frac{B_o}{\sigma}+\chi \right)\right) \nonumber \\
&\quad -\frac{1}{\sigma}  \bE_{B_o} \left( B_o \left(\phi\left(\frac{B_o}{\sigma}-\chi \right)-\phi\left(\frac{B_o}{\sigma}+\chi \right)\right)\right),
\end{align}
where (a) holds because of the Stein's lemma (Lemma \ref{thm:SURE}). As a result of \eqref{eq:DerSigma}, we can write
\begin{align}
\displaybreak[0]
&\lim_{\sigma \rightarrow \infty} \frac{\partial R_B(\sigma, \chi ; p_{\beta})}{\partial \sigma^2} \nonumber \\
&\displaybreak[0]= \lim_{\sigma \rightarrow \infty} \Bigg((1+\chi^2) \bE_{B_o} \left(\Phi\left(\frac{B_o}{\sigma}-\chi \right)+\Phi\left(\frac{-B_o}{\sigma}-\chi \right)\right) \nonumber \\
&\quad \quad\quad-\chi \bE_{B_o}  \left(\phi\left(\frac{B_o}{\sigma}-\chi \right)+\phi\left(\frac{B_o}{\sigma}+\chi \right)\right) -\frac{1}{\sigma}  \bE_{B_o} \left( B_o \left(\phi\left(\frac{B_o}{\sigma}-\chi \right)-\phi\left(\frac{B_o}{\sigma}+\chi \right)\right)\right) \Bigg) \nonumber \\
&= 2(1+\chi ^2) \Phi(-\chi)-2 \chi  \phi(\chi) \overset{(b)}{=} 2(1+\chi^2) Q(\chi)-2 \chi  \phi(\chi) \overset{(c)}{>} 2 \chi  \phi(\chi)-2 \chi  \phi(\chi) =0,
\end{align}
where in (b), $Q(\chi) \triangleq \int_{\chi}^{\infty}\phi(w)dw$ and in (c) we have used the well-known lower-bound for the Q-function $\left(\frac{\chi}{1+\chi^2}\right)\phi(\chi)<Q(\chi)$.
Now, by using the fact that \newline $\bar{R}_B(\sigma, \chi;p_\beta)=\bE_{W,B_o}\left(\left(\eta(B_o+\sigma W;\sigma \chi)-B_o\right)^2\right)$ is an increasing function of $\sigma$ for any fixed $\chi$, if $\sigma_1<\sigma_2$, then we can write 
\begin{align}\label{eq:ineq}
\bE_{W,B_o} \left(\eta \left(B_o+\sigma_1 W;\sigma_1 \chi\right)-B_o\right)^2 <\bE_{W,B_o}\left(\eta(B_o+\sigma_2 W;\sigma_2 \chi)-B_o\right)^2.
\end{align}
We can take the infimum from both sides of the inequality in \eqref{eq:ineq} and obtain
\begin{align}\label{eq:inf}
&\inf_{\chi} \bE_{W,B_o} \left(\eta \left(B_o+\sigma_1 W;\sigma_1 \chi\right)-B_o\right)^2  < \inf_{\chi} \bE_{W,B_o} \left(\eta(B_o+\sigma_2 W;\sigma_2 \chi)-B_o\right)^2.
\end{align}
Note that in general when we take infimum from both sides the $<$ may change to $\leq$. However, since $\mathbb{P} (B \neq 0) \neq 0$, according to Lemma \ref{lem:quasiconvexsoft}, $\bE_{W,B_o} \left(\eta \left(B_o+\sigma_2 W;\sigma_2 \chi\right)-B_o\right)^2$ is a bowl shaped function of $\chi$ and hence $\inf_{\chi} \bE_{W,B_o} \left(\eta \left(B_o+\sigma_2 W;\sigma_2 \chi\right)-B_o\right)^2$ is achieved at a finite value of $\bar{\chi}_2$. According to \eqref{eq:ineq}, $\bE_{W,B_o} \left(\eta \left(B_o+\sigma_1 W;\sigma_1 \bar{\chi}_2\right)-B_o\right)^2$ is strictly smaller than \newline $\inf_{\chi} \bE_{W,B_o} \left(\eta \left(B_o+\sigma_2 W;\sigma_2 \chi\right)-B_o\right)^2$, and hence \eqref{eq:inf} is also correct with strict inequality. 
Let $\tau=\sigma \chi$ and suppose that $\tau^*=\sigma_2 \chi^*$ is the threshold by which the infimum of the right hand side (RHS) of \eqref{eq:inf} is achieved. If we apply $\chi^*$ to the both sides of \eqref{eq:ineq}, then we can write \eqref{eq:ineq}
\begin{align}\label{eq:ineq2}
&\inf_{\tau} \bE_{W,B_o}\left(\left(\eta(B_o+\sigma_1 W;\tau)-B_o\right)^2\right) \nonumber \\
&=\inf_{\chi} \bE_{W,B_o}\left(\left(\eta(B_o+\sigma_1 W;\chi \sigma_1)-B_o\right)^2\right) \leq\bE_{W,B_o}\left(\left(\eta \left(B_o+\sigma_1 W;\left(\frac{\tau^*}{\sigma_2}\right)\sigma_1\right)-B_o\right)^2\right) \nonumber \\
&<\bE_{W,B_o}\left(\left(\eta(B_o+\sigma_2 W;\tau^*)-B_o\right)^2\right) = \inf_{\tau} \bE_{W,B_o}\left(\left(\eta(B_o+\sigma_2 W;\tau)-B_o\right)^2\right), 
\end{align}
which means that $\inf_{\tau}R_B(\sigma, \tau;p_\beta)$ is an increasing function of $\sigma$.
 \end{proof}
 

Having completed the proof of Lemma \ref{lem:inftau}, the proof of Theorem \ref{thm:stepwiseoptimal}  is performed using an induction argument. Suppose that $\tau^{*,1}, \tau^{*, 2}, \ldots, \tau^{*, T-1}$ is optimal for iteration $T$. Our goal is to show that $\tau^{*,1}, \tau^{*, 2}, \ldots, \tau^{*, T-2}$ is optimal for iteration $T-1$ as well. Now we use a contradiction argument. Suppose that $\tau^{*,1}, \tau^{*, 2}, \ldots, \tau^{*, T-2}$ is not optimal for iteration $T-1$; then there exists $\tau^{1}, \tau^2, \ldots, \tau^{T-2}$ such that 
\[
\sigma^{T-1} (\tau^1, \ldots, \tau^{T-2}) < \sigma^{T-1}(\tau^{*,1}, \tau^{*, 2}, \ldots, \tau^{*, T-2}). 
\]
Define $\tau^{**, T-1}$ as
\[ 
\arg \min_{\tau}  \mathbb{E}_{B_o, W} \left[(\eta(B_o + \sigma^{T-1} (\tau^1, \ldots, \tau^{T-2})  W; \tau) -B_o)^2\right].
\]
We can now prove that 
\[
\sigma^{T}(\tau^1, \ldots, \tau^{T-2}, \tau^{**, T-1}) < \sigma^T (\tau^{*,1}, \tau^{*,2}, \ldots, \tau^{*,T-1}).
\]
From Theorem \ref{thm:ampeqpseudo_lip} we have
\begin{equation}\label{eq:optimalstepwise}
(\sigma^{t+1})^2 = \sigma_{\omega}^2+\frac{1}{\delta} \mathbb{E}_{B_o, W} \left[(\eta(B_o + \sigma^t W; \tau^t ) -B_o)^2\right].
\end{equation}
Since, $\sigma^{T-1} (\tau^1, \ldots, \tau^{T-2}) < \sigma^{T-1}(\tau^{*,1}, \tau^{*, 2}, \ldots, \tau^{*, T-2})$, Lemma \ref{lem:inftau} combined with  \eqref{eq:optimalstepwise} prove that 
\[
\sigma^{T}(\tau^1, \ldots, \tau^{T-2},  \tau^{**, T-1}) < \sigma^T (\tau^{*,1}, \tau^{*,2}, \ldots, \tau^{*,T-1}).
\]
that contradicts the optimality of $\tau^{*,1}, \tau^{*,2}, \ldots, \tau^{*,T-1}$. Therefore, we conclude that if $\tau^{*,1}, \tau^{*, 2}, \ldots,$ $ \tau^{*, T-1}$ is optimal for iteration $T$, then it is optimal for every iteration $t<T$. The rest of the induction argument is similar and hence for the sake of brevity we skip it.

\subsection{Proof of Theorem \ref{cor:empconvopt}} \label{sec:proofuniform}

\subsubsection{Roadmap of the proof} \label{sec:proofroadmapconsistencyamp}
We break the rather long proof of this theorem into two main steps:
\begin{enumerate}
\item First, we prove that the risk estimate presented in  \eqref{eq:empriskdefamp} provides a consistent estimate of the risk $R_B(\sigma^t, \tau^; p_\beta)$. Since we would like to optimize the risk estimate over the parameter $\tau^t$, we require a uniform notion of consistency, i.e., $$\lim_{h \rightarrow 0} \lim_{p\rightarrow \infty}  \sup_{\tau^t \in \mathcal{T}} \left| \hat{R}_{h,p}^{t}(\tau^t) - R_B(\sigma^t, \tau^t; p_\beta)\right| = 0$$ in probability. Note that the convergence is uniform on $\mathcal{T}^t$. After discussing several useful lemmas we prove this claim in Theorem \ref{thm:amptune1}.  \\

\item Once we prove this claim we use the properties of the solution path of AMP, in particular Theorem \ref{thm:stepwiseoptimal}, to show the consistency of our parameter tuning scheme across $t$ iterations. 
\end{enumerate}

\subsubsection{Uniform convergence of the risk estimate}

We start with a few lemmas that will be later used to prove
$$\lim_{h \rightarrow 0} \lim_{p\rightarrow \infty}  \sup_{\tau^t \in \mathcal{T}} \left| \hat{R}_{h,p}^{t}(\tau^t) - R_B(\sigma^t, \tau^t; p_\beta)\right| = 0$$ in probability.

 Our first lemma is concerned with the pointwise (with respect to $\tau^t$) convergence of the risk estimate to $R_B(\sigma^t, \tau^t; p_\beta)$. 

\begin{lemma}\label{thm:amptune} Consider a converging sequence $\{\beta_o(p), X(p), w(p)\}$, and let the elements of $X$ be drawn iid from $N(0,1/n)$.Furthermore,  let ${\hat{R}_{h,p}^t(\tau^t)}$ denote the estimate of the risk at iteration $t$ of AMP as defined in \eqref{eq:empriskdefamp}. Then,
\begin{align}\label{eq:MSEAMP22}
\lim_{p\rightarrow \infty} {\hat{R}_{h,p}^{t}(\tau^t, \tau^{t-1}, \ldots, \tau^1}) = \bE_{B_o,W}\left[(\tilde{\eta}_h(B_o+\sigma^tW;\tau^t)-B_o)^2\right],
\end{align}
almost surely, where $B_o$ and $W$ are two random variables with distributions $p_\beta$ and $N(0,1)$, respectively and $\sigma^t$ satisfies \eqref{eq:ampevolution}.
\end{lemma}
\begin{proof}
By applying Stein's lemma (Lemma \ref{thm:SURE}) to the right hand side of \eqref{eq:MSEAMP22}, we can rewrite it as
\begin{align}\label{eq:lemsteinrisk}
&\bE_{B_o,W}\left[(\tilde{\eta}_h(B_o+\sigma^tW;\tau^t)-B_o)^2\right] \nonumber \\
&~~~=\bE_{B_o,W}\left[(\tilde{\eta}_h(B_o+\sigma^tW;\tau^t)-(B_o+\sigma^tW))^2\right]+\left(\sigma^t\right)^2\nonumber \\
&~~~~+2\left(\sigma^t\right)^2\bE_{B_o,W}\left[(\tilde{\eta}'_h(B_o+\sigma^tW;\tau^t)-1)\right].
\end{align}
Similarly, we can decompose the left hand side (LHS) of \eqref{eq:MSEAMP22} to
\begin{align}
{\hat{R^t}(\tau^t , \tau^{t-1}, \ldots, \tau^1)}&=\frac{1}{p} \|\tilde{\eta}_h(\beta^t+X^*z^t;\tau^t)-(\beta^t+X^*z^t)\|_2^2+\left(\sigma^t\right)^2\nonumber \\
&~~~~+\frac{1}{p}2\left(\sigma^t\right)^2\left[ \mathbf{1}^T(\tilde{\eta}_h'(\beta^t+X^*z^t;\tau^t)-\mathbf{1})\right].
\end{align}
Let $X_{(:, i)}$ denote the $i^{\rm th}$ column of the matrix $X$. Define:
\begin{align}\label{eq:defbt}
b^t \triangleq \beta^t + X^* z^t-\beta_o.
\end{align}
Considering the following function
\begin{align}\label{eq:psi1}
\psi_1(b_i^{t},\beta_{o,i})&\triangleq \left(\tilde{\eta}_h(b_i^t +\beta_{o,i}  ;\tau^{t})-(b_i^t+ \beta_{o,i})\right)^2 =  \left(\tilde{\eta}_h(\beta_i^{t}+X_{(:,i)}^*z^{t};\tau^{t})-(\beta_i^{t}+X_{(:,i)}^*z^{t})\right)^2
\end{align}
It is straightforward to use Lemma 1 of \cite{BaMo10} to prove
\begin{align}\label{eq:psi1Asymp}
\lim_{p\rightarrow \infty}\frac{1}{p}\sum_{i=1}^p \psi_1(b_i^{t},\beta_{o,i})&=\bE\left[\psi_1(B_o+\sigma^tW;\tau^t),B_o)\right] \nonumber \\
&=\bE_{B_o,W}\left[(\tilde{\eta}_h(B_o+\sigma^tW;\tau^t)-(B_o+\sigma^tW))^2\right],
\end{align}
almost surely.  Furthermore, it is straightforward to note that the derivative of $\tilde{\eta}_h$ is bounded and hence we can employ the mean value theorem to show that $\tilde{\eta}_h$ is also a Lipschitz function and hence  by Lemma 1 of \cite{BaMo10} almost surely
\begin{align}\label{eq:psi2Asymp}
\lim_{p \rightarrow \infty} \frac{\mathbf{1}^T(\tilde{\eta}_h'(\beta^t+X^*z^t;\tau^t)-\mathbf{1}) }{p} = \mathbb{E} (\tilde{\eta}'_h(B_o+ \sigma^t W; \tau^t)-1).
\end{align}
Combining \eqref{eq:lemsteinrisk}, \eqref{eq:psi1Asymp} and \eqref{eq:psi2Asymp} finishes the proof. 
\end{proof}

Lemma \ref{thm:amptune} is only concerned with the pointwise convergence of the risk. The next theorem proves the uniform convergence. Define
\begin{equation}
R_A^t(\tau^t, \tau^{t-1}, \ldots, \tau^1) \triangleq R_B(\tau^t , \sigma^t; p_\beta), 
\end{equation}
where $\sigma^t$ is derived from the iterations of \eqref{eq:ampevolution}. Even though this new notation seems redundant, it will be useful in our proofs.

\begin{theorem}\label{thm:amptune1} Consider a converging sequence $\{\beta_o(p), X(p), w(p)\}$, and let the elements of $X$ be drawn iid from $N(0,1/n)$. Furthermore, let $\hat{R}_{h,p}^t(\tau^t, \tau^{t-1}, \ldots, \tau^1 )$ denote the estimate of the Bayes risk at iteration $t$ of AMP as defined in \eqref{eq:empriskdefamp}. Let $\mathcal{T}^t \subset \mathbb{R}$ denote a compact set. Then,
\begin{align}\label{eq:MSEAMP}
\lim_{h \rightarrow 0} \lim_{p\rightarrow \infty}  \sup_{\tau^t \in \mathcal{T}} \left| \hat{R}_{h,p}^{t}(\tau^t, \tau^{t-1}, \ldots, \tau^1) - R_A^t(\tau^t, \tau^{t-1}, \ldots, \tau^1)\right| = 0
\end{align}
 in probability, for every $\tau^1, \ldots, \tau^{t-1} \in \mathcal{T}^1 \times \ldots \times \mathcal{T}^{t-1}$.
\end{theorem}

\begin{proof}
We first define the the following function:
\[
R^t_{A,h}(\tau^t, \tau^{t-1}, \ldots, \tau^1) \triangleq \mathbb{E} (\tilde{\eta}_h(B+ \sigma^t W; \tau^t) -B)^2,
\]
where $B$ and $W$ are two independent random variables with distributions $p_\beta$ and $N(0,1)$ respectively, and $\sigma^t$ satisfies \eqref{eq:ampevolution}. Note that this is the asymptotic risk of AMP for the smoothed version of the soft thresholding function. By the triangle inequality we have
\begin{eqnarray}\label{eq:comdinedtermsupconv}
 \lefteqn{\left| \hat{R}_{h,p}^{t}(\tau^t, \tau^{t-1}, \ldots, \tau^1) - R_A^t(\tau^t, \tau^{t-1},\ldots,  \tau^1\right|} \nonumber \\
  & \leq  \left| \hat{R}_{h,p}^{t}(\tau^t, \tau^{t-1}, \ldots, \tau^1) - R_{A,h}^t ( \tau^t, \tau^{t-1}, \ldots, \tau^1)\right| \nonumber  \\&~~ ~+  \left| R_{A,h}^t ( \tau^t, \tau^{t-1}, \ldots, \tau^1) - R_{A}^t ( \tau^t, \tau^{t-1}, \ldots, \tau^1) \right|. 
\end{eqnarray}

Hence, we first prove that 
\begin{equation}\label{eq:asymptotfirststep}
 \lim_{p\rightarrow \infty}  \sup_{\tau^t \in \mathcal{T}^t} \left| \hat{R}_{h,p}^{t}(\tau^t, \tau^{t-1}, \ldots, \tau^1) - R_{A,h}^t ( \tau^t, \tau^{t-1}, \ldots, \tau^1)\right|=0, 
\end{equation}
in probability for every $h>0$ and every $\tau^1, \ldots, \tau^{t-1}$.  Secondly, we prove that 
\begin{equation} \label{eq:asymptotsecondtstep}
\lim_{h \rightarrow 0} \sup_{\tau \in \mathcal{T}^t}  \left| R_{A,h}^t ( \tau^t, \tau^{t-1}, \ldots, \tau^1) - R_{A}^t ( \tau^t, \tau^{t-1}, \ldots, \tau^1) \right|=0.
\end{equation}
Combining these two results will establish the theorem. To establish \eqref{eq:asymptotfirststep}, we start with the following definitions whose importance will be clear later: 
 \begin{eqnarray}
 U_{h,p} (b_i, \beta_{o,i} ,\tau, \sigma) & \triangleq & (\tilde{\eta}_h(b_i + \beta_{o,i};\tau)-(b_i+ \beta_{o,i}))^2+ \sigma^2+2 \sigma^2\left[(\tilde{\eta}_h'(b_i +\beta_{o,i} ;\tau)-1)\right], \\
\bar{U}(b_i, \beta_{o,i}, \rho, \tau, \sigma) &\triangleq& \sup_{\tilde{\tau} : |\tilde{\tau} - \tau| \leq \rho} U_{h,p} (b_i, \beta_{o,i} ,\tilde{\tau}, \sigma)-\mathbb{E} U_{h,p} (\sigma W, B ,\tilde{\tau}, \sigma). \nonumber
 \end{eqnarray}
where $B$ and $W$ are two independent random variables with distributions $p_\beta$ and $N(0,1)$ respectively. The following remarks will clarify some of the main features and connections of these definitions:
\begin{enumerate}
\item It is straightforward to verify that $$\hat{R}^t_{h,p} (\tau^t, \tau^{t-1}, \ldots, \tau^1)= \frac{1}{p} \sum_{i=1}^p U_{h,p} (b^t_i, \beta_{o,i} ,\tau^t, \sigma^t),$$
where $b^t = \beta^t + X^* z^t - \beta_o$ and $\beta^t$ is the estimate of AMP with threshold parameters $\tau^i$ at the $i^{\rm th}$ iteration.  

\item According to Lemma \ref{thm:amptune}, $$\frac{1}{p} \sum_{i=1}^p U_{h,p} (b_i^t, \beta_{o,i} ,\tau^t, \sigma^t) \overset{a.s.}{\rightarrow} R_{A,h}^t(\tau^t, \tau^{t-1}, \ldots, \tau^1).
$$

\item According to Lemma \ref{thm:SURE}, $R_{A,h}(\tau^t, \tau^{t-1}, \ldots, \tau^1) = \mathbb{E} U_{h,p} (\sigma^t W, B ,\tau, \sigma^t)$, where the expectation is with respect to two independent random variables $W \sim N(0,1)$ and $B \sim p_\beta$. 

\end{enumerate}

The next four lemmas prove several basic properties of $R_{A, h}$, $U_{h,p}$ and $\bar{U}_{h,p}$ that will be useful later in our proof.

\begin{lemma}\label{lem:continuityRA}
$R_{A,h}^t (\tau^t, \tau^{t-1}, \ldots, \tau^1)$ is a continuous function of $\tau^t$, for every $\tau^1, \tau^2, \ldots, \tau^{t-1} \in \mathcal{T}^1 \times \mathcal{T}^2, \ldots, \mathcal{T}^{t-1}$.  
\end{lemma}
\begin{proof}
The proof is a straightforward application of the dominated convergence theorem. 
\begin{eqnarray*}
\lim_{\tilde{\tau}^t \rightarrow \tau^t} R_{A,h}^t (\tilde{\tau}^t, \tau^{t-1}, \ldots, \tau^1) &=& \lim_{\tilde{\tau}^t \rightarrow \tau^t} \mathbb{E} U_{h,p} (\sigma^t W, B ,\tilde{\tau}^t, \sigma^t) \nonumber \\
 &\overset{(a)}{=}&\mathbb{E}  \lim_{\tilde{\tau}^t \rightarrow \tau^t} U_{h,p} (\sigma^t W, B ,\tilde{\tau}^t, \sigma^t) \nonumber \\
 & \overset{(b)}{=}& \mathbb{E}  U_{h,p} (\sigma^t W, B ,{\tau}^t, \sigma^t).
\end{eqnarray*}
Equality (a) is due to the fact that $ U_{h,p} $ is a bounded function of both $\sigma^t W $ and $B$ and hence dominated convergence theorem can be applied. Equality (b) uses the continuity of $U_{h,p}$ with respect to $\tau^t$. 
\end{proof}

\begin{lemma}\label{lem:lipshtizfist1}
$U_{h,p} (b_i, \beta_{o,i} ,{\tau} , \sigma)$ is a Lipschitz function of $(b_i, \beta_{o,i})$ with Lipschitz constant
\[
L_U  \triangleq \sqrt{2} \max_{\tau \in \mathcal{T}}  2 \tau  (\sup_\zeta|\tilde{\eta}_h' (\zeta; \tau)|+1)+ 2 \sigma^2 \sup_{\tilde{\zeta}}| \tilde{\eta}_h''(\tilde{\zeta}; \tau )| ). 
\]
\end{lemma}
It is important to note that both $|\tilde{\eta}_h' (\zeta; \tau)|$ and $|\tilde{\eta}_h''(\tilde{\zeta}; \tau )|$ are bounded functions of $\zeta$ and $\tilde{\zeta}$ respectively for a fixed $\tau$. Since $\tau$ belongs to a compact set $L_U$ is bounded as well.   
\begin{proof}
Define $s_i \triangleq b_i + \beta_{o,i}$ and $\tilde{s}_i  \triangleq \tilde{b}_i + \tilde{\beta}_i$. 
\begin{eqnarray*}
\lefteqn{|U_{h,p} (b_i, \beta_{o,i} ,{\tau}, \sigma) - U_{h,p} (\tilde{b}_i, \tilde{\beta}_{o,i} ,{\tau}, \sigma )|} \nonumber \\
&=& | (\tilde{\eta}_h(s_i;\tau)-s_i)^2+2 \sigma^2\tilde{\eta}_h'(s_i ;\tau)-  (\tilde{\eta}_h(\tilde{s}_i;\tau)-\tilde{s}_i)^2+2 \sigma^2\tilde{\eta}_h'(\tilde{s}_i ;\tau) | \nonumber \\
&\overset{(a)}{=}&  |(\tilde{\eta}_h' (\zeta; \tau)+1) (s_i - \tilde{s}_i) (\tilde{\eta}_h(s_i;\tau)-s_i + \tilde{\eta}_h(\tilde{s}_i;\tau)-\tilde{s}_i) |+ 2 \sigma^2| \tilde{\eta}_h''(\tilde{\zeta}; \tau) ||(s_i - \tilde{s}_i)| \nonumber \\
&\overset{(b)}{\leq}&  2 \tau (\sup_\zeta |\tilde{\eta}_h' (\zeta; \tau)|+1)+ 2 \sigma^2( \sup_{\tilde{\zeta}}| \tilde{\eta}_h''(\tilde{\zeta}; \tau )|   ) |(s_i - \tilde{s}_i)|. 
\end{eqnarray*}
Note that Equality (a) is derived from the mean value theorem. To obtain Inequality (b) we used the fact that $|\eta_h(s, \tau) - s| \leq \tau$. Finally, to show that the function is Lipschitz  we employ the inequality $|s_i- \tilde{s}_i| \leq \sqrt{2} \sqrt{(b_i - \tilde{b}_i)^2 + (\beta_i^t - \tilde{\beta}_i)^2} $.
\end{proof}

\begin{lemma}\label{lem:lipschitzsupfunc}
$\bar{U}(b_i, \beta_{o,i}, \rho, \tau, \sigma) $ is also a Lipschitz function of $(b_i, \beta_{o,i})$ with Lipschitz constant $L_U$ defined in Lemma \ref{lem:lipshtizfist1}. 
\end{lemma}
\begin{proof}
From Lemma \ref{lem:lipshtizfist1} we have
\[
U_{h,p} (b_i, \beta_{o,i} ,\tilde{\tau}, \sigma) \leq  U_{h,p} (\tilde{b}_i, \tilde{\beta}_{o,i} ,\tilde {\tau}, \sigma) + L_U \sqrt{(b_i - \tilde{b}_i)^2 + (\beta_i^t - \tilde{\beta}_i)^2}. 
\]
By subtracting the constant (in terms of $b_i$ and $\beta_{o,i}$ ) $\mathbb{E} U_{h,p} (\sigma^t W, B ,\tilde{\tau}, \sigma)$ and taking the supremum with respect to $\tilde{\tau}$ we obtain 
\[
\bar{U}(b_i, \beta_{o,i}, \rho, \tau, \sigma)  \leq \bar{U}(\tilde{b}_i, \tilde{\beta}_{o,i}, \rho, \tau, \sigma) + L_U \sqrt{(b_i - \tilde{b}_i)^2 + (\beta_i^t - \tilde{\beta}_i)^2}.
\]
The proof of the other direction is similar. 
\end{proof}

\begin{lemma} \label{lem:mainlemmainasymptot} Let $W$ and $B$ denote two independent random variables with distributions $N(0,1)$ and $p_\beta$ respectively. Then,
$$\lim_{\rho \rightarrow 0} \mathbb{E} \bar{U} (\sigma W, B, \rho, \tau, \sigma) = 0. $$
\end{lemma}
\begin{proof}
Since $\bar{U}$ is a bounded function, we can exchange the order of $\lim_{\rho \rightarrow 0}$ and $\mathbb{E}$. Hence, 
\begin{eqnarray}
\lim_{\rho \rightarrow 0}  \mathbb{E} \bar{U} (\sigma W, B, \rho, \tau, \sigma) & =& \mathbb{E} \lim_{\rho \rightarrow 0} \bar{U} (\sigma W, B, \rho, \tau, \sigma)  \nonumber \\
&\overset{(a)}{=}& \mathbb{E} U_{h,p} (\sigma^t W, B ,{\tau}, \sigma) - \mathbb{E} U_{h,p} (\sigma^t \tilde{W}, \tilde{B} ,{\tau}, \sigma) =0.
\end{eqnarray}
Note that to obtain Equality (a) we use the continuity of $U_{h,p}$ in $\tau$. 
\end{proof}

Lemma \ref{lem:mainlemmainasymptot} implies that for any $\epsilon>0$, there exists $\rho_{\tau_0}$ such that if $|\tau- \tau_0| < \rho_{\tau_0}$, then $ \mathbb{E} \bar{U}(\sigma^t W, B, \rho, \tau) < \epsilon$. Note that we have a subscript $\tau_0$  for $\rho_{\tau_0}$ to emphasize on the fact that $\rho$ is dependent on the choice of $\tau_0$. Define
\[
\mathcal{B} (c, \rho) = \{ \tau \in \mathcal{T} \ | \ |\tau -c | \leq \rho \}. 
\]
 Consider the set of all the balls $\mathcal{B}(\tau, \rho_{\tau})$ for every $\tau \in \mathcal{T}$. This set forms a covering of $\mathcal{T}$. Since $\mathcal{T}$ is compact, it has a finite subcover. Let $\mathcal{B}(\tau_1^*, \rho_1^*)$, $\mathcal{B}(\tau_2^*, \rho_2^*)$, $\ldots, \mathcal{B}(\tau_M^*, \rho_M^*)$ denote this finite subcover. We have
\begin{eqnarray}
\lefteqn{\mathbb{P} \left( \sup_\tau  \frac{1}{p} \sum_{i=1}^p U_{h,p} (b_i^t, \beta_{o,i}, {\tau}, \sigma^t)  - R_{A,h} (\tau, \tau^{t-1}, \ldots, \tau^1) > 2 \epsilon \right)  } \nonumber \\
&\leq& \mathbb{P} \left( \max_i  \frac{1}{p} \sum_{i=1}^p \bar{U}_{h,p} (b_i^t, \beta_{o,i},\rho^*_i, {\tau}^*_i, \sigma^t)  > 2 \epsilon \right) \nonumber \\
& <& M \max_i \mathbb{P} \left( \frac{1}{p} \sum_{i=1}^p \bar{U}_{h,p} (b_i^t, \beta_{o,i}, \rho^*_i, {\tau}^*_i , \sigma^t)  > 2 \epsilon \right).
\end{eqnarray}
 Note that the first inequality is due to the definition of $\bar{U}_{h,p}$ and the second inequality is a simple application of the union bound. The last step of the proof is to show that
\begin{eqnarray}\label{eq:lastinuniform}
\mathbb{P} \left( \frac{1}{p} \sum_{i=1}^p \bar{U}_{h,p} (b_i^t, \beta_{o,i}, \rho^*_i, {\tau}^*_i , \sigma^t)  > 2 \epsilon \right) \rightarrow 0,
\end{eqnarray}
as $p \rightarrow \infty$. Note that if we combine Lemma \ref{lem:lipschitzsupfunc} and Lemma 1 of \cite{BaMo10} we obtain 
\begin{equation}\label{eq:probfromBayatipaper}
\mathbb{P} \left( \frac{1}{p} \sum_{i=1}^p \bar{U}_{h,p} (b_i^t, \beta_{o,i}, \rho^*_i, {\tau}^*_i , \sigma^t) - \mathbb{E}  \bar{U}_{h,p} (\sigma^t W, B, \rho^*_i, {\tau}^*_i, \sigma^t )   >  \epsilon \right) \rightarrow 0,
\end{equation}
 as $p \rightarrow \infty$. Furthermore, from the construction of the covering we have
\begin{eqnarray}\label{eq:boundexplast}
\max_{i=1, \ldots, M} \mathbb{E}  \bar{U}_{h,p} (\sigma^t W, B, \rho^*_i, {\tau}^*_i, \sigma^t ) < \epsilon. 
\end{eqnarray}
Hence by combining \eqref{eq:probfromBayatipaper} and \eqref{eq:boundexplast} we obtain \eqref{eq:lastinuniform}. 
\end{proof}


At this point we refer the reader to \eqref{eq:comdinedtermsupconv}. So, far we have been able to prove 
 \eqref{eq:asymptotfirststep}. Hence, if we prove \eqref{eq:asymptotsecondtstep}, it will establish \eqref{eq:comdinedtermsupconv} and will finish the proof of Theorem \ref{thm:amptune1}.
Hence, in this step we would like to prove that $\sup_{\tau^t \in \mathcal{T}^t}  \left| R_{A,h}^t ( \tau^t, \tau^{t-1}, \ldots, \tau^1) - R_{A}^t ( \tau^t, \tau^{t-1}, \ldots, \tau^1) \right|
$ is a continuous function of $h$. In Lemma \ref{lem:continuityRA} we showed that $R_{A,h}^t ( \tau^t, \tau^{t-1}, \ldots, \tau^1)$ is a continuous function of $ \tau^t$. It is straightforward to use the same argument to show that it is a continuous function of $(h, \tau^t)$. Hence, the function \newline $ \left| R_{A,h}^t ( \tau^t, \tau^{t-1}, \ldots, \tau^1) - R_{A}^t ( \tau^t, \tau^{t-1}, \ldots, \tau^1) \right|$ is also  a continuous function of $(h, \tau)$. We require the following standard result from analysis.  

\begin{lemma}\label{lem:continuitytwovariables}
Let $f(h, \tau)$ denote a continuous function from $\mathbb{R}^2$ to $\mathbb{R}$. Also, assume that $\mathcal{T}$ is a compact subset of $\mathbb{R}$. Then,
\[
\lim_{h \rightarrow h_o} \sup_{\tau \in \mathcal{T}} f(h, \tau) = \sup_{\tau \in \mathcal{T}} f(h_o, \tau). 
\]
\end{lemma}
According to this lemma $\sup_{\tau \in \mathcal{R}} f(h, \tau) $ is a continuous function of $h$. This is a standard result and its proof can be found elsewhere. For instance it is equivalent to Lemma 12 in \cite{zheng2015does}. Applying Lemma \ref{lem:continuitytwovariables} to $ | R_{A,h}^t ( \tau^t, \tau^{t-1}, \ldots, \tau^1)$ $- R_{A}^t ( \tau^t, \tau^{t-1}, \ldots, \tau^1) |$ proves  \eqref{eq:asymptotsecondtstep}.

\subsubsection{Consistency of the parameter tuning} At this point we remind the reader that as we discussed in Section \ref{sec:proofroadmapconsistencyamp} we broke the proof of Theorem \ref{cor:empconvopt} in two steps. The first step was to prove
\begin{align}
\lim_{h \rightarrow 0} \lim_{p\rightarrow \infty}  \sup_{\tau^t \in \mathcal{T}} \left| \hat{R}_{h,p}^{t}(\tau^t, \tau^{t-1}, \ldots, \tau^1) - R_A^t(\tau^t, \tau^{t-1}, \ldots, \tau^1)\right| = 0
\end{align}
 in probability, for every $\tau^1, \ldots, \tau^{t-1} \in \mathcal{T}^1 \times \ldots \times \mathcal{T}^{t-1}$, that was established in Theorem \ref{thm:amptune1}. In this section we would like to prove the second step, i.e., the consistency of $\hat{\tau}_{p,h}^1, \hat{\tau}_{p,h}^2, \ldots, \hat{\tau}_{p,h}^t$. For the proof we employ an induction. As a base of induction, we first prove that $\hat{\tau}_{p,h}^1 \rightarrow \tau^{*,1}$ in probability. First note that from the proof of Lemma \ref{lem:quasiconvexsoft}, we conclude that for every $\epsilon>0$ we have
\[
\inf_{\tau: |\tau- \tau^{*,1}|>\epsilon } R_A^1(\tau) >R_A^1(\tau^{*,1}).
\]
In the rest of the proof we assume that $\inf_{\tau: |\tau- \tau^{*,1}|>\epsilon } R_A^1(\tau)-R_A^1(\tau^{*,1})= 2 \gamma$, where $\gamma>0$ is a fixed number. 
We proved in Theorem \ref{thm:amptune1} that 
\begin{equation}
\sup_{\tau \in \mathcal{T}} |R_{A,h}^1(\tau) - R_A^1(\tau)| \rightarrow 0,
\end{equation}
as $h \rightarrow 0$. Hence, we can find $h_o$ such that for every $h<h_o$, $\sup_{\tau \in \mathcal{T}} |R_{A,h}^1(\tau) - R_A^1(\tau)| < \gamma/2$. This implies that for $h< h_0$
\begin{eqnarray}
R_{A,h}(\tau^{*,1}) &<& R_A(\tau^{*,1} )+ \gamma/2,  \label{eq:inf1} \\
\inf_{\tau: |\tau-\tau^{*,1}|> \epsilon}  R_{A,h}^1(\tau) &>&  \inf_{\tau: |\tau-\tau^{*,1}|> \epsilon}  R_A^1(\tau ) - \gamma/2.  \label{eq:inf2}
\end{eqnarray}
In Theorem \ref{thm:amptune1} we proved that
\[
\mathbb{P} \left(\sup_\tau |\hat{R}_{h,p}^1 (\tau) - R_{A,h}^1(\tau)| > \gamma/2\right)\rightarrow 0.
\]
as $p \rightarrow \infty$. As a result, we conclude that
\begin{eqnarray}
&&\mathbb{P} (\hat{R}_{h,p}^1 (\tau^{*,1}) > R_{A,h}^1(\tau^{*,1})+\gamma/2 ) \rightarrow 0, \label{eq:inf3} \\
&&\mathbb{P} \left(\inf_{\tau: |\tau-\tau^{*,1}|> \epsilon}  \hat{R}_{h,p}^1 ( \tau) <\inf_{\tau: |\tau-\tau_{*,1}|> \epsilon} R_{A,h}^1(\tau)-\gamma/2 \right) \rightarrow 0. \label{eq:inf4}
\end{eqnarray}
Hence, by combining \eqref{eq:inf1} and \eqref{eq:inf3} we conclude that
\begin{equation}\label{eq:corfinal1}
\mathbb{P} \left( \hat{R}_{h,p}^1 (\tau^{*,1}) > R_{A}(\tau^{*,1})+\gamma \right) \rightarrow 0. 
\end{equation}
It is also straightforward to combine \eqref{eq:inf2} and \eqref{eq:inf4} and obtain
\begin{equation}\label{eq:corfinal2}
\mathbb{P} \left( \inf_{\tau : \ |\tau-\tau^{*,1}|> \epsilon } \hat{R}_{h,p}^1 ( \tau) < \inf_{\tau : \ |\tau-\tau^{*,1}|> \epsilon } R_{A,h}^1(\tau)-\gamma \right) \rightarrow 0. 
\end{equation}
Combining \eqref{eq:corfinal1} and  \eqref{eq:corfinal2} proves that if $h< h_0$, then
\begin{equation}\label{eq:asymptotconsistencytrick}
\mathbb{P} (|\hat{\tau}^1_{p,h} - \tau^{*,1}| > \epsilon) \rightarrow 0, 
 \end{equation}
 as $p \rightarrow \infty$.

Now we use an induction to show that if $\hat{\tau}^1_{p,h} \overset{p}{\rightarrow} \tau^{*,1}, \hat{\tau}^2_{p,h} \overset{p}{\rightarrow} \tau^{*,2}, \ldots, \hat{\tau}^t_{p,h} \overset{p}{\rightarrow} \tau^{*,t}$, then $\hat{\tau}^{t+1}_{p,h} \overset{p}{\rightarrow} \tau^{*,t+1}$. To keep the notations simpler we only prove this claim for $t=1$. The proof for an arbitrary $t$ is the same. Our proof uses the following steps:
\begin{enumerate}
\item We first prove that $ \left|  \hat{R}_{h,p}^2(\tau^2, \hat{\tau}^1_{p,h})- {R}_A^2(\tau^2, {\tau}^{1,*}) \right| \overset{p}{\rightarrow} 0$. Note that the main reason this can not be derived from Theorem \ref{thm:amptune1} is that now we have used a data-dependent threshold $\hat{\tau}^1_{p,h}$ in the first iteration. In Theorem \ref{thm:amptune1} the threshold does not depend on data.

\item Next we prove that $\sup_{\tau^2 \in \mathcal{T}^2}   | \hat{R}_{h,p}(\tau^2, \hat{\tau}^1_{p,h})- {R}_{B}(\tau^2, {\tau}^{*,1}) | \overset{p}{\rightarrow} 0$. Given the proof technique we used for proving Theorem \ref{thm:amptune1} and the fact that we have already proved \newline $ \left|  \hat{R}_{h,p}(\tau^2, \hat{\tau}^1_{p,h})- {R}_{B}(\tau^2, {\tau}^{*,1}) \right| \overset{p}{\rightarrow} 0$ the proof of this statement is straightforward and will be skipped. 

\item Finally we use the fact that $\sup_{\tau^2 \in \mathcal{T}^2}   | \hat{R}_{h,p}(\tau^2, \hat{\tau}^1_{p,h})- {R}_{B}(\tau^2, {\tau}^{1,*}) | \overset{p}{\rightarrow} 0$ and the proof technique developed in \eqref{eq:asymptotconsistencytrick} to show that $\hat{\tau}^2_{p,h} \rightarrow \tau^{2,*}$. Note that by Theorem \ref{thm:stepwiseoptimal} we already know that even though $\tau^1$ is set to $\tau^{*,1}$, the optimal choice of $\tau^2$ can still be achieved. Since this is exactly the same as the proof of \eqref{eq:asymptotconsistencytrick} we will skip the proof.
\end{enumerate}

We only prove the first of the above three steps. First note that
\begin{eqnarray}
 \lefteqn{\left|  \hat{R}_{h,p}^2(\tau^2, \hat{\tau}^1_{p,h})- {R}_{A}^2(\tau^2, {\tau}^{*,1}) \right|  \leq  \left|  \hat{R}^2_{h,p}(\tau^2, \hat{\tau}^1_{p,h})-  \hat{R}^2_{h,p}(\tau^2, {\tau}^{*,1}) \right| } \nonumber \\
 &&+ \left|  \hat{R}^2_{h,p}(\tau^2, {\tau}^{*,1})- {R}^2_{A,h}(\tau^2, {\tau}^{*,1}) \right| + \left| {R}^2_{A,h}(\tau^2, {\tau}^{*,1}) - {R}^2_A(\tau^2, {\tau}^{*,1}) \right|.
\end{eqnarray}
In our proof we consider each of the three terms on the right and prove that they converge to zero in probability. 

 $\hat{R}^2_{h,p}(\tau^2, \tau^1) $ is differentiable in terms of $\tau^1$ and $\tau^2$. Furthermore, the derivative is bounded with probability one. Hence, by the mean value theorem we have
 \[
 |\hat{R}^2_{h,p}(\tau^2, \hat{\tau}^1_{p,h}) - \hat{R}^2_{h,p}(\tau^2, {\tau}^{*,1})| \leq C |\hat{\tau}^1_{p,h}- {\tau}^{*,1}|,
 \]
 where $C$ is an upper bound on the derivative of  $\hat{R}_{h,p}$ in terms of $\tau^1$. 
  Hence, it is straightforward to use the base of the induction and prove that 
  \begin{eqnarray*}
  \lefteqn{ \mathbb{P} \left( |\hat{R}_{h,p}^2(\tau^2, \hat{\tau}^1_{p,h}) - \hat{R}^2_{h,p}(\tau^2, {\tau}^{*,1})| > \epsilon \right)} \nonumber \\
   &\leq& \mathbb{P} \left( \sup_{(\tau_1, \tau_2) \in \mathcal{T}_1 \times \mathcal{T}_2} ( \hat{R}^2_{h,p})' (\tau_2, \tau_1)  > C \right)   + \mathbb{P} (C |\hat{\tau}^1_{p,h}- {\tau}^{*,1}| > \epsilon). 
  \end{eqnarray*}
  Since both probabilities go to zero as $p \rightarrow \infty$ we conclude that 
  \begin{equation}\label{eq:lastproof1}
  \mathbb{P} \left( |\hat{R}^2_{h,p}(\tau^2, \hat{\tau}^1_{p,h}) - \hat{R}^2_{h,p}(\tau^2, {\tau}^{*,1})| > \epsilon \right)\rightarrow 0.
  \end{equation}
  Furthermore, according to Theorem \ref{thm:amptune1} we have
  \begin{equation}\label{eq:lastproof2}
   | \hat{R}^2_{h,p}(\tau^2, {\tau}^{*,1})- {R}_{A,h}^2(\tau^2, {\tau}^{*,1}) | \overset{p}{\rightarrow} 0.
  \end{equation}
  By combining \eqref{eq:lastproof1} and \eqref{eq:lastproof2} we obtain
  \[
   |\hat{R}_{h,p}^2(\tau^2, \hat{\tau}^1_{p,h}) - {R}_{A,h}^2(\tau^2, {\tau}^{*,1})| \overset{p}{\rightarrow} 0.
  \]
  The proof of $ \left| {R}_{A,h}^2(\tau^2, {\tau}^{*,1}) - {R}_A^2(\tau^2, {\tau}^{*,1}) \right|  \rightarrow 0$ as $h \rightarrow 0$, is a straightforward application of the continuity of ${R}_{A,h}^2(\tau^2, {\tau}^{*,1})$ with respect to $(h, \tau^2)$ and is hence skipped. This completes our proof of the consistency of $\hat{\tau}_{p,h}^{2}$.
  
 \subsection{Proof of Proposition \ref{prop:optlamoptampcon}} \label{sec:prooflasttheorem}
 
 \subsubsection{Roadmap of the proof}
 Since the proof of this theorem is also long we describe the steps of the proof below. 
 \begin{enumerate}
 \item We consider the AMP algorithm in which the threshold is set to $\chi \sigma^t$ at every iteration. We show that for every value of $\lambda$ there is a unique value of $\chi$ such that the solution of AMP has the same MSE as the solution of LASSO and vice versa. Hence, we conclude that if $\lambda_{opt}$ is the best choice of the regularization parameter for LASSO (that minimizes MSE), then $\chi_{opt}$ the corresponding value of $\chi$ is the best value that can be picked in AMP (when the threshold value at iteration $t$ is $\chi \sigma^t$). Clearly the other direction is also true, i.e., if $\chi_{opt}$ is the best value of the threshold for AMP, then the corresponding $\lambda$ will be the optimal value of the regularization parameter for LASSO. 
 
 \item In the second step we connect two versions of AMP: (i) AMP with threshold $\chi_{opt} \sigma^t$ and (ii) optimal AMP with threshold $\tau^{*,t}$ defined in Definition \ref{def:tau}.  We prove that the final solutions of these two AMP algorithms have the same mean square error. Note that at any iteration these two thresholding policies may lead to different estimates. In fact, the AMP algorithm that employs $\tau^{*,1}, \tau^{*,2}, \ldots$ converges to its final solution faster. However, they eventually converge to solutions with the same MSE. 
  \end{enumerate}
  
  These two steps are proved in Sections \ref{sec:step1lastproof} and \ref{sec:step2lastproof} respectively.
   
  \subsubsection{Proof of Step 1}\label{sec:step1lastproof}
  
   According to Theorem \ref{thm:eqpseudolip} the solution of LASSO satisfies the following equation:
  \begin{equation}\label{eq:lassoobslastproof1}
\lim_{p \rightarrow \infty} \frac{1}{p} \|\hat{\beta}^{\lambda}(p)-{\beta}_{o}(p)\|_2^2 = \bE_{B,W} (\eta(B+\hsig W; \chi \hsig)- B))^2,
\end{equation}
where $\hsig$ and $\chi$ satisfy the following equations with $\sigma_{\omega}$ being the variance of the input noise:
\begin{eqnarray} \label{eq:fixedpointlastproof1}
\hsig^2 &=& \sigma_{\omega}^2+\frac{1}{\delta} \mathbb{E}_{B, W} [(\eta(B +\hsig W; \chi \hsig) -B)^2], \label{eq:fixedpoint11last} \\
\lambda &=& \chi \hsig \left(1-\frac{1}{\delta} \mathbb{P}(|B +\hsig W| > \chi \hsig) \right). \label{eq:fixedpoint21last}
\end{eqnarray}
Define $\lambda_{opt} \triangleq \arg\min_{\lambda}  \lim_{p \rightarrow \infty} \frac{1}{p} \|\hat{\beta}^{\lambda}(p)-{\beta}_{o}(p)\|_2^2$. 
This is the value of $\lambda$ that minimizes the asymptotic MSE. Let $\hat{\sigma}_{opt}$ and $\chi_{opt}$ denote the values of $(\hat{\sigma}, \chi)$ that satisfy  \eqref{eq:fixedpointlastproof1} with $\lambda_{opt}$. Suppose that we run the AMP algorithm with thresholding policy $\chi_{opt} \sigma^t$. According to Theorem \ref{thm:ampeqpseudo_lip} the fixed point of this AMP algorithm will satisfy
\[
\lim_{t \rightarrow \infty} \lim_{p \rightarrow \infty} \frac{1}{p} \|\hat{\beta}^t(p) - \beta_o(p)\|_2^2 =  \bE_{B,W} (\eta(B+\hsig W; \chi_{opt} \hsig)- B))^2,
\] 
 where 
\begin{eqnarray} \label{eq:fixedpointlastproofAMP1}
\hsig^2 &=& \sigma_{\omega}^2+\frac{1}{\delta} \mathbb{E}_{B, W} [(\eta(B +\hsig W; \chi \hsig) -B)^2]. 
\end{eqnarray}
According to Lemma \ref{lem:uniquefixedpointconc}, \eqref{eq:fixedpointlastproofAMP1} has a unique fixed point. Hence, the MSE of AMP with threshold $\chi_{opt}\sigma^t$ is the same as the MSE of LASSO with optimal threshold parameter $\lambda$. As the next step we would like to show that if we use another threshold of the form $\chi \sigma^t$ in AMP, it will lead to larger MSE. Suppose that $\bar{\chi}$ leads to a smaller asymptotic MSE in AMP. Define $\bar{\sigma}$ as the fixed point of  
$\sigma_{\omega}^2+\frac{1}{\delta} \mathbb{E}_{B, W} [(\eta(B +\hsig W; \bar{\chi} \hsig) -B)^2]$. Then define 
\[
\bar{\lambda} \triangleq  \bar{\chi} \bar{\sigma} \left(1-\frac{1}{\delta} \mathbb{P}(|B +\bar{\sigma} W| > \bar{\chi} \bar{\sigma}) \right).
 \]
Suppose that we use $\bar{\lambda}$ in LASSO, then according to Theorem \ref{thm:eqpseudolip} the MSE of LASSO will be the same as MSE of AMP with threshold $\bar{\chi} \sigma^t$ that is lower than the threshold of LASSO for $\lambda_{opt}$. This contradiction confirms that we cannot find $\bar{\chi}$ to reduce the final asymptotic MSE of AMP. 

  \subsubsection{Proof of Step 2}\label{sec:step2lastproof} Suppose that $(\sigma^t, \tau^{*,t})$ converge to $\sigma_\infty, \tau_\infty$ in the optimal AMP algorithm. Our claim is that  $\sigma_\infty$ is the same as $\hat{\sigma}_{opt}$, discussed in the last section. By the definition of $\tau^{*,t}$ we have $\sigma_\infty \leq  \hat{\sigma}_{opt}$. Hence, we only have to prove that $\sigma_\infty \geq \hat{\sigma}_{opt}$. Define $\chi_\infty \triangleq \frac{\tau_\infty}{\sigma_\infty}$. Suppose that we run AMP with thresholding policy $\chi_\infty \sigma^t$. Then the fixed point of this new version of AMP will satisfy 
\begin{eqnarray}\label{eq:fixedpointoptimalchi2}
\hsig^2 &=& \sigma_{\omega}^2+\frac{1}{\delta} \mathbb{E}_{B, W} [(\eta(B +\hsig W; \chi_\infty \hsig) -B)^2]. 
\end{eqnarray}
 First note that this equation has a solution at $\hsig = \sigma_\infty$. That is because $\chi_\infty \sigma_\infty = \tau_\infty$ and \eqref{eq:fixedpointoptimalchi2} becomes the same as the fixed point equation for the optimal AMP. Secondly, note that according to Lemma \ref{lem:uniquefixedpointconc}, \eqref{eq:fixedpointoptimalchi2} has a unique fixed point. Therefore, AMP with threshold $\chi_\infty \sigma^t$ will converge to this unique fixed point. Since by construction the MSE of AMP with threshold $\chi_{opt} \sigma^t$ is smaller than AMP with threshold $\chi_{\infty} \sigma^t$, we conclude that $\sigma_\infty \geq \hat{\sigma}_{opt}$. This completes the proof of the second step.

\section{Simulation results}\label{sec:simulation results}
\subsection{Objectives}
Our theoretical results are based on the asymptotic analysis. Asymptotic analysis is usually criticized for its lack of applicability to the medium problem sizes. In this section our goal is to evaluate the performance of our risk estimates and our tuning schemes in medium problem sizes through simulations. Our results in this section intend to answer the following questions:

\begin{enumerate}
\item The first question of practical importance is for what problem sizes such asymptotic analysis is useful. While addressing this question in its full generality is quite challenging, we will provide some simulation results to provide some rule of thumb for practitioners and also shed some light on the finite sample size applicability of our results.

\item Under the asymptotic settings the choice of the parameters in the risk estimate and the tuning schemes are clear. However, they are less clear in the medium problem sizes. In Section \ref{sec:practical_bisec} we show that the parameters of our algorithms can be tuned easily for medium problem sizes as well.  Furthermore, we will present a parameter free version of approximate message passing that requires no tuning from the user. 

\item We present some simulations on the overall performance of our parameter-free AMP algorithm and compare it with some of the other existing versions of AMP.   
\end{enumerate}

\subsection{Risk estimate and impact of sample size} This section explores the accuracy of the risk estimate presented in \eqref{eq:empriskdefamp}. We first explain how we set $h$ and $\sigma^t$ in \eqref{eq:empriskdefamp} and then show the accuracy of the final estimate in Section \ref{sec:riskest_finitep}. 

\begin{figure}[h!]
\centering
\includegraphics[width= 9cm]{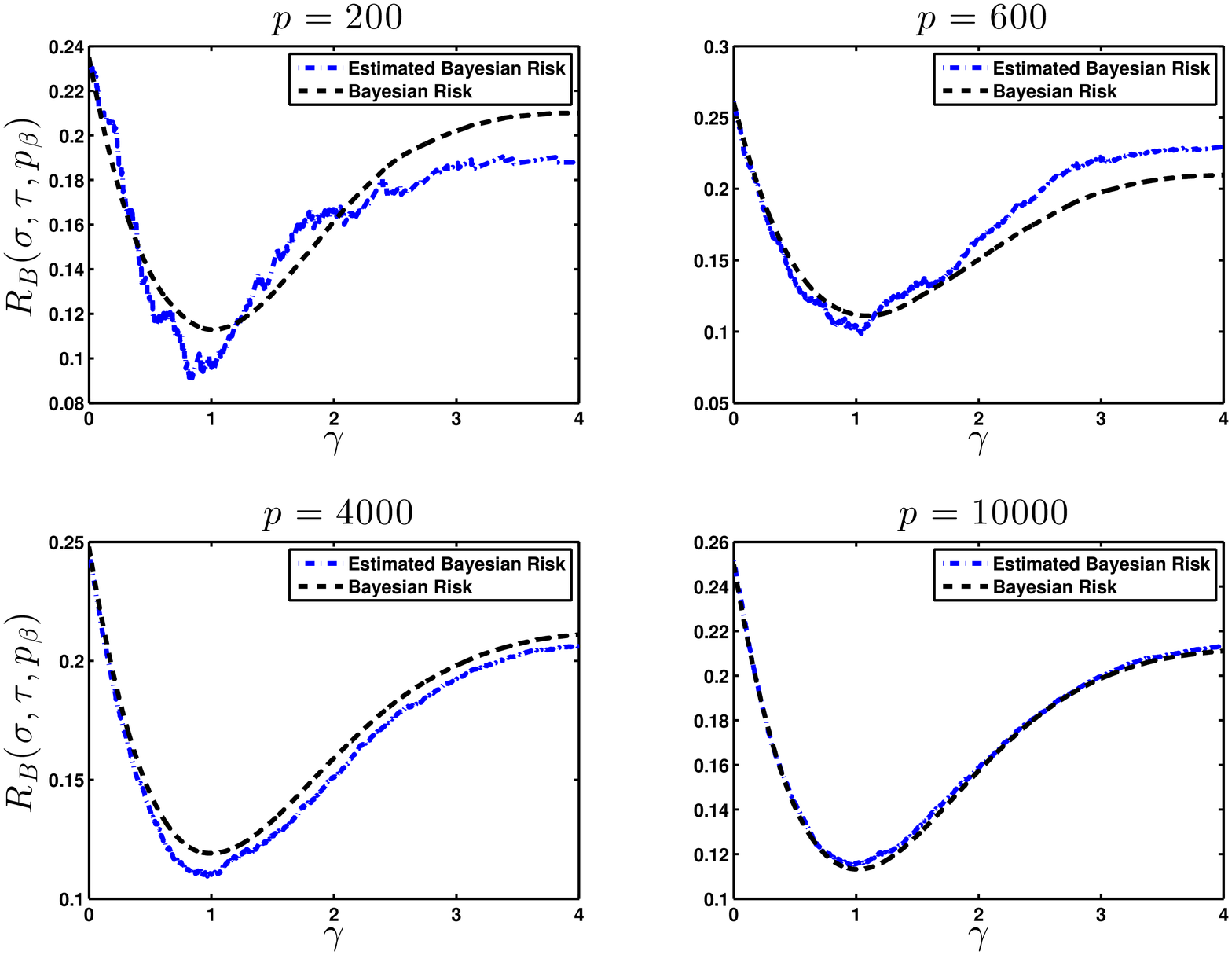}
\caption{Performance of the risk estimator for different values of $p$ in the 1\ts{st} iteration of AMP. In this experiment $\delta=0.85, \rho=0.25$ and we consider noiseless measurements. Black curve is the Bayes risk to which our estimates will converge as $p \rightarrow \infty$. }
\label{fig:N1}
\end{figure}

\begin{figure}[h!]
\centering
\includegraphics[width= 9cm]{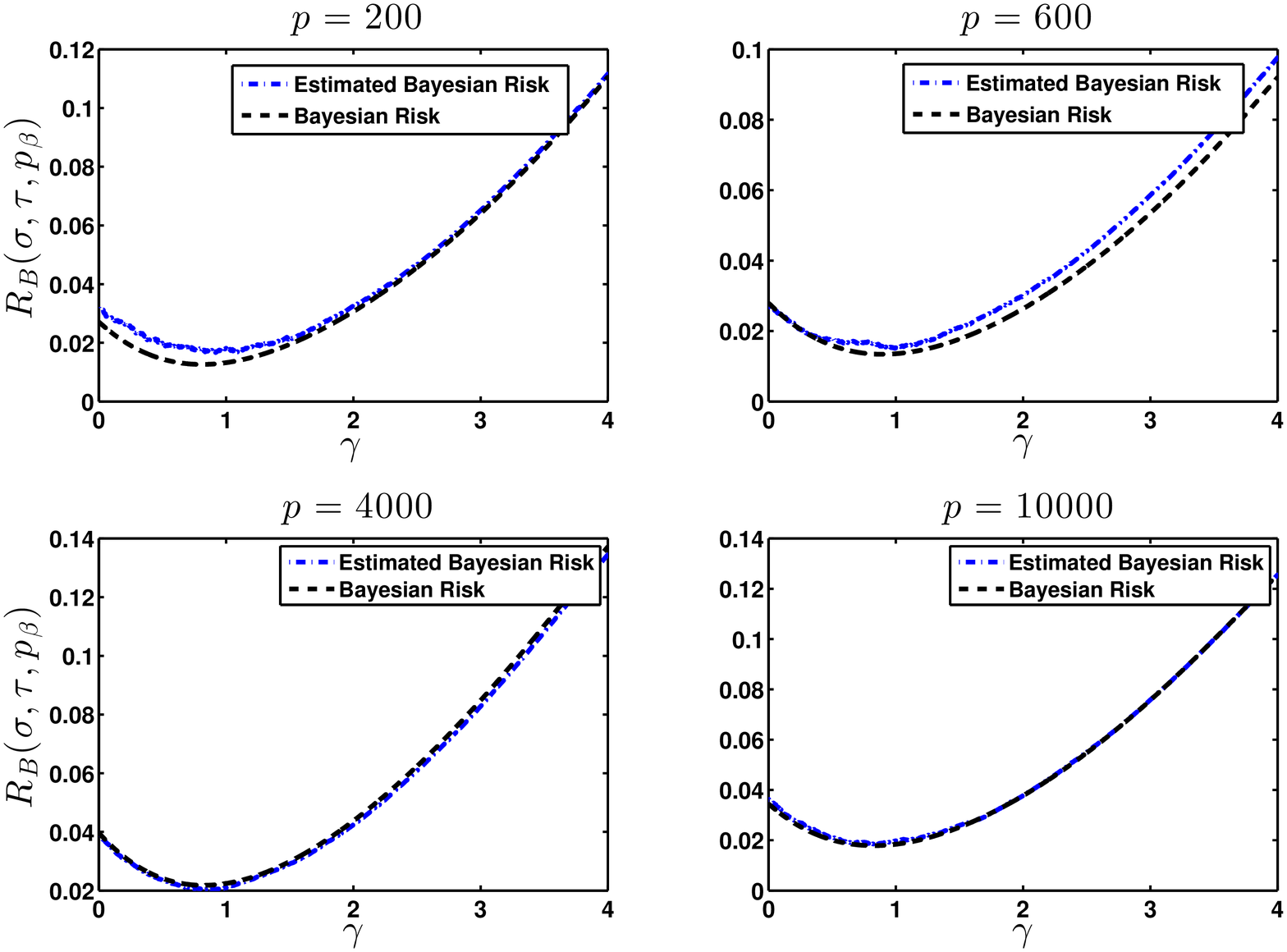}
\caption{Performance of the risk estimator for different values of $p$ in the 10\ts{th} iteration of AMP. In this experiment $\delta=0.85, \rho=0.25$ and we consider noiseless measurements. Black curve is the Bayes risk to which our estimates will converge as $p \rightarrow \infty$.}
\label{fig:N4}
\end{figure}

\subsubsection{Smoothing parameter $h$}
The risk estimate, given in \eqref{eq:empriskdefamp}, employs a Gaussian kernel with bandwidth $h$ for smoothing the soft-thresholding function. While this step is required for our theoretical analysis, our simulations suggest that the estimate is accurate even for $h=0$. $h=0$ is hence used in all our simulations. We believe all our theoretical results are also true when $h=0$. However we leave the theoretical justification of this claim for future investigations.   

\subsubsection{Estimating the noise variance}
The second issue with our risk estimate is the existence of $\sigma^t$ in \eqref{eq:empriskdefamp} that is not known in practice. As we mentioned in Section \ref{ssec:sureestimate}, the estimate $\left( \hat{\sigma}^t \right)^2 = \frac{1}{n} \sum_{i=1}^n (z^t_i)^2$ provides a consistent estimate of $\sigma^t$ and all the results we proved regarding the risk estimate are also accurate even if we replace $\sigma^t$ with $\hat{\sigma}^t$. In all the simulations reported here we have used the estimate of the variance instead of its actual value. 

\subsubsection{Impact of $p$ on the risk estimate}\label{sec:riskest_finitep}
The main goal of this section is to evaluate the accuracy of the risk estimate for different values of $p$. In our simulations we have assumed that $\beta_o$ has $k$ non-zero coefficients, all of which are equal to $1$. This distribution is known to be the least favorable distribution for LASSO and AMP (with soft thresholding) \cite{DoMaMoNSPT}.  We present the average risk and also the estimated risk at iterations $1$ and $10$ of AMP. Given the notation $\rho=k/n$ we have considered the following three cases:  
\begin{enumerate}

\item Figure \ref{fig:N1} and Figure \ref{fig:N4}: $\delta=0.85$, and $\rho=0.25$, $\sigma_w = 0$, $p \in \{200,600,4000,10000\}$,

\item Figure \ref{fig:N2} and Figure \ref{fig:N5}: $\delta=0.85$, and $\rho=0.25$, $\sigma_w = 0.5$, $p \in \{200,600,4000,10000\}$,

\item Figure \ref{fig:N3} and Figure \ref{fig:N6}: $\delta=0.2$, and $\rho=0.1$, $\sigma_w = 0.1$, $p \in \{200,600,4000,10000\}$.

\end{enumerate}

The results of our simulations are summarized in Figures \ref{fig:N1} to \ref{fig:N6}. A rule of thumb that these figures suggest is that if $p>5000$, the risk estimate seems to be accurate. In many cases our estimate is accurate even if $p< 5000$. However, it is usually inaccurate when $p< 500$.

\begin{figure}[h!]
\centering
\includegraphics[width= 9cm]{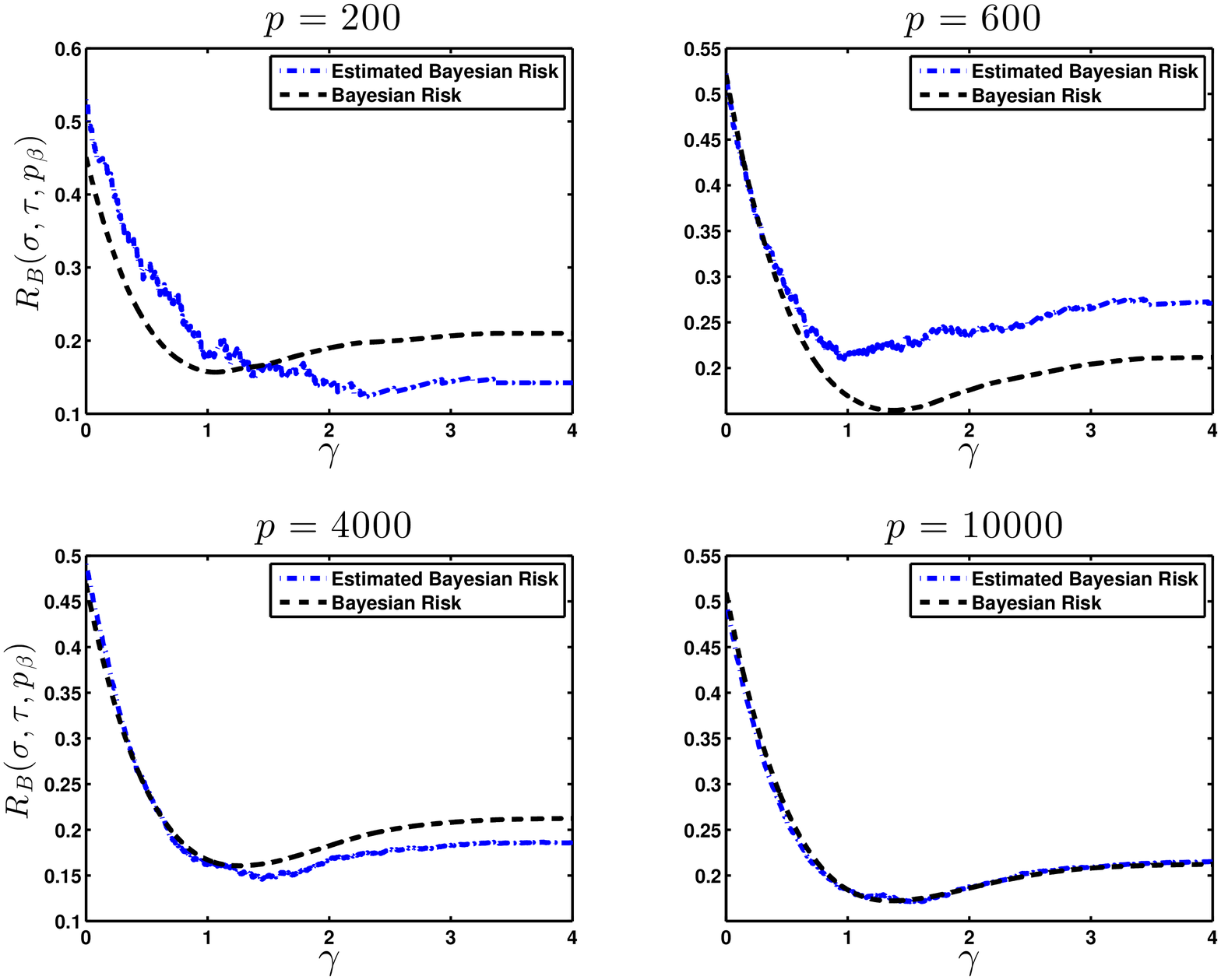}
\caption{Performance of the risk estimator for different values of $p$ in the 1\ts{st} iteration of AMP. In this experiment $\delta=0.85, \rho=0.25$ and the standard deviation of measurement noise is 0.5. Black curve is the Bayes risk to which our estimates will converge as $p \rightarrow \infty$.}
\label{fig:N2}
\end{figure}

\begin{figure}[h!]
\centering
\includegraphics[width= 9cm]{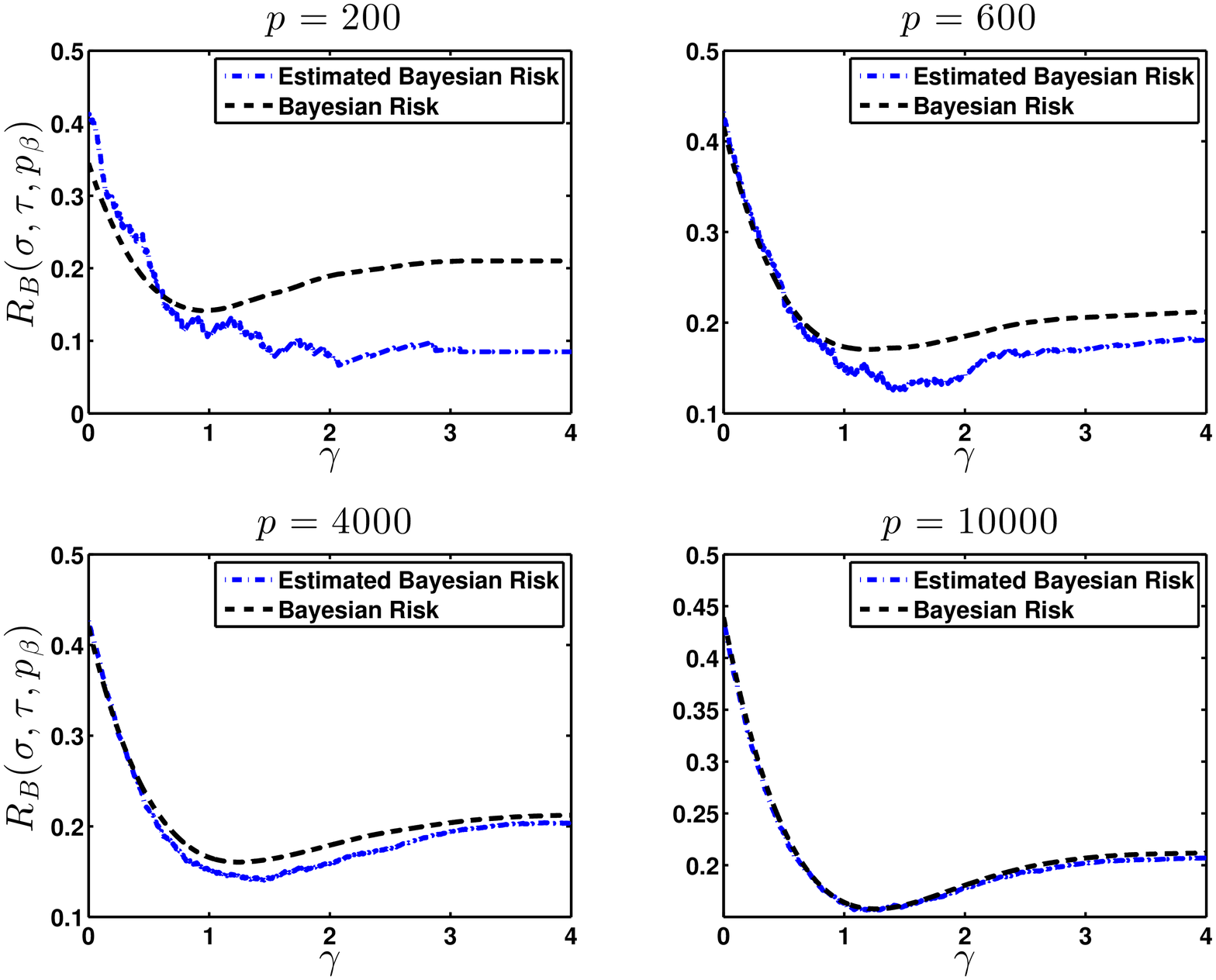}
\caption{Performance of the risk estimator for different values of $p$ in the 10\ts{th} iteration of AMP. In this experiment $\delta=0.85, \rho=0.25$ and the standard deviation of measurement noise is 0.5. Black curve is the Bayes risk to which our estimates will converge as $p \rightarrow \infty$.}
\label{fig:N5}
\end{figure}

\begin{figure}[h!]
\centering
\includegraphics[width= 9cm]{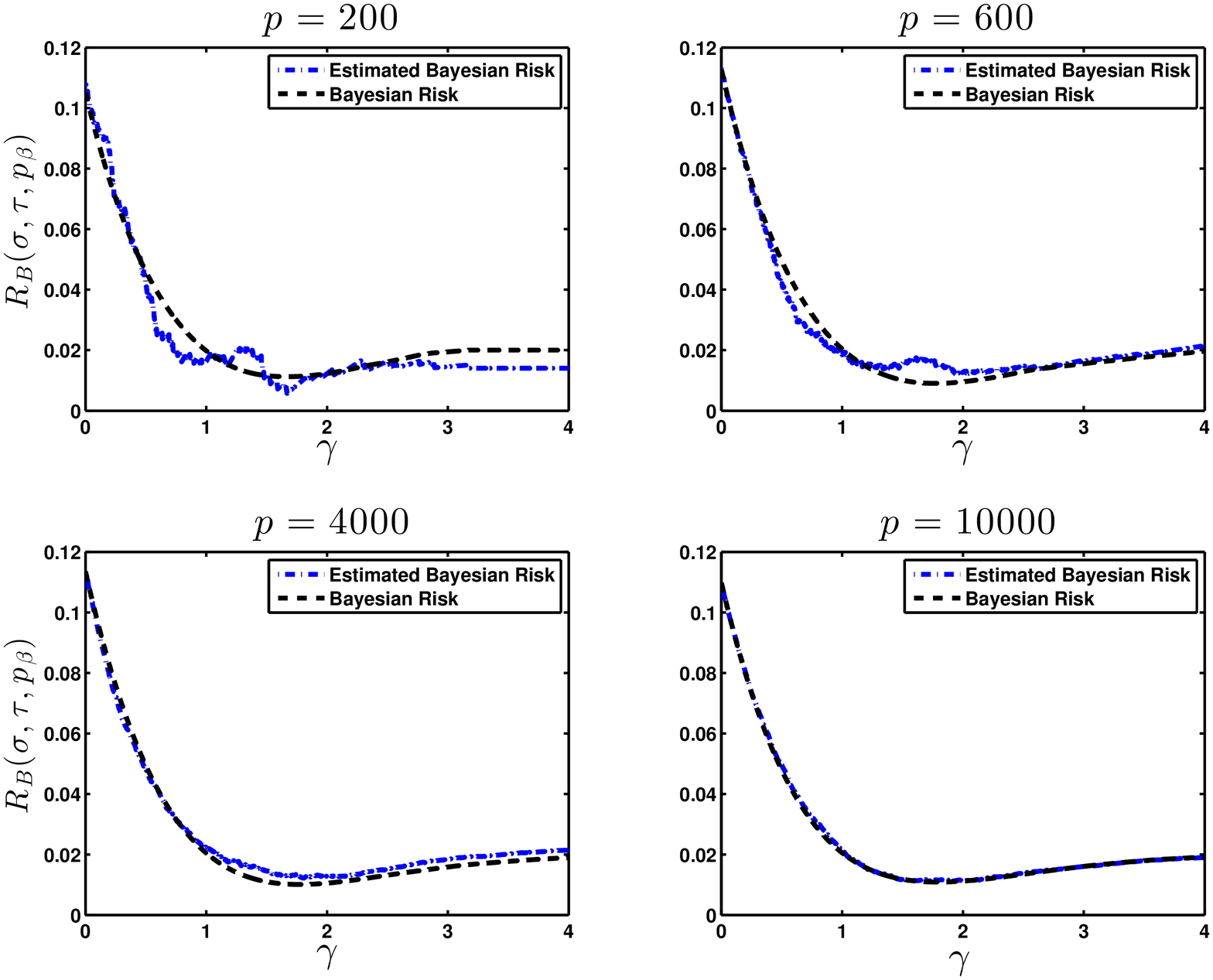}
\caption{Performance of the risk estimator for different values of $p$ in the 1\ts{st} iteration of AMP. In this experiment $\delta=0.2, \rho=0.1$ and the standard deviation of measurement noise is 0.1. Black curve is the Bayes risk to which our estimates will converge as $p \rightarrow \infty$.}
\label{fig:N3}
\end{figure}

\begin{figure}[h!]
\centering
\includegraphics[width= 9cm]{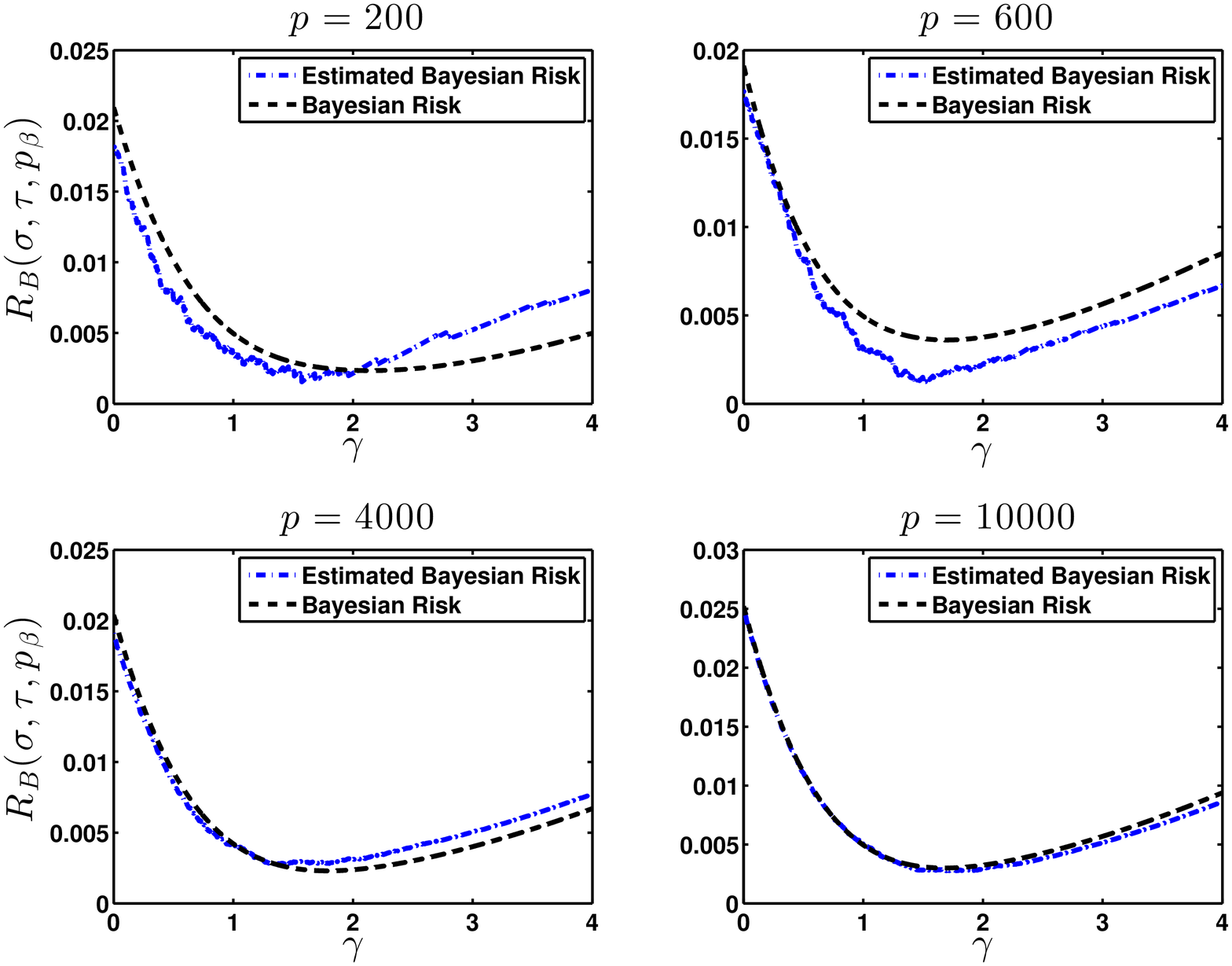}
\caption{Performance of the risk estimator for different values of $p$ in the 10\ts{th} iteration of AMP. In this experiment $\delta=0.2, \rho=0.1$ and the standard deviation of measurement noise is 0.1. Black curve is the Bayes risk to which our estimates will converge as $p \rightarrow \infty$.}
\label{fig:N6}
\end{figure}

\subsection{Practical bisection method}\label{sec:practical_bisec}
In this section we explore the performance of the bisection method we proposed in Section \ref{ssec:sureestimate}.

\subsubsection{Roadmap}
The bisection method introduced in Section \ref{ssec:sureestimate} has the parameters $\epsilon$ and $\Delta$, $\bar{\tau}$, and $\underline{\tau}$. Setting these parameters under the asymptotic setting is straightforward; Set $\epsilon$, $\Delta$ and $\underline{\tau}$ to a very small number, e.g. $10^{-10}$, and set $\tau$ to a large number. Under these settings the algorithm will work well in the limit $p \rightarrow \infty$. This is not necessarily true in medium problem sizes as we will describe later. Hence, before evaluating the performance of our algorithm on medium problem sizes, we should explain how these parameters are set. As we will show in the next few sections, setting these parameters is straightforward for even medium problem sizes. Sections \ref{sec:bisect_eps}, \ref{sec:bisect_delta}, and \ref{sec:bisect_tau} explore the impact of these parameters on the performance of the bisection algorithm and then present an automated method for setting these parameters. Section \ref{sec:practicalbisect} evaluates the performance of the automated bisection method on several problem instances. 

\subsection{Reparameterization} 
 A simple reparametrization of AMP may help us in setting the parameters of the bisection method. Define $\gamma^t \triangleq \tau^t/ \sigma^t$.  The problem of estimating the optimal value of $\tau^t$ is the same as that of $\gamma^t$ for $\sigma^t>0$. Furthermore, at iteration $t$ define $\bar{\gamma} = \bar{\tau}/ \sigma$ and also $\underline{\gamma} = \underline{\tau}/ \sigma$.


\subsubsection{Setting $\epsilon$}\label{sec:bisect_eps}
As described in Section \ref{ssec:sureestimate} we expect $\epsilon$ to be small. The exact value of $\epsilon$ impacts the number of iterations the bisection algorithm requires for convergence and also the accuracy of the bisection method. Since $15$ iterations of bisection is sufficient for an accurate estimate of $\gamma^t$ (note that at every iteration of the bisection method the size of the search interval is halved), we can set the number of iterations of the bi-section method to 15 and set $\epsilon$ to $0$. If we want to reduce the number of iterations of the bisection method, then we can pick larger values of $\epsilon$. However, note that each iteration of the bisection method is in-expensive and whether we do it for $10$ or $15$ iterations does not affect the overall computational complexity of the algorithm much.   

%

\subsubsection{Setting $\Delta$} \label{sec:bisect_delta}
 The second parameter of the bisection method is $\Delta$. While asymptotically we want $\Delta$ to be very small, in finite sample sizes this is slightly more tricky. If $\Delta$ is too small, then our estimate of the derivative of $\hat{R}$, introduced in Section \ref{ssec:sureestimate}, will not be accurate and hence the algorithm may be trapped at a point far from the optimal value of $\tau^t$. Figure \ref{fig:Delta} shows this phenomenon. This figure shows the performance of the bisection method for 6 different values of $\Delta$. As is clear in this figure, the algorithm fails when $\Delta$ is too small ( $\Delta=10^{-3}$ ) or too large ($\Delta > 1$ ). However, for a wide range of values of $\Delta$, i.e. from $\Delta= 0.01$ to $\Delta = 1$, the algorithm performs well.  Also, note that as we increase the dimension, the range of values of $\Delta$ for which the algorithm works well expands as predicted by our asymptotic results.  This phenomenon can be observed by comparing Figures \ref{fig:Delta} and \ref{fig:Delta2} and is due to the fact that our risk estimate becomes more accurate as $p$ increases. Furthermore, remember that the main reason we preferred the bisection method over the grid search in the AMP tuning  was the fact that we have to apply it at every iteration of AMP. If we wanted to apply it in only one iteration, then we could have done a grid search. This motivates the following approach for setting $\Delta$: Set $\Delta$ accurately in the first iteration and use the same $\Delta$ for the other iterations of AMP. Hence, we do the following steps for setting $\Delta$.

\begin{figure}[h!]
\includegraphics[width= 13cm]{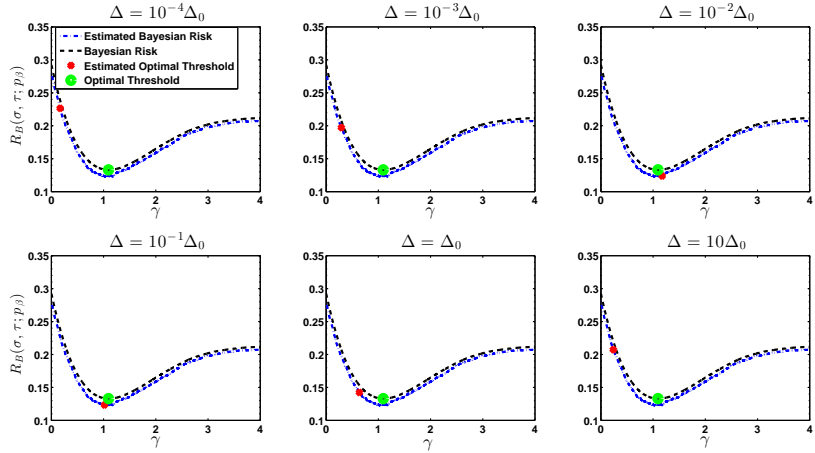}
\caption{Performance of modified bisection algorithm in estimating $\hat{\gamma}_{opt}$ for different values of $\Delta$. In this experiment $p=4000, \delta=0.85, \rho=0.25$, and the standard deviation of the noise of the measurements $\sigma$ is 0.2. Blue dashed curve and black dashed curve show the estimated and the true risk functions, respectively. Green circle shows the optimal threshold. Finally, the red cross shows the estimated optimal threshold. }
\label{fig:Delta}
\end{figure}
\begin{figure}[h!]
\includegraphics[width=13cm]{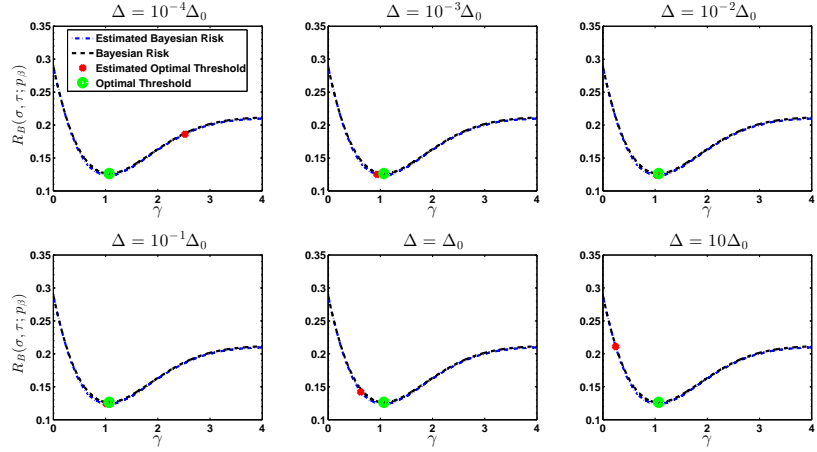}
\caption{Performance of modified bisection algorithm in estimating $\hat{\gamma}_{opt}$ for different values of $\Delta$. In this experiment $p=40000, \delta=0.85, \rho=0.25$, and the standard deviation of the noise of the measurements $\sigma$ is 0.2. Blue dashed curve and black dashed curve show the estimated and true risk functions, respectively. Green circle shows the optimal threshold. Finally, the red cross shows the estimated optimal threshold. }
\label{fig:Delta2}
\end{figure}

\begin{enumerate}[i.]
\item Set $\Delta$ according to the following procedure in the first iteration. For every $\Delta \in \mathcal{D} = \{10^{-\ell_1}, 5\times10^{- \ell_1+1}, \ldots, 5^{\ell_2} \times 10^{-\ell_1}  \}$. Typical values for $\ell_1$ and $\ell_2$ are $-5$ and $6$ respectively. Run the bisection method for the first iteration of AMP with these different values. Pick the one that leads to the best $\gamma$, i.e., the one that has the minimum risk estimate.  Call this value of $\Delta$, $\Delta^*$. 
\item Use $\Delta^*$ for all the other iterations of AMP. 
\end{enumerate}

Note that the computational complexity of this approach is in general less than six times that of a bisection method in the first iteration and is the same for the other iterations. Hence, it does not have a major effect on the overall computational complexity of AMP.

\subsubsection{Setting $\bar{\gamma}$ and $\underline{\gamma}$} \label{sec:bisect_tau}
Setting these two parameters is also straightforward. $\underline{\gamma} \geq 0$. Hence, we set it to $0$. Also we set $\bar{\gamma} = \sup_i |(\beta^t + X^*z^t)_i|/ \sigma^t$. Note that any value of $\gamma > \bar{\gamma}$ leads to zero estimate and will hence give the same estimate as $\bar{\gamma}$.

\subsection{Performance of the practical subsection method}\label{sec:practicalbisect}

\subsubsection{Summary of the bisection method}

Our practical bisection method is summarized in the following flowchart. In this section we evaluate the quality of the estimates obtained from the modified bisection algorithm at different iterations of AMP. We use following experimental set-up outlined in Algorithm \ref{alg:Bisect}.

\begin{itemize}
\item[(i)] We set $p=2000$ in all experiments unless we mention otherwise. Number of measurements $n$ and the level of sparsity $k$ are obtained according to $n=\lfloor \delta p \rfloor$ and $k=\lfloor \rho n \rfloor$. In these experiments we set $\delta =0.85$ and $\rho=0.25$. $X$ is a measurement matrix having iid entries drawn from Gaussian distribution $N\left(0,\frac{1}{n}\right)$. The signal to be reconstructed $\beta_o\in \mathbb{R}^p$ has only $k$ non-zero values. We have tested the performance on several distributions for the non-zero entries of $\beta_o$. However, in the following experiments we consider a point mass at 1 for the distribution for non-zero entries of $\beta_o$.
\item[(ii)] The maximum number of iterations of AMP is set to 200. One can study other stopping rules to improve the efficiency of the algorithm. The measurements are given by $y = X\beta_o + w$, where $w\sim N(0, \sigma_w^2 I)$. We consider 2 different cases for the noise: $\sigma_w = 0, 0.2$. 

\end{itemize}

\begin{algorithm}                      
\caption{Finding the $\arg \min_{\gamma} \hat{R}(\tau)$ with the modified bisection algorithm}         
\label{alg:Bisect}                     
\begin{algorithmic}                    
\REQUIRE $\hat{R}(\tau)$
\ENSURE $\arg \min_{\gamma} \hat{R}(\tau)$ 
\STATE $\bar{\gamma} = \sup_i |(\beta^t + X^*z^t)_i|/ \sigma^t$
\STATE $\underline{\gamma} =0$
\FOR {$i=1:15$}
\STATE $\gamma=\frac{\underline{\gamma}+\bar{\gamma}}{2}$
\STATE $\text{Diff}=\frac{\hat{R}(\gamma+\Delta^*)-\hat{R}(\gamma)}{\Delta^*}$
\IF{$\text{Diff}>0$}
	\STATE $\bar{\gamma}=\gamma$
\ELSIF{$\text{Diff}<0$}
\STATE $\underline{\gamma}=\gamma$
\ELSE
\STATE $\text{Break}$
\ENDIF
\ENDFOR

\end{algorithmic}
\end{algorithm}

  Figure \ref{fig:noise_free} shows the performance of modified bisection algorithm for a set of noise-free measurements, i.e. $\sigma_w=0$. It contains 4 plots each of which corresponds to a specific iteration of AMP. Each plot in Figure \ref{fig:noise_free} contains the Bayesian risk (defined in \eqref{equ:BayesRisk}), estimate of the Bayesian risk, the estimate of the optimal value of $\gamma$ obtained from the bisection method, $\hat{\gamma}_{opt}$, and the optimal value of $\gamma$, $\gamma_{opt}$.
\begin{figure}[h!]
\includegraphics[width= 11cm]{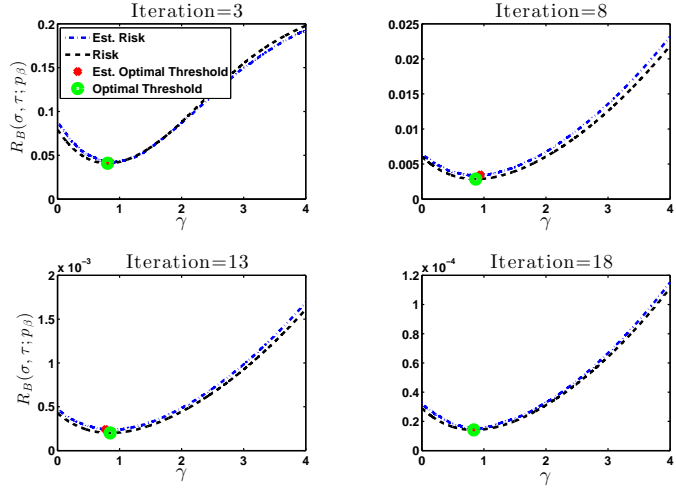}
\caption{Performance of modified bisection method in estimating $\hat{\gamma}_{opt}$ in different iterations of AMP. In this experiment $p=2000, \delta=0.85, \rho=0.25$, and we consider noiseless measurements.}
\label{fig:noise_free}
\end{figure}

We now add some measurement noise to our system and set $\sigma_w= 0.2$ and repeat the same experiment. Figure \ref{fig:Est_Effect_Iteration_Sig2} summarizes our simulation result. 


\begin{figure}[h!]
\includegraphics[width= 11cm]{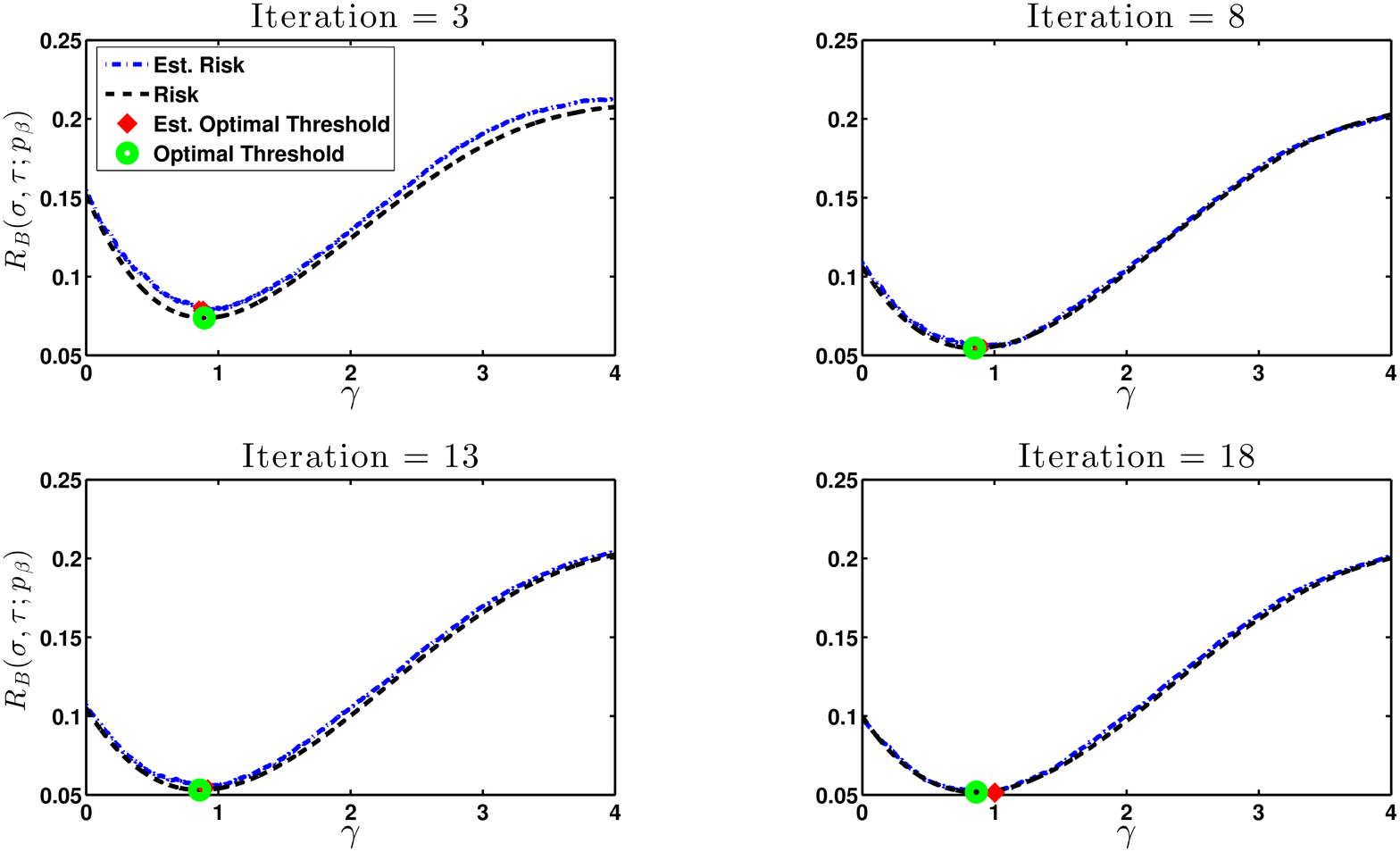}
\caption{Performance of modified bisection method in estimating $\hat{\gamma}_{opt}$ in different iterations of AMP. In this experiment $p=2000, \delta=0.85, \rho=0.25$, and the standard deviation of the noise of the measurements $\sigma$ is 0.2.}
\label{fig:Est_Effect_Iteration_Sig2}
\end{figure}

\subsubsection{Comparison with other tuning approaches of AMP} 
Now we compare the performance of our proposed algorithm (modified bisection method for AMP tuning) with other tuning approaches proposed for AMP. There are two main approaches that have been proposed elsewhere: (i) Maximin tuning that picks the parameter that achieves the highest phase transition \cite{MaDo09sp, DoMaMo09}, and (ii) minimax tuning that selects the parameter that achieves the lowest noise sensitivity \citep{DoMaMoNSPT}. Among these two approaches the second approach is not practical since it requires the sparsity level of $\beta_o$. Furthermore, to be optimal for any signal, it requires to know the distribution of the non-zero coefficients of $\beta_o$ as well. Hence, while the second method is a useful tool for theoretical analysis it cannot be considered a practical parameter tuning scheme. On the other hand, the maximin method is a method of tuning for the noiseless settings and has shown good performance in simulations \cite{DoMaMo09}. In this section, we compare the performance of our algorithm with the maximin approach under the noiseless setting (the setting in which maximin is designed for) to show that even in such cases our approach provides a better convergence rate than the maximin approach.  

Figure \ref{fig:MSE} compares the convergence of the AMP for different tuning schemes. As we can see from the figures, the tuning scheme proposed in this paper is able to reconstruct the original signal up to a certain reconstruction error by the least number of iterations. This is much faster than the AMP that uses Maximin tuning. For instance, the AMP with maximin tuning achieves requires more than 17 iterations the achieve the accuracy that parameterless AMP achieves in 10 iterations. Note that the extra calculations that the bisection method requires in 10 iterations, is much less than that of a one matrix-vector multiplication. Hence, overall with less computations we achieve more accurate results. Another advantage of parameterless AMP over the AMP that is tuned according to maximin approach is that parameterless AMP adapts itself to the distribution of $\beta$, and the noise variance and will achieve the optimal MSE, while maximin AMP is only optimal for the noiseless settings.

We should also note that although grid-search AMP performs almost like the parameterless AMP, it needs an exhaustive search among many threshold values to find the optimal one and hence is computationally expensive. 

\begin{figure}[h!]
\centering
\includegraphics[width= 11cm]{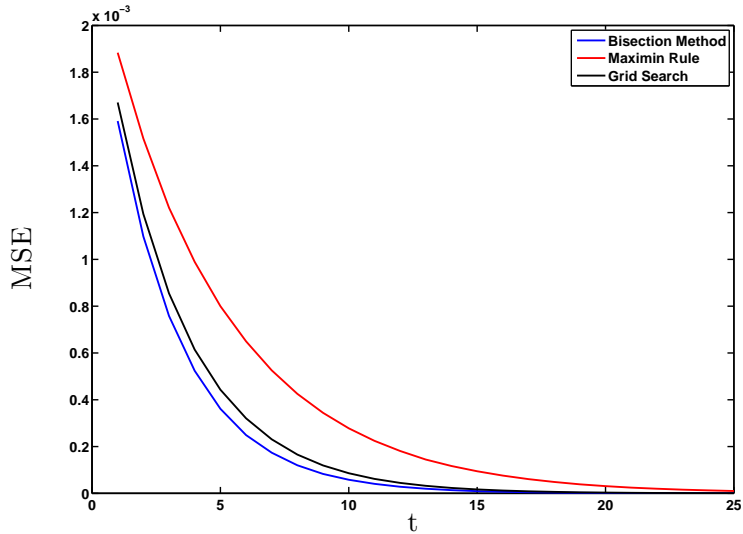}
\caption{MSE of AMP at each iteration for three different threshold setting approaches. The blue curve corresponds to the modified bisection method introduced in this paper. In the black curve, the threshold is set to a constant value. This constant threshold, which causes the smallest reconstruction error, is found by the grid search method. In the red curve, threshold is set to the value which gives the best phase transition \citep{DoMaMoNSPT}.}
\label{fig:MSE}
\end{figure}

 \subsection{Simulation Details}\label{sec:simulationdetails}

Here we include the details of the simulations whose results we reported in the previous sections.

\subsubsection{Figure \ref{fig:activeset} }\label{subsec:fig:activeset}
The dataset we used in this simulation is taken from \cite{efron2004least}. The response variables $y \in \mathbb{R}^{442}$ correspond to the diabetes progression in one year in 442 patients. We have 10  variables, namely age, sex, body mass index, average blood pressure, and six blood serum measurements. Therefore $\beta \in \mathbb{R}^{10}$.  
We have solved LASSO for different values of  $\lambda$ and presented the number of nonzero elements in $\beta$, i.e. $\|\beta \|_0$, as a function of $\lambda$ in Figure \ref{fig:activeset} . As mentioned previously, for this specific problem, this function is not monotone decreasing.

\subsubsection{Figure \ref{fig:LassoPathRandom}}\label{subsec:fig:LassoPathRandom}
Here, we consider the recovery of a sparse signal in the presence of measurement noise. $\beta_o \in \mathbb{R}^{2000}$ is a sparse signal with 100 nonzero elements equal to 1.  We observe an undesampled noisy version of $\beta_o$, i.e.,  $y=X \beta_o+w$ where $X$ is a $1000\times 2000$ matrix with iid entries having normal distribution $N(0,1)$, and $w$ is a $1000\times 1$ noise vector with iid entries having the Gaussian distribution $N(0,0.7)$. We consider 100 equi-spaced values of $\lambda$ between 0 and 0.25. We then solve the LASSO and measure the $\|\hat{\beta}_{\lambda}\|_0$. Figure \ref{fig:LassoPathRandom} represents the $\|\hat{\beta}_{\lambda}\|_0$ versus $\lambda$.

\subsubsection{Figures  \ref{fig:MSE22}}\label{subsec:fig:2RiskFunctions}
The setting here is the same as for Figure \ref{fig:LassoPathRandom} with two differences. First, the $\lambda$ values are 100 equi-spaced values between 0 and 1. Second, the entries of $w$ are obtained from the Gaussian distributions $N(0,2)$ and $N(0,0.4)$.

\section{Conclusions}\label{sec:con}
In this paper, we have characterized the behavior of LASSO's solution and AMP's estimates as a function of the regularization parameter $\lambda$ and threshold parameters $\tau^1, \tau^2, \ldots$ in the asymptotic setting  (the ambient dimension $p \rightarrow \infty$, while the ratio of the number of measurements, $n$, to the ambient dimension is fixed to $\delta$). Our results have demonstrated that in the asymptotic setting this behavior is consistent with our intuition. For instance in the case of LASSO we have : (i) The size of the active set $\|\hat{\beta}^\lambda\|_0/p$ is a decreasing function of $\lambda$. (ii) The mean square error $\|\hat{\beta}_{\lambda} - \beta_o\|_2^2/p$ is a quasi-convex function of $\lambda$.
We also explored one of the practical implications of such results by studying the problem of tuning the regularization or  threshold parameters of LASSO and AMP. We derived a computationally efficient and asymptotically consistent parameter tuning scheme for the AMP algorithm. We have also been able to show that the corresponding AMP converges to a solution that has the same MSE as the solution of LASSO with the optimal value of $\lambda$. 

\bibliographystyle{unsrt}
\bibliography{Annals_main}
\end{document}